\documentclass[onefignum,onetabnum]{siamart171218}

\usepackage{lipsum}
\usepackage{amsfonts}
\usepackage{graphicx}
\usepackage{hyperref}
\usepackage{epstopdf}
\usepackage{algorithm}
\usepackage{algorithmic}
\usepackage{enumerate}
\usepackage{titlesec}
\usepackage{subcaption}

\usepackage{amsmath}
\usepackage{amssymb}
\usepackage{algorithm}
\usepackage{algorithmic}
  \usepackage{hyperref}  

\usepackage{dsfont}
\ifpdf
  \DeclareGraphicsExtensions{.eps,.pdf,.png,.jpg}
\else
  \DeclareGraphicsExtensions{.eps}
\fi

\newsiamremark{remark}{Remark}
\crefname{hypothesis}{Hypothesis}{Hypotheses}
\newsiamthm{claim}{Claim}
\newsiamthm{definirion}{Definition}
\newsiamthm{hypothesis}{Assumption}
\usepackage{comment}
\usepackage{tikz}
\usetikzlibrary{arrows.meta, positioning}
\usepackage{caption}

\usepackage{geometry}
\headers{Single loop method for quadratic minmax optimization}{Stefano Cipolla, Oliver Stein, and Alain Zemkoho}

\title{{A single loop method for quadratic minmax optimization}\thanks{Submitted to the editors DATE.
\funding{The work of the third author is partly funded by an Alexander von Humboldt Research Fellowship for  Experienced Researchers within the Continuous Optimization Chair at the Institute for Operations Research, Karlsruhe Institute of Technology (KIT).}}} 

\author{Stefano Cipolla \thanks{School of Mathematical Sciences, University of Southampton 
  (\email{s.cipolla@soton.ac.uk}).}
\and Oliver Stein \thanks{Institute for Operations Research (IOR), Karlsruhe Institute of Technology (KIT) 
  (\email{stein@kit.edu}).}
\and Alain Zemkoho \thanks{School of Mathematical Sciences, University of Southampton 
  (\email{a.b.zemkoho@soton.ac.uk}).}}

\usepackage{amsopn}
\DeclareMathOperator{\diag}{diag}

\makeatletter
\newcommand*{\addFileDependency}[1]{
  \typeout{(#1)}
  \@addtofilelist{#1}
  \IfFileExists{#1}{}{\typeout{No file #1.}}
}
\makeatother

\newcommand{\R}{{\mathbb{R}}}
\newcommand{\1}{\mathds{1}}

\DeclareMathOperator{\kernel}{ker}
\DeclareMathOperator{\inertia}{In}

\newsiamthm{example}{Example}

\AtBeginDocument{\colorlet{default}{.}} 
\usepackage{appendix}

\ifpdf
\hypersetup{
   pdftitle={A single loop method for quadratic minmax optimization with coupled inner constraints},
  pdfauthor={Stefano Cipolla, Oliver Stein, and Alain Zemkoho}
}
\fi

\usepackage{xr}

\begin{document}

\maketitle
%
\begin{abstract}
We consider a quadratic minmax problem with coupled inner constraints and propose a method to compute a class of stationary points. To motivate the need to compute such stationary points, we first show that they are meaningful, in the sense that they can be locally optimal for our problem under suitable{non-degeneracy} conditions. Then based on a suitable  log barrier function, we build an infeasible interior point-type \textit{single loop method} (which does not explicitly distinguish between the outer and inner problem) and prove that a non-degenerate stationary point is an attraction point as the algorithm moves along the designed central path. We show in particular that our method is polynomial in the special case where the inner feasible set of our constrained minmax problem is independent from outer variables. Our numerical experiments, on both synthetic data and a class of min cost flow problems, showcase the behavior of our method and how it outperforms existing algorithms from the literature in terms of the quality of the computed stationary points.
\end{abstract}




\section{Introduction.}\label{sec:intro}
We consider the quadratic minmax optimization problem
\begin{equation}\label{eq:MinMax}\tag{$P$}
	\min_{x \in X} \max_{y \in Y(x)} f(x,y),
\end{equation}
where the{convex-concave} objective function $f :\mathbb{R}^n\times \mathbb{R}^m \rightarrow \mathbb{R}$ is given by
\begin{equation}\label{eq:f(x,y)}
    f(x,y):=\frac{1}{2}\begin{bmatrix}
	x^\top & y^\top
\end{bmatrix} \underbrace{ \left[\begin{array}{cc}
   Q_{11}  &  Q_{12}\\
   Q_{12}^\top  & -Q_{22}
\end{array} \right]}_{=:Q}\begin{bmatrix}
	x \\ y
\end{bmatrix}  +  \begin{bmatrix} c_x^\top & c_y^\top
\end{bmatrix} \begin{bmatrix}
	x \\ y
\end{bmatrix}
\end{equation}
with the data vectors $c_x\in \mathbb{R}^n$, $c_y\in \mathbb{R}^m$, a matrix $Q_{12}\in\R^{n\times m}$ as well as symmetric positive semi-definite matrices $Q_{11}\in \mathbb{R}^{n\times n}$ and $Q_{22}\in \mathbb{R}^{m\times m}$. The \textit{min} and \textit{max} operators represent the \textit{outer} and \textit{inner} optimization problems, respectively. Analogously, the \textit{outer} and \textit{inner} feasible sets of problem \eqref{eq:MinMax} are respectively defined by
\begin{equation}\label{eq:XandY}
\begin{array}{rll}
X&:=&\left\{x \in \mathbb{R}^n:\;\, A_Ox = b_O, \;\; x \geq 0\right\},\\[1ex]
Y(x)&:=&\{ y \in \mathbb{R}^m:\;\,  A_Ix+B_Iy = b_I, \;\; y \geq 0 \},
\end{array}
\end{equation}
  {where inequalities between vectors are meant in the component-wise sense.} The corresponding data matrices  are such that 
$A_O\in \mathbb{R}^{p\times n}$, $b_O\in \mathbb{R}^p$,  $A_I\in \mathbb{R}^{q\times n}$, $B_I\in \mathbb{R}^{q\times m}$, and $b_I\in \mathbb{R}^q$ with the dimensions $p< n$ and $q< m$. 

There has been a regained interest in minmax optimization in recent years, thanks to a wide range of applications discovered in artificial intelligence; see, e.g., \cite{rafique2022weakly,farnia2016minimax,jin2020local}. The specific model in \eqref{eq:MinMax} has many applications, including in support vector machines, adversarial learning, interdiction games, and robust quadratic programming problems \cite{liu2015projection,razaviyayn2020nonconvex,MR4654111}. 

Clearly, \eqref{eq:MinMax} is a special class of convex-concave minmax problems. However, it is still quite difficult to solve due to the presence of the outer variable in the inner feasible set; more precisely, $A_Ix+B_Iy = b_I$ corresponds to the coupled constraint. Hence, overall, the problem is a quadratic version of the minmax problem with coupled inner constraint. Problem \eqref{eq:MinMax} was shown in \cite{MR4654111} to be NP-hard. One of the key challenges in addressing the problem is that the famous von Neumann minmax theorem cannot hold in this case, as viewing \eqref{eq:MinMax} as a game problem, the order of play becomes crucial; the outer player has to play first by selecting $x\in X$, and then, the inner player can react by selecting $y\in Y(x)$. This is the reason why some papers (see, e.g., \cite{jin2020local,goktas2021convex}) refer to such problems where the minmax theorem cannot hold as Stackelberg games. Problem \eqref{eq:MinMax} is strictly speaking not a Stackelberg game problem. However, interpreting it as that of finding 
\begin{equation}\label{eq:Stackelberg-Interpret}
   x\in \underset{\quad\;\; x\in X}{\arg \min}\; f(x, y(x)) \;\,\mbox{ with }\;\;\; y(x)\in \underset{\quad y\in Y(x)}{\arg \max }\; f(x, y) 
\end{equation}
can be seen as a special class of Stackelberg (or \textit{bilevel optimization}) problem, where the leader and follower have the same objective function, but which are optimized in opposite directions; i.e.,  w.r.t. $x$ (resp. $y$) for the upper-level (resp. lower-level) player. 
We do not use the specific concept of solution \eqref{eq:Stackelberg-Interpret} in this paper. To introduce the notion of optimal solution used here, consider the optimal value function 
\begin{equation}\label{VarPhi}
    \varphi(x):= \max_{y \in Y(x)} f(x,y) 
\end{equation}
associated to the inner  optimization problem (parameterized by $x\in X$)   
\begin{equation}\label{eq:InnerProblem}\tag{\text{$P_I(x)$}}
	\max_{y}\ f(x,y) \ \text{ s.t. }\ A_Ix+B_Iy-b_I=0,\ y\geq0.
\end{equation}
To focus our attention on the main ideas, we assume throughout that the corresponding inner optimal solution set is non-empty for all outer variables; i.e., 
\begin{equation}\label{ass:NonEmptyness}
   S(x):=\underset{y\in Y(x)}{\arg{\max}}\; f(x, y) \neq \emptyset \;\,\mbox{ for all } \;\, x\in X.
\end{equation} 
In view of \eqref{ass:NonEmptyness}, $\varphi$ is a real-valued function on $X$.
Based on this interpretation, the outer problem in  \eqref{eq:MinMax} can be equivalently rewritten as
\begin{equation}\label{eq:OuterProblem}\tag{\text{$P_O$}}
	\min_{x}\ \varphi(x) \ \text{ s.t. }\ A_Ox-b_O=0,\ x\geq0.
\end{equation}
Hence, we use the following concept of optimal solution throughout this paper. 
\begin{definition}\label{def:localOptimalSolution}
A point $\bar x$ will be said to be a local optimal solution of problem \eqref{eq:MinMax}  if there exists a neighborhood $U$ of the point such that 
\begin{equation}\label{def:localOptSol}
\varphi(\bar x) \leq \varphi(x) \;\,\mbox{ for all }\;\, x\in U\cap X.
\end{equation}
\end{definition}
It goes without saying that if $U$ is the whole space $\mathbb{R}^n$ in \eqref{def:localOptSol}, then the point $\bar x$ will be said to be a global optimal solution of problem \eqref{eq:MinMax}.  Other concepts of optimal solutions for minmax problems, where the order of play is important, can be found in \cite{jin2020local,dai2024rate}.

The following concept of stationarity, which will be at the center of the study in this paper, is tightly linked to the notion of local optimal solution in{Definition \ref{def:localOptimalSolution}}.  In the next definition, the symbol ``$\circ$'' represents the Hadamard product for two vectors; i.e.,  
$a\circ b:= [a_1b_1,\ldots, a_nb_n]^\top$ for $a,b\in\R^n$, which we use for reasons to become apparent later.
\begin{definition}\label{Def:StationarityConcept}
A point $\bar x\in X$ is called stationary for problem \eqref{eq:MinMax} if for some vector $\bar y\in S(\bar x)$,  there exist outer and inner multipliers $(\bar \lambda_O,\, \bar s_O)\in\R^p\times \R^n$ and  $(\bar\lambda_I, \, \bar s_I)\in\R^q\times \R^m$,  respectively, such that the vector $(\bar x,\bar y,\bar \lambda_O,\bar s_O,\bar\lambda_I,\bar s_I)$ solves 
the joint outer and inner KKT systems
       \begin{align}\label{eq:KKTjoint}
        \begin{split}
        \left.\begin{array}{rr}
            Q_{11}x+Q_{12}y+c_x+A_I^\top\lambda_I+A_O^\top\lambda_O-s_O&=0,\\
            A_Ox-b_O &=0,\\
            x \circ s_O &=0,\\
            x, s_O&\geq0,
            \end{array}\right\}\mbox{outer KKT system}\\
        \left.\begin{array}{rr}
            -Q_{22}y+Q_{12}^\top x+c_y+B_I^\top\lambda_I+s_I&=0,\\
        A_Ix+B_Iy-b_I&=0,\\
            y \circ s_I &=0,\\
            y, s_I&\geq0.
        \end{array}\right\}\mbox{inner KKT system}
        \end{split}
        \end{align}
\end{definition}
The inner part of \eqref{eq:KKTjoint} corresponds to the KKT conditions of the inner problem \eqref{eq:InnerProblem} for given $x\in X$. Indeed, in view of our concavity assumption on the objective function of \eqref{eq:InnerProblem}, and the polyhedrality of its feasible set, the condition $y\in S(x)$ is equivalently characterized by $y$ being a KKT point of \eqref{eq:InnerProblem}, i.e., satisfying the inner KKT system   with corresponding multipliers $\lambda_I$, $s_I$. In the following we shall abbreviate the set of points $(y,\lambda_I,s_I)$ solving the inner KKT system for given $x\in X$ by $
KKT_I(x)$.  Secondly, apart from the pair $(x, y)$, the outer and inner KKT systems are only connected by the Lagrange multiplier $\lambda_I$ associated to the coupled constraint in \eqref{eq:MinMax}. If there are no coupling constraint (i.e., $A_I=0$) and no coupling variables (i.e., $Q_{12}=0$), both KKT systems are completely independent from each other. 

  {We shall frequently return to the following running example to illustrate our results.}

  {
\begin{example}\label{ex:running}
The problem data
\begin{align*}
    Q_{11}=\begin{bmatrix}
        1 & 0\\ 0 & 0
    \end{bmatrix},\quad
    Q_{22}=\begin{bmatrix}
        4/3 & 0 & 0\\ 0 & 0 & 0\\ 0 & 0 & 0
    \end{bmatrix},\quad
    Q_{12}=\begin{bmatrix}
        0 & 0 & 0\\ 0 & 0 & 0
    \end{bmatrix},\quad
    c_x=\begin{bmatrix}
        -\delta\\0
    \end{bmatrix},\quad
    c_y=\begin{bmatrix}
        8/3\\0\\0
    \end{bmatrix},
\end{align*}
\begin{align*}
    A_O=\begin{bmatrix}
        1 & 1
    \end{bmatrix},\quad
    b_O=4,\quad
    A_I=\begin{bmatrix}
        1 & 0\\ -1 & 0
    \end{bmatrix},\quad
    B_I=\begin{bmatrix}
        -1 & 1 & 0\\ 1 & 0 & 1
    \end{bmatrix},\quad
    b_I=\begin{bmatrix}
        0\\1
    \end{bmatrix}
\end{align*}
with parameter $\delta> 4/3$ yield
\begin{align*}
    X&=\{x\in\mathbb{R}^2\mid x_1+x_2=4,\ x\geq0\},\\
    Y(x)&=\{y\in\mathbb{R}^3\mid x_1-y_1+y_2=0,\ -x_1+y_1+y_3=1,\ y\geq0\}\ \text{ for }\ x\in X,\\
    f(x,y)&=\frac12x_1^2-\delta x_1-\frac23 y_1^2+\frac83 y_1.
\end{align*}
The point $\bar x=[4,0|^\top$ is stationary for \eqref{eq:MinMax}, given that $(\bar x,\bar y,\bar \lambda_O,\bar s_O,\bar\lambda_I,\bar s_I)$, with the following values, solves the KKT system \eqref{eq:KKTjoint}:
\begin{align*}
    \bar x=\begin{bmatrix}
        4\\0
    \end{bmatrix},\quad\bar y=\begin{bmatrix}
        4\\0\\1
    \end{bmatrix},\quad\bar\lambda_O=\delta-4/3,\quad\bar s_O=\begin{bmatrix}
        0\\\delta-4/3
    \end{bmatrix},\quad\bar\lambda_I=\begin{bmatrix}
        -8/3\\0
    \end{bmatrix},\quad\bar s_I=\begin{bmatrix}
        0\\8/3\\0
    \end{bmatrix}.
\end{align*}
\end{example}
}

The joint KKT system \eqref{eq:KKTjoint} is not new, as it has been derived in various papers, including \cite{zemkoho2014simple}, where some relevant references can be found. More recently, these conditions have been obtained in \cite{dai2024optimality} and have also been used for the development of algorithms as it will become clear in Subsection~\ref{sec:related}. 
Before proceeding, it is important to note that the primary objective of this paper is to develop a tractable and robust algorithm to compute stationary points in the sense of Definition \ref{Def:StationarityConcept}. On the other hand, before discussing any algorithmic approach, we first address the following key questions, which do not seem to have been rigorously addressed for the quadratic minmax program in \eqref{eq:MinMax}: 
(i) How can we show that a local optimal solution of problem \eqref{eq:MinMax}, in the sense of  \eqref{def:localOptSol}, is a stationary point?
(ii) Conversely, under what conditions can a stationary point be locally optimal for problem \eqref{eq:MinMax}, in the sense of  \eqref{def:localOptSol}?
(iii) How is our concept of stationarity (in Definition \ref{Def:StationarityConcept}) connected to a notion of saddle point associated to the minmax problem \eqref{eq:MinMax}?
We address these questions in this paper, and subsequently provide a framework to ensure that a stationary point is \textit{non-degenerate} (definition to be provided). 
We subsequently establish that the non-degeneracy of a stationary point in the sense of Definition \ref{Def:StationarityConcept} ensures that such a point serves as an attraction point for our interior point method as it traverses the central path. 

\subsection{Related work.}\label{sec:related} Problem \eqref{eq:MinMax} is a special case of the problem 
\begin{equation}\label{eq:MinMax-General}
	\min_{x \in \mathcal{X}} \max_{y \in \mathcal{Y}(x)} \mathfrak{f}(x,y),
\end{equation}
where $\mathfrak{f}$ is a general real-valued function and $\mathcal{X}$ and $\mathcal{Y}(x)$ (for $x\in X$) are general subsets of $\mathbb{R}^n$ and $\mathbb{R}^m$, respectively.  Various special cases of this problem have been considered in the literature in recent years, including \cite{zhang2024primal,Yu-HongDai2020296,dai2024rate,dai2024optimality,MR4654111,goktas2021convex,hu2024minimizationapproachminimaxoptimization,lu2024first,lu2023first,vzakovic2000interior,zhang2025iterativeminimaxgamescoupled}, which represent, to the best of our knowledge, the main contributions in the existing literature where the inner feasible set effectively depends on the outer variable. For instance, in \cite{dai2024optimality,zhang2025iterativeminimaxgamescoupled}, $\mathcal{Y}(x)$ is defined by an equality constraint of the form $A_Ix+B_Iy = b_I$, while it is an inequality $A_Ix+B_Iy \leq b_I$ in \cite{zhang2024primal}. The constraint is a general inequality  $g(x, y)\leq 0$ in \cite{goktas2021convex}, and even further generalized in \cite{hu2024minimizationapproachminimaxoptimization}, where we have $g(x, y)\in \mathcal{K}$, where $\mathcal{K}$ is closed convex cone in $\mathbb{R}^m$. 

Apart from \cite{goktas2021convex,vzakovic2000interior}, the main idea in all these papers consists of building a surrogate unconstrained minmax problem for \eqref{eq:MinMax-General} using a Lagrangian \cite{zhang2024primal,Yu-HongDai2020296,dai2024rate,dai2024optimality,zhang2025iterativeminimaxgamescoupled} or some augmented Lagrangian-type function \cite{hu2024minimizationapproachminimaxoptimization,lu2024first,lu2023first}. The resulting problem is then solved by some form of gradient descent ascent--type algorithm, building on techniques to solve unconstrained minmax problems. The article \cite{MR4654111} is rather mainly focused on using the Lagrangian function to build duality results for a class of problem \eqref{eq:MinMax-General}. As for \cite{hu2024minimizationapproachminimaxoptimization}, it is dedicated to the construction of a special parametric minimization problem using the concept of partial forward-backward envelope. The main effort there is to show that the new problem and the corresponding version of  \eqref{eq:MinMax-General} can possess the same \textit{first-order minmax point}. It turns out that this concept, which is the main focus of the analysis there, is simply a compact form of our stationarity notion in Definition \ref{def:localOptimalSolution}. 

For the papers \cite{zhang2024primal,dai2024rate,dai2024optimality,zhang2025iterativeminimaxgamescoupled,lu2024first,lu2023first}, which focus on the development of numerical methods, the main effort is concentrated on estimating the iteration complexity for the corresponding first order method. All these algorithms also aim to compute the corresponding approximate versions of the stationarity system \eqref{eq:KKTjoint}. 
However, for the articles \cite{goktas2021convex,vzakovic2000interior}, they use completely different paths and solution concepts to develop algorithms for special versions of problem \eqref{eq:MinMax-General}. Similarly to the previous papers,  \cite{goktas2021convex} uses a first order method, but instead to compute a type of solution point $(\bar x, \bar y)$ called $(\epsilon, \delta)$--Stackelberg equilibrium, which is defined by the fulfillment of the conditions
\[
\bar y\in \mathcal{Y}(\bar x) \;\mbox{ and }\; \psi(\bar x) -\delta \leq \mathfrak{f}(\bar x, \bar y) \leq \min_{x\in \mathcal{X}} \psi(x) + \epsilon,
\]
 where $\psi$ is the optimal value function of associated to the corresponding inner problem of \eqref{eq:MinMax-General}. Clearly, such a point appears to be some form of approximation of the solution concept in Definition \ref{def:localOptimalSolution}. As for \cite{vzakovic2000interior}, the algorithm developed there, is, in some sense, the closest to the one proposed and studied in this paper. This is mainly for two reasons: first, the method in \cite{vzakovic2000interior} is a second order-type algorithm, similarly to our approach. Secondly, it builds from some concept of central path and interior point-type idea. However, the notion of solution there as well as the overall algorithmic process is completely different from what we do here. 

\subsection{Main contributions.}
The main contributions of this paper can be summarized as addressing the following two questions:
\begin{itemize}
    \item[(Q1)] How is the stationary concept in Definition \ref{Def:StationarityConcept} related to the notion of local optimal solution in Definition \ref{def:localOptimalSolution}?
    \item[(Q2)] How to tractably compute the stationary points in the sense of Definition \ref{Def:StationarityConcept}?
\end{itemize}
The motivation for (Q1) is that computing stationary points is important, but it is useful to know whether such points are meaningful with regards to their potential of being locally optimal for problem \eqref{eq:OuterProblem}. In particular, we show that if a stationary point is such that a{non-degeneracy condition} 
tailored to \eqref{eq:OuterProblem} holds, then this point is a local optimal solution for the problem; cf. Section \ref{sec:Stationary and saddle points}.  

With regards to (Q2), we introduce log barrier functions tailored to problem \eqref{eq:MinMax} that we use to build a central path. We use such a central path to design an infeasible interior point method converging, under a suitable framework provided in Sections \ref{sec:Stationary and saddle points} and \ref{sec:ipm}, to a stationary point in the sense Definition \ref{Def:StationarityConcept}. A key feature of our algorithm is that, unlike in descent-ascent methods, where the outer and inner iterates are computed alternately, we do not distinguish them, and therefore calculate our iterates in a \textit{single loop}. In particular, we show that when the inner feasible set is decoupled (i.e., when $A_I =0$), our method is polynomial; more precisely, we establish the complexity of our algorithm with iteration count $O((n+m)^2|\log(\varepsilon)|)$ to achieve an $\varepsilon$-approximation for a solution of the stationarity conditions in  Definition \ref{Def:StationarityConcept}. Moreover, we apply our method to both synthetic data and a class of min cost flow problems, and we demonstrate favorable performance when compared to related methods from the literature; cf. Section \ref{sec:Numerical experiments}.

\subsection{Organization of the paper.}
In Section~\ref{sec:Stationary and saddle points}, we establish a link between the stationarity concept, the local optimality, and a usual notion of saddle points tailored to problem \eqref{eq:MinMax}.{For the latter analysis,} a nondegeneracy concept is introduced and studied, which also prepares the ground for the convergence analysis of our algorithm, which is introduced in the subsequent section. In particular, in Section \ref{sec:ipm}, we develop an infeasible interior point method that we show to be polynomial in the decoupled case, where $Y$ is independent from $x$. After conducting our numerical experiments in Section \ref{sec:Numerical experiments}, some concluding remarks are provided in Section \ref{sec:conclusions}.

\section{Stationary and saddle points.}\label{sec:Stationary and saddle points}

The first step in this section is to formally link the stationarity concept in Definition \ref{Def:StationarityConcept} and the notion of local optimality for problem \eqref{eq:MinMax} in Definition \ref{def:localOptimalSolution}; cf. Subsection \ref{sect:Derivation of stationarity conditions}. 
In Subsection \ref{sec:meaningfulness}, we construct a sufficient condition for a stationary point to be a \textit{strict} local optimal solution for \eqref{eq:MinMax} in the sense of \eqref{def:localOptSol}. Subsequently,{Subsection~\ref{sec:ndstat} discusses an appropriate non-degeneracy concept for \eqref{eq:MinMax} and, based on this,}  Subsection \ref{Stationarity and local saddle points} presents a relationship between a stationary point and a local saddle point concept for \eqref{eq:MinMax}.


\subsection{An auxiliary quadratic problem and stationarity derivation.}\label{sect:Derivation of stationarity conditions}
In order to be able to derive the stationarity in Definition \ref{Def:StationarityConcept} (as necessary optimality condition of  \eqref{eq:MinMax}) for \textit{free}, as one would expect for any quadratic optimization problem (with affine linear constraints), we exploit the connection between our problem and the following single--level minimization problem:
\begin{equation}\label{eq:QPmodel}\tag{$QP$}
\begin{split}
        \min_{x,y,\lambda_I,s_I}\,&\frac{1}{2}y^\top Q_{22}y+  \lambda_I^\top (A_Ix-b_I)  +   \frac{1}{2}x^\top Q_{11}x+c_x^\top x\\
        \text{s.t. }\ & A_Ox-b_O=0,\ x\geq0,\\
        & -Q_{22}y + c_y +Q_{12}^\top x + B_I^\top\lambda_I +s_I =0,\ s_I \geq0.
\end{split}
\end{equation}

The global and local relationships between this problem and problem \eqref{eq:OuterProblem} rely on a Wolfe duality result for the inner problem \eqref{eq:InnerProblem}. To establish it, we introduce the following notation for each $x\in X$:
\[
\begin{array}{rll}
  q(x,y,\lambda_I,s_I) & := &  \frac{1}{2}y^\top Q_{22}y+  \lambda_I^\top (A_Ix-b_I)  +   \frac{1}{2}x^\top Q_{11}x+c_x^\top x,\\[2ex]
   F(x) &:=&\left\{(y,\lambda_I,s_I)\in \mathbb{R}^m\times \mathbb{R}^q\times \mathbb{R}^m:\right.\\[2ex]
        &  &        \qquad\qquad \left.-Q_{22}y + c_y +Q_{12}^\top x + B_I^\top\lambda_I +s_I =0,\ s_I \geq0\right\}, \\[2ex]   
        S_D(x) & := & \underset{(y,\lambda_I,s_I)\in F(x)}{\arg\min} q(x,y,\lambda_I,s_I),\\[2ex]
   \varphi_D(x) & := & \underset{(y,\lambda_I,s_I)\in F(x)}{\min} q(x,y,\lambda_I,s_I).
  \end{array}
\] 
\begin{lemma}\label{lem:Wolfe}
    The following assertions hold for each $x\in X$:
    \begin{itemize}
    \item[(a)] $\emptyset\neq KKT_I(x)\subseteq S_D(x)$.
    \item[(b)] $\varphi(x)=\varphi_D(x)$.
    \end{itemize}
\end{lemma}
\begin{proof}
By \eqref{ass:NonEmptyness} and the polyhedrality of the inner feasible set $Y(x)$, for each $x\in X$, the set $KKT_I(x)$ is non-empty. Furthermore, observe that with the
Lagrangian 
\begin{align*}
L_I^x&(y,\lambda_I,s_I)= -\textstyle{\frac{1}{2}}y^\top Q_{22}y + c_y^\top y+x^\top Q_{12}y  + \lambda_I^\top(A_Ix+B_Iy -b_I) + s_I^\top y+ \textstyle{\frac{1}{2}}x^\top Q_{11}x + c_x^\top x\\ 
&= y^\top ( - Q_{22}y + c_y + Q_{12}^\top x  + B_I^\top \lambda_I +s_I)
  +\textstyle{\frac{1}{2}}y^\top Q_{22}y+  \lambda_I^\top(A_Ix-b_I)+ \textstyle{\frac{1}{2}}x^\top Q_{11}x + c_x^\top x
  \end{align*}
of \eqref{eq:InnerProblem}, the feasible set of its Wolfe dual problem is defined by
the stationarity condition 
\begin{align*}
0=\nabla_yL_I^x(y,\lambda_I,s_I) = -Q_{22}y +c_y +Q_{12}^\top x +B_{I}^\top\lambda_I +s_I
\end{align*}
and $s_I\geq0$, and thus its objective function reduces to 
\begin{align*}
  L_I^x(y,\lambda_I,s_I) =  \textstyle{\frac{1}{2}}y^\top Q_{22}y+  \lambda_I^\top(A_Ix-b_I)+ \frac{1}{2}x^\top Q_{11}x + c_x^\top x.
\end{align*}  
Therefore the Wolfe dual problem of \eqref{eq:InnerProblem}
can be written as 
\begin{equation}\label{eq:WolfeDual}\tag{\text{$D_I(x)$}}
\underset{y,\lambda_I,s_I}{\min} q(x,y,\lambda_I,s_I) \ \text{ s.t. }\ (y,\lambda_I,s_I)\in F(x),
\end{equation}
and weak Wolfe duality 
\cite{Wolfe1961} yields $\varphi_D(x)\geq\varphi(x)$.
Moreover, for any $(y,\lambda_I,s_I)\in KKT_I(x)$ one has $(y,\lambda_I,s_I)\in F(x)$ and $y\in Y(x)$, which yields
\begin{align*}
\varphi_D(x)\leq q(x,y,\lambda_I,s_I)=L_I^x(y,\lambda_I,s_I)=f(x,y)\leq\varphi(x)\leq\varphi_D(x).
\end{align*}
This shows the strong duality result (b) as well as the second assertion of (a).
\end{proof}
We remark that equality of $KKT_I(x)$ and $S_D(x)$ in 
Lemma~\ref{lem:Wolfe}(a) cannot be expected, since the primal feasibility condition $y\in Y(x)$ is not part of the definition of $F(x)$. Nevertheless, Lemma~\ref{lem:Wolfe}(a) shows that, like the optimal solution mapping $S:X\rightrightarrows\R^m$ of \eqref{eq:InnerProblem}, also the optimal solution mapping $S_D:X\rightrightarrows\mathbb{R}^m\times \mathbb{R}^q\times \mathbb{R}^m$ of \eqref{eq:WolfeDual}
has non-empty images for all $x\in X$, so that with $\varphi$ also $\varphi_D$ is real-valued on $X$.

Hence, problem \eqref{eq:OuterProblem} can be equivalently written as  
\begin{equation}\label{eq:min-D}\tag{\text{$P_D$}}
    \min_x\varphi_D(x) \ \text{ s.t. } x\in X.
\end{equation}
Thanks to this dual representation of problem \eqref{eq:OuterProblem}, we can now establish the following  global equivalence between it and problem  \eqref{eq:QPmodel}.
\begin{proposition}\label{thm:GlobalRelationshipTo(QP)}The following statements hold true:
\begin{itemize}
    \item[(a)] Let $\bar x$ be a global optimal solution of problem \eqref{eq:OuterProblem}. Then, for all $(\bar y, \bar\lambda_I, \bar s_I)\in S_D(\bar x)$, the quadruple $(\bar x, \bar y, \bar\lambda_I, \bar s_I)$ is a global optimal solution of problem \eqref{eq:QPmodel}. 
    \item[(b)] Let $(\bar x, \bar y, \bar\lambda_I, \bar s_I)$ be a global optimal solution of  problem \eqref{eq:QPmodel}. Then the point $\bar x$ is a global optimal solution of problem \eqref{eq:OuterProblem}. 
\end{itemize}
\end{proposition}
\begin{proof}
For (a), note that $\bar x$ is also a global optimal solution of problem \eqref{eq:min-D}. Hence, let $(\bar y, \bar\lambda_I, \bar s_I) \in S_D(\bar x)$, then it follows that for any $x\in X$ and $(y, \lambda_I, s_I)\in F(x)$, we have 
\[
q(\bar x, \bar y, \bar\lambda_I, \bar s_I)=\varphi_D(\bar x) \leq \varphi_D(x)\leq q(x, y, \lambda_I, s_I). 
\]
Hence, $(\bar x, \bar y, \bar \lambda_I, \bar s_I)$ is a global optimal solution of problem \eqref{eq:QPmodel}. 

As for (b), considering the hypothesis, it automatically holds that $(\bar y, \bar{\lambda}_I, \bar{s}_I)\in S_D(\bar x)$.  
Otherwise, we can find $(\tilde{y}, \tilde{\lambda}_I, \tilde{s}_I)\in F(\bar x)$ such that 
\[
q(\bar x, \bar y, \bar{\lambda}_I, \bar{s}_I) > q(\bar x, \tilde{y}, \tilde{\lambda}_I, \tilde{s}_I),
\]
and since $(\bar x, \tilde{y}, \tilde{\lambda}_I, \tilde{s}_I)$ is a feasible point to problem \eqref{eq:QPmodel}, we have a contradiction.
Therefore, for any $x\in X$ and $(y, \lambda_I, s_I)\in F(x)$,
\[
\varphi_D(\bar x)=q(\bar x, \bar y, \bar\lambda_I, \bar s_I) \leq q(x, y, \lambda_I, s_I).
\]
  {By Lemma~\ref{lem:Wolfe}}, we have $S_D(x) \neq \emptyset$, and hence, it follows in particular that 
\[
\varphi_D(\bar x)=q(\bar x, \bar y, \bar\lambda_I, \bar s_I) \leq \underset{(y, \lambda_I, s_I)\in F(x)}{\min} q(x, y, \lambda_I, s_I) =\varphi_D(x).
\]
Therefore, $\bar x$ is a global optimal solution of problem \eqref{eq:min-D} and hence for problem \eqref{eq:OuterProblem}.
\end{proof}

For{one direction of} the local relationship, we need the inner semicontinuity of the set-valued mapping $S_D$, which will be said to hold at a point   $(\bar x, \bar y, \bar\lambda_I, \bar s_I)$ if for any sequence $(x^k)$ converging to $\bar x$, there is a sequence $(y^k_I,  \lambda^k_I, s^k_I)\in S_D(x^k)$ that converges to $(\bar y, \bar\lambda_I, \bar s_I)$ as $k \rightarrow \infty$.{We will state a sufficient condition for inner semicontinuity of $S_D$ in Remark~\ref{rem:isc_suff_cond}.}
\begin{proposition}\label{thm:LocalRelationshipTo(QP)}The following statements hold true:
\begin{itemize}
    \item[(a)] Let $\bar x$ be a local optimal solution of \eqref{eq:OuterProblem}. Then, for all $(\bar y, \bar\lambda_I, \bar s_I)\in S_D(\bar x)$, the quadruple $(\bar x, \bar y, \bar\lambda_I, \bar s_I)$ is a local optimal solution of problem \eqref{eq:QPmodel}. 
    \item[(b)] Let the point  $(\bar x, \bar y, \bar\lambda_I, \bar s_I)$, where $S_D$ is inner semicontinuous, be a local optimal solution of  \eqref{eq:QPmodel}. Then $\bar x$ is a local optimal solution of problem \eqref{eq:OuterProblem}. 
\end{itemize}
\end{proposition}
\begin{proof}For (a), $\bar x$ also being a local optimal solution of problem \eqref{eq:min-D}, we claim that for any $(\bar y,\bar\lambda_I, \bar s_I)\in S_D(\bar x)$, the point $(\bar x, \bar y,\bar\lambda_I, \bar s_I)$ is automatically a local optimal solution of problem \eqref{eq:QPmodel}. 
Otherwise, consider a point $(\bar y,\bar\lambda_I, \bar s_I)\in S_D(\bar x)$ such that  $(\bar x, \bar y,\bar\lambda_I, \bar s_I)$ is not a local optimal solution of problem \eqref{eq:QPmodel}. Then we can find a sequence $(x^k, y^k, \lambda^k_I, s^k_I)$ from the graph of $F$ on $X$ with $x^k \rightarrow \bar x$, $y^k \rightarrow \bar y$, $\lambda^k_I \rightarrow \bar\lambda_I$, and $s^k_I \rightarrow \bar s_I$ such that we have 
\[
q(x^k, y^k, \lambda^k_I, s^k_I) < q(\bar x, \bar y,\bar\lambda_I, \bar s_I) = \varphi_D(\bar x) \;\mbox{ for all }\, k.
\]
Then considering the definition of $\varphi_D$, it follows that 
\[
\varphi_D(x^k) \leq q(x^k, y^k, \lambda^k_I, s^k_I) < q(\bar x, \bar y,\bar\lambda_I, \bar s_I) = \varphi_D(\bar x) \;\mbox{ for all }\, k.
\]
This clearly contradicts the fact $\bar x$ is a local optimal solution of problem \eqref{eq:min-D} given that $x^k\in X$ for all $k$. 
Now, given that $S(\bar x) \neq \emptyset$, it follows that $S_D(\bar x)\neq \emptyset$ by the strong Wolfe duality result. Thus, there exists some  $(\bar y,\lambda_I,  s_I)\in S_D(\bar x)$ such that  $(\bar x, \bar y,\lambda_I, s_I)$ is locally optimal for  \eqref{eq:QPmodel}.

As for (b), if $\bar x$ is not a local optimal solution of problem \eqref{eq:min-D}, then we can find a feasible sequence $x^k\rightarrow \bar x$ such that $\varphi_D(\bar x) > \varphi_D(x^k)$ for all $k$. As the set-valued mapping $S_D$ is inner semicontinuous at $(\bar x, \bar y, \bar\lambda_I, \bar s_I)$, we can find a sequence $\left(y^k, \lambda^k_I, s^k_I\right)\in S_D(x^k)$ that converges to $(\bar y, \bar\lambda_I, \bar s_I)$. It follows by the construction that 
\[
q(\bar x, \bar y, \bar\lambda_I, \bar s_I) = \varphi_D(\bar x) > \varphi_D(x^k) = q\left(x^k, y^k, \lambda^k_I, s^k_I\right)
\]
with $x^k\in X$, $\left(y^k, \lambda^k_I, s^k_I\right)\in F(x^k)$ for all $k$,{which contradicts the assumed local minimality of the point $(\bar x, \bar y, \bar\lambda_I, \bar s_I)$ for problem \eqref{eq:QPmodel}}. 
\end{proof}
\begin{remark}\label{rem:AI=0}  In the special case $A_I =0$, when no coupling constraints in the inner problem \eqref{eq:InnerProblem} exist and the outer and inner problems are only coupled via the matrix $Q_{12}$, the optimal value function $\varphi$ \eqref{VarPhi} is convex as the pointwise maximum of convex functions.  Therefore, \eqref{eq:OuterProblem} and  \eqref{eq:min-D} are convex, possibly non\-smooth, optimization problems. Similarly, in this case, also \eqref{eq:QPmodel} is a convex problem. Therefore, for all these problems local minimal points are necessarily global minimal points. Hence, in the case $A_I =0$, by{Proposition} \ref{thm:GlobalRelationshipTo(QP)} we can state that these problems are equivalent without distinguishing global and local minimal points. 
In particular, in this case, the inner semicontinuity of $S_D$ required in{Proposition} \ref{thm:LocalRelationshipTo(QP)} is not necessary. 
\end{remark}

Finally, thanks to Proposition \ref{thm:LocalRelationshipTo(QP)}(a), we can derive the main result of this subsection. 
\begin{corollary}\label{theo:NecessaryOptCond-withoutCQ}
If $\bar x$ is locally optimal for problem \eqref{eq:OuterProblem}, then it is  stationary. 
\end{corollary}
\begin{proof} Since $S_D(\bar x)\neq \emptyset$, it follows from Proposition \ref{thm:LocalRelationshipTo(QP)}(a) that there exists $(\bar y, \bar\lambda_I, \bar s_I)\in S_D(\bar x)$ such that $(\bar x, \bar y, \bar\lambda_I, \bar s_I)$ is a locally optimal for \eqref{eq:QPmodel}.{Since the feasible set of the latter problem is defined only by affine linear constraints, with its Lagrangian
\begin{align*}
    \mathcal{L}&(x,y,\lambda_I,s_I,\lambda_O,s_O,\gamma,\tau):=\frac{1}{2}y^\top Q_{22}y+  \lambda_I^\top (A_Ix-b_I)  +   \frac{1}{2}x^\top Q_{11}x+c_x^\top x\\
    &\qquad \quad +\lambda_O^\top(A_Ox-b_O)-s_O^\top x+\gamma^\top(-Q_{22}y + c_y +Q_{12}^\top x + B_I^\top\lambda_I +s_I)-\tau^\top s_I
\end{align*}
we can find a multiplier vector $(\lambda_O, s_O, \gamma, \tau)$ such that the KKT conditions
\begin{align*}
    \nabla_{(x,y,\lambda_I,s_I)} \mathcal{L}(\bar x,\bar y, \lambda_I,  s_I, \lambda_O,  s_O,\gamma,\tau)&=0,\\
    A_O\bar x-b_O&=0, \\ 
    \bar x\circ s_O & =0,\\
    \bar x, \, s_O & \geq 0,\\
    -Q_{22}\bar y + c_y +Q_{12}^\top\bar x + B_I^\top\lambda_I +s_I&=0,\\ 
    \tau\circ s_I &=0,\\
    \tau, \, s_I&\geq0
\end{align*} 
are satisfied.
The equations corresponding to $\nabla_{y} \mathcal{L}$ and $\nabla_{s_I} \mathcal{L}$ in the first line of the system particularly yield $Q_{22}\bar y=Q_{22}\gamma$ and $\tau=\gamma$, respectively. 
Using these identities to write $\tau$ and $Q_{22}\bar y$ in terms of $\gamma$, and calculating $\nabla_{x} \mathcal{L}$ and $\nabla_{\lambda_I} \mathcal{L}$,
leads to an equivalent system
which coincides with \eqref{eq:KKTjoint}, where 
$\gamma$ solves the inner KKT system 
with multipliers $\lambda_I$ and $s_I$. This implies $\gamma\in S(\bar x)$ and concludes the proof.}
\end{proof}

\begin{remark}
    We point out that in the above proof of Corollary~\ref{theo:NecessaryOptCond-withoutCQ} the firstly chosen point $\bar y$ is not the one which necessarily leads to the stationarity condition, but the possibly different point $\gamma\in S(\bar x)$.{In particular, the primal feasibility conditions $A_I\bar x+B_Iy=b_I$, $y\geq0$ do not necessarily hold for $\bar y$, but they do hold for $\gamma$.} Also note that $\gamma$ is a dual variable corresponding to an equation from the Wolfe dual of the inner problem and, thus, a ``bi-dual variable''. Therefore it makes sense that it lies in the minimal point set $S(\bar x)$ of the primal inner problem.
\end{remark}

To conclude this subsection, note that the stationarity system in \eqref{eq:KKTjoint} can be derived by multiple other ways. 
For instance, if the optimal solution set-valued mapping $S$ reduces to a single-valued and continuously differentiable function, the stationarity conditions in Definition \ref{Def:StationarityConcept} could also be obtained from the Lagrange multiplier rule of the minimization of the function $f(x, y(x))$ subject to $x\in X$, under the non-degeneracy of the outer and inner points of interest $\bar x$ and $\bar y\in S(\bar x)$, respectively,{cf. Subsection~\ref{sec:meaningfulness}}.



\subsection{On the meaningfulness of stationarity.}\label{sec:meaningfulness}
As the algorithm that we introduce in this paper aims to compute stationary points in the sense of Definition \ref{Def:StationarityConcept}, an interesting question is to know whether such points are actually meaningful for problem \eqref{eq:MinMax}. The \textit{meaningfulness} here is in terms of whether such points are connected to a solution concept associated to problem \eqref{eq:MinMax}. Indeed, this is the question that we address in this subsection. More precisely, we construct here a framework that ensures that a stationary point, in the sense of Definition \ref{Def:StationarityConcept}, is a local optimal solution of problem \eqref{eq:MinMax}, in the sense of Definition \ref{def:localOptimalSolution}. 

  {This framework uses inner and outer non-degeneracy conditions.} Let
\[
{\cal I}_{\bar x} :=\{i=1,\ldots,n:\,\bar x_i=0\} \mbox{ and } {\cal I}_{\bar y}:=\{j=1,\ldots,m:\,\bar y_j=0\}
\]
 denote the active index sets of the inequality constraints $x\geq0$ at $\bar x$ and of $y\geq0$ at $\bar y$, respectively. 
Furthermore, $\mathds{1}_{\bar x}$ and $\mathds{1}_{\bar y}$ stand for the matrices with row vectors $(e^i_n)^\top$, $i\in {\cal I}_{\bar x}$, and $(e^k_m)^\top$, $j\in {\cal I}_{\bar y}$, respectively.  We say that the outer linear independence constraint qualification (OLICQ) holds at the point  $\bar x\in X$ if the matrix 
\begin{equation}\label{OLICQ}\tag{OLICQ}
    [A_O^\top,\1_{\bar x}^\top] \;\, \mbox{ possesses full column rank},
\end{equation}
and similarly, the inner linear independence constraint qualification (ILICQ) holds at $\bar y\in Y(\bar x)$ if the matrix
\begin{equation}\label{ILICQ}\tag{ILICQ}
    [B_I^\top,\1^\top_{\bar y}] \;\, \mbox{ possesses full column rank}.
\end{equation}
Given a stationary point $\bar x$ of \eqref{eq:MinMax} with some $\bar y\in S(\bar x)$, the \eqref{ILICQ} implies that the corresponding inner multipliers $\bar\lambda_I$, $\bar s_I$ are uniquely determined, and therefore, under the \eqref{ILICQ}, the \eqref{OLICQ} implies the same for the outer multipliers $\bar\lambda_O$, $\bar s_O$.

%
%

Under the \eqref{ILICQ}, inner strict complementary slackness (ISCS) is said to hold in the case 
\begin{equation}\label{ISCS}\tag{ISCS}
    \min_{i=1,\ldots,m}([\bar s_I]_i+[\bar y]_i)>0,
\end{equation}
and under the \eqref{ILICQ} and \eqref{OLICQ}, outer strict complementary slackness (OSCS) if
\begin{equation}\label{OSCS}\tag{OSCS}
    \min_{i=1,\ldots,n}([\bar s_O]_i+[\bar x]_i)>0.
\end{equation}
Finally, we state inner and outer second order sufficiency conditions. To this end, for $\bar x\in X$, let
\begin{align}\label{eq:defTI}
    T(\bar y,Y(\bar x)):=\kernel\begin{bmatrix}
        B_I\\\1_{\bar y}\end{bmatrix}
\end{align}
denote the tangent space to $Y(\bar x)$ at $\bar y$. Then the inner second order sufficiency condition (ISOSC) is said to hold if
\begin{equation}\label{ISOSC}\tag{ISOSC}
 \eta^\top (-Q_{22})\eta <0 \;\, \mbox{ for all } \eta\in T(\bar y,Y(\bar x))\setminus\{0 \}.
\end{equation}
Analogously, let 
\begin{align}\label{eq:defTO}
    T(\bar x,X)=\kernel\begin{bmatrix}
        A_O\\\1_{\bar x}\end{bmatrix}
\end{align}
denote the tangent space to $X$ at $\bar x$. Under \eqref{ILICQ}, \eqref{ISCS} and \eqref{ISOSC} the outer second order sufficiency condition (OSOSC) is said to hold if
\begin{equation}\label{OSOSC}\tag{OSOSC}
 d^\top Q_{11}'d >0 \;\, \mbox{ for all } d\in T(\bar x,X)\setminus\{0\},
\end{equation}
where
\begin{align}\label{eq:D2phi}
Q_{11}':=Q_{11}-\begin{bmatrix}
    Q_{12} & A_I^\top & 0
\end{bmatrix}\begin{bmatrix}
        -Q_{22} & B_I^\top &\1_{\bar y}^\top\\
        B_I & 0 & 0\\
       \1_{\bar y} & 0 & 0
   \end{bmatrix}^{-1}\begin{bmatrix}
    Q_{12}^\top\\A_I\\0
    \end{bmatrix}.
\end{align}
The inverted matrix in \eqref{eq:D2phi} is indeed non-singular, which can be shown with the techniques from \cite{fiacco1990nonlinear,jongen1986critical}. For completeness, and since these results also play a crucial role in Subsection~\ref{sec:ndstat}, we briefly recall the corresponding arguments. More details are given in Subsection~\ref{sec:ndstat}.

Let us abbreviate the system of equations in the inner KKT system in \eqref{eq:KKTjoint} by
\begin{align}\label{eq:defPhiI}
0=\Phi_I(x,y,\lambda_I,s_I):=\begin{bmatrix}
        -Q_{22}y+Q_{12}^\top x+c_y+B_I^\top\lambda_I+s_I\\
        A_Ix+B_Iy-b_I\\
        s_I\circ y
   \end{bmatrix}.
\end{align}
The Jacobian with respect to the inner variables and multipliers $(y,\lambda_I,s_I)$ of $\Phi_I$ at a stationary point $(\bar x,\bar y,\bar\lambda_I,\bar s_I)$ is
 \begin{align}\label{eq:DPhiI}   D_{(y,\lambda_I,s_I)}\Phi_I(\bar x,\bar y,\bar\lambda_I,\bar s_I)=\begin{bmatrix}
        -Q_{22} & B_I^\top & \1\\
        B_I & 0 & 0\\
        \diag(\bar s_I)  & 0 & \diag(\bar y)
    \end{bmatrix},
 \end{align}
 where, e.g., $\diag(\bar s_I)$ stands for the diagonal matrix with diagonal entries from $\bar s_I$ and results as the partial Jacobian of the Hadamard product $s_I\circ y$ with respect to $y$. Moreover, here and in the following $\1$ stands for the identity matrix of appropriate dimension.
$D_{(y,\lambda_I,s_I)}\Phi_I(\bar x,\bar y,\bar\lambda_I,\bar s_I)$ is non-singular iff \eqref{ISCS} holds and the matrix
\begin{align}\label{eq:M}
    \begin{bmatrix}
        -Q_{22} & B_I^\top &\1_{\bar y}^\top\\
        B_I & 0 & 0\\
       \1_{\bar y} & 0 & 0
    \end{bmatrix} 
\end{align}
is non-singular. The latter is the case under \eqref{ILICQ} and \eqref{ISOSC} (cf. Lemma~\ref{lem:inertia}).

Under \eqref{ILICQ}, \eqref{ISCS} and \eqref{ISOSC}, this result allows us to apply the Implicit Function Theorem to the system $\Phi_I(x,y,\lambda_I,s_I)=0$ at $(\bar x,\bar y,\bar\lambda_I,\bar s_I)$, which
yields the existence of a unique smooth function $(y(x),\lambda_I(x),s_I(x))$ with $(y(\bar x),\lambda_I(\bar x),s_I(\bar x))=(\bar y,\bar\lambda_I,\bar s_I)$ and $\Phi_I(x,y(x),\lambda_I(x),s_I(x))\equiv 0$ for all $x$ from some neighborhood $U$ of $\bar x$. The \eqref{ISCS} and continuity arguments also imply $y(x)\geq0$ and $s_I(x)\geq0$ for sufficiently small $U$, so that  $y(x)\in S(x)$ holds for all $x\in U$. This yields the local description $\varphi(x)=f(x,y(x))$ of the inner maximal value function for all $x\in U$. In particular, the function $\varphi$ is smooth around $\bar x$ and, therefore, local optimality conditions for \eqref{eq:OuterProblem} can be obtained from conditions for smooth problems (cf. also \cite{hettich2005semi,wetterling1970definitheitsbedingungen}). To formulate them, some computations based on the Implicit Function Theorem yield $\nabla\varphi(\bar x)=Q_{11}\bar x+Q_{12}\bar y+c_x+A_I^\top \bar\lambda_I$ and $D^2\varphi(\bar x)=Q_{11}'$ with the matrix from \eqref{eq:D2phi}.

\begin{theorem}\label{the:ndmeaningfulness}
Let $\bar x$ be a stationary point in the sense of Definition \ref{Def:StationarityConcept} such that the corresponding pair $(\bar x, \bar y)$ with $\bar y\in S(\bar x)$ satisfies  \eqref{ILICQ}, \eqref{ISCS}, \eqref{ISOSC}, \eqref{OLICQ}, \eqref{OSCS} and  \eqref{OSOSC}. 
 Then $\bar x$ is a strict local minimal point for problem \eqref{eq:OuterProblem}. 
\end{theorem}
\begin{proof}
As seen above, under \eqref{ILICQ}, \eqref{ISCS}, \eqref{ISOSC}, with the implicitly defined function $y(x)$ the problem \eqref{eq:OuterProblem} is locally equivalent to the minimization of $\varphi(x)=f(x,y(x))$ over $X$. In particular, the strict local optimal solutions of both problems coincide.

A standard second order sufficiency condition for $\bar x\in X$ to be a strict local optimal solution of $\varphi(x)=f(x,y(x))$ on $X$ consists of the KKT condition
\begin{align*}
\nabla\varphi(\bar x)+A_O^\top\bar\lambda_O-\bar s_O&=0,\\
A_O\bar x-b_O&=0,\\
\bar x\circ \bar s_O&=0,\\
\bar x,\bar s_O&\geq0
\end{align*}
with multipliers $\bar\lambda_O$, $\bar s_O$,
the \eqref{OLICQ}, \eqref{OSCS} as well as the second order condition
\begin{align*}
     d^\top D^2\varphi(\bar x) d >0 \;\, \mbox{ for all } d\in T(\bar x,X)\setminus\{0\}.
\end{align*}
In view of the above formulas for $\nabla\varphi(\bar x)$ and $D^2\varphi(\bar x)$ this means that the triple $(\bar x,\bar\lambda_O,\bar s_O)$ must satisfy the outer KKT system in \eqref{eq:KKTjoint}, \eqref{OLICQ}, \eqref{OSCS} and \eqref{OSOSC}. Since all these conditions are covered by the assumptions, $\bar x$ is a strict local optimal solution of $\varphi(x)=f(x,y(x))$ on $X$, which shows the assertion.
\end{proof}

\begin{remark}
The simultaneous satisfaction of the inner and outer non-degeneracy assumptions \eqref{ILICQ}, \eqref{ISCS}, \eqref{ISOSC}, \eqref{OLICQ}, \eqref{OSCS} and  \eqref{OSOSC} is a mild requirement in a topological sense. In fact, the epigraph reformulation \cite{stein2025tutorial} of \eqref{eq:MinMax} yields a generalized semi-infinite optimization problem (GSIP), and the corresponding non-degeneracy conditions are known to hold generically at each solution point of a  GSIP (cf. \cite{stein2002generalized} for the linear and \cite{gunzel2008generalized} for the nonlinear case).
\end{remark}

\begin{example}\label{ex:running2}
    In the situation of Example~\ref{ex:running} one obtains ${\cal I}_{\bar x}={\cal I}_{\bar y}=\{2\}$ and
    \begin{align*}
        [A_O^\top,\1_{\bar x}^\top]=\begin{pmatrix}
            1 & 0\\1 & 1
        \end{pmatrix},\quad [B_I^\top,\1^\top_{\bar y}]=\begin{pmatrix}
            -1 & 1 & 0\\1 & 0 & 1\\0 & 1 & 0
        \end{pmatrix},
    \end{align*}
    so that \eqref{OLICQ} and \eqref{ILICQ} hold at $\bar x$ and $\bar y\in Y(\bar x)$, respectively. Moreover, one obtains $T(\bar x,X)\setminus\{0\}=\emptyset$ and $T(\bar y,Y(\bar x))\setminus\{0\}=\emptyset$, so that \eqref{OSOSC} and \eqref{ISOSC} are trivially satisfied. Hence, Theorem~\ref{the:ndmeaningfulness} ensures that $\bar x$ is a strict local optimal solution of \eqref{eq:OuterProblem}.
\end{example}

\begin{remark}\label{rem:isc_suff_cond}
As seen above, at a point $(\bar x,\bar y,\bar\lambda_I,\bar s_I)$ with $(\bar y,\bar\lambda_I,\bar s_I)\in KKT_I(\bar x)$ the conditions \eqref{ILICQ}, \eqref{ISCS} and \eqref{ISOSC} imply the existence of a unique locally defined smooth function $(y(x),\lambda_I(x),s_I(x))$  which satisfies $(y(\bar x),\lambda_I(\bar x),s_I(\bar x))=(\bar y,\bar\lambda_I,\bar s_I)$ as well as $KKT_I(x)=\{(y(x),\lambda_I(x),s_I(x))\}$ for all $x$ from some neighborhood $U$ of $\bar x$.
Therefore, in view of $KKT_I(x)\subseteq S_D(x)$ (Lemma~\ref{lem:Wolfe}), the inner semicontinuity of $S_D$ required in Proposition~\ref{thm:LocalRelationshipTo(QP)}(b) holds at every such point 
$(\bar x,\bar y,\bar\lambda_I,\bar s_I)$.
\end{remark}

\begin{remark}
    The assertion of Theorem~\ref{the:ndmeaningfulness} remains valid in the absence of \eqref{OLICQ} and \eqref{OSCS}, if in \eqref{OSOSC} the tangent space $T(\bar x,X)$ is replaced by the outer critical cone
 \begin{align*}
\{d\in \mathbb{R}^n\mid \left(Q_{11}\bar x + Q_{12}\bar y + c_x + A^\top_I \bar\lambda_I \right)^\top d =  0,\ A_Od=0, \;\, d_i\geq 0, \;\, i\in \mathcal{I}_{\bar x} \}.
 \end{align*}
\end{remark}

 {To close this subsection, it might also be useful to note that a meaningfulness result (i.e., ensuring the same conclusion as in Theorem \ref{the:ndmeaningfulness} for a stationary point in the sense of Definition \ref{Def:StationarityConcept}) can be obtained without the need of the inner or outer strict complementarity condition; see Section \ref{Proof of Theorem Meaningfulness} for an alternative result to Theorem \ref{the:ndmeaningfulness} with the corresponding proof}. 

\subsection{Non-degenerate stationary points.}\label{sec:ndstat}
In Section~\ref{sec:ipm}, we will devise an interior point method (IPM) to approximate a stationary point of problem \eqref{eq:MinMax} in the sense of Definition~\ref{Def:StationarityConcept}. The convergence proof will require the non-singularity of the Jacobian of the system of equations appearing in \eqref{eq:KKTjoint}, which we abbreviate by
\begin{align*}
    0=\Phi(x,\lambda_O,s_O,y,\lambda_I,s_I):=\begin{bmatrix}
        Q_{11}x+Q_{12}y+c_x+A_I^\top\lambda_I+A_O^\top\lambda_O-s_O\\
            A_Ox-b_O\\
            s_O \circ x\\
            -Q_{22}y+Q_{12}^\top x+c_y+B_I^\top\lambda_I+s_I\\
        A_Ix+B_Iy-b_I\\
            s_I \circ y
    \end{bmatrix}.
\end{align*}
The Jacobian of $\Phi$ at $(\bar x,\bar\lambda_O,\bar s_O,\bar y,\bar\lambda_I,\bar s_I)$ is
\begin{align}\label{eq:DPhi}
    D\Phi(\bar x,\bar\lambda_O,\bar s_O,\bar y,\bar\lambda_I,\bar s_I)=\begin{bmatrix}
        Q_{11} & A_O^\top & -\mathds{1} & Q_{12} & A_I^\top & 0\\
        A_O & 0 & 0 & 0 & 0 & 0\\
        \diag(\bar s_O) & 0 & \diag(\bar x) & 0 & 0 & 0\\
        Q_{12}^\top & 0 & 0 & -Q_{22} & B_I^\top & \mathds{1}\\
        A_I & 0 & 0 & B_I & 0 & 0\\
        0 & 0 & 0 & \diag(\bar s_I) & 0 & \diag(\bar y)
    \end{bmatrix}.
\end{align}
  {Note that for the analysis of the matrix $D\Phi(\bar x,\bar\lambda_O,\bar s_O,\bar y,\bar\lambda_I,\bar s_I)$ at a solution point of \eqref{eq:KKTjoint} we may employ the additional conditions $\bar x,\bar s_O,\bar y,\bar s_I\geq0$ as well as $\bar y\in S(\bar x)$.}

\begin{definition}\label{def:ndstat}
    For a stationary point $\bar x$ of \eqref{eq:MinMax} with $\bar y\in S(\bar x)$ and multipliers $(\bar\lambda_O,\bar s_O,\bar\lambda_I,\bar s_I)$ we call $(\bar x,\bar y,\bar \lambda_O,\bar \lambda_I,\bar s_O,\bar s_I)$ non-degenerate if $D\Phi(\bar x,\bar\lambda_O,\bar s_O,\bar y,\bar\lambda_I,\bar s_I)$
is non-singular. We call $\bar x$ with $\bar y\in S(\bar x)$ a non-degenerate stationary point if it is stationary and if multipliers $(\bar\lambda_O,\bar s_O,\bar\lambda_I,\bar s_I)$ exist such that $(\bar x,\bar y,\bar \lambda_O,\bar \lambda_I,\bar s_O,\bar s_I)$ is non-degenerate.
\end{definition}

A natural sufficient condition for the non-singularity of $D\Phi(\bar x,\bar\lambda_O,\bar s_O,\bar y,\bar\lambda_I,\bar s_I)${at a solution point of \eqref{eq:KKTjoint}} arises from its block structure. It is strongly related to the notions of \eqref{ILICQ}, \eqref{ISCS},{\eqref{ISOSC}},\eqref{OLICQ},  {\eqref{OSCS} and \eqref{OSOSC}} introduced above.
To see this, we start with a discussion of the non-singularity of the lower right 3-by-3 block of $D\Phi(\bar x,\bar\lambda_O,\bar s_O,\bar y,\bar\lambda_I,\bar s_I)$,
  {namely the Jacobian $D_{(y,\lambda_I,s_I)}\Phi_I(\bar x,\bar y,\bar\lambda_I,\bar s_I)$ of $\Phi_I$ from \eqref{eq:defPhiI} (see \eqref{eq:DPhiI}).} 
By Laplace expansion of its determinant it is not hard to see that it is non-singular iff \eqref{ISCS} holds and the matrix{from \eqref{eq:M}} is non-singular. 

The latter can be characterized by 
the following result from \cite[Lemma 3.4]{MR787745} (see also \cite{MR907394}), where the inertia $\inertia(C)$ of a square matrix $C$ denotes the triple of the numbers of its positive, negative, and zero eigenvalues. Furthermore, for $A=A^\top \in \mathbb{R}^{q \times q }$ and $B \in \mathbb{R}^{t \times q}$ let the columns of $N$ form a basis of $\kernel B$. Then $A|_{\kernel B}:=N^\top A N$ is called the restriction of $A$ to $\kernel B$.

\begin{lemma} \label{lem:inertia}
 Let $A=A^\top \in \mathbb{R}^{q \times q }$ and let $B \in \mathbb{R}^{t \times q}$ possess rank $r$. Then the matrix  
 \begin{equation*}
     C= \begin{bmatrix}
         A & B^\top \\
         B & 0
         \end{bmatrix}
 \end{equation*}
 satisfies
\begin{equation}\label{eq:inertia}
    \inertia(C)=\inertia(A|_{\kernel B})+(r,r,t-r).
\end{equation}
 \end{lemma}
 
\begin{remark}\label{rem:emptymatrix}
 In the statement of Lemma~\ref{lem:inertia}, $t=q=r$ leads to $\kernel B=\{0\}$, so that no matrix $N$ can be chosen. This case can be formally encompassed by considering a $(q,0)$-matrix $N$, which makes $N^\top AN$ a $(0,0)$-matrix with inertia $(0,0,0)$. With this convention, the inertia formula of the lemma holds also in this case. In particular, for $A=0$ and $t=q=r$ not only the matrix $C$ is non-singular, but also the empty matrix $0|_{\{0\}}$ is, due the absence of vanishing eigenvalues.
\end{remark}

With the notation from Lemma~\ref{lem:inertia}, by \eqref{eq:inertia} $C$ is non-singular iff $A|_{\kernel B}$ is non-singular and $B$ possesses full row rank. Therefore, the non-singularity of the matrix in \eqref{eq:M}
is equivalent to the non-singularity of $-Q_{22}$ restricted to the tangent space (cf. \eqref{eq:defTI})
\begin{align*}
    T(\bar y,Y(\bar x))=\kernel\begin{bmatrix}
        B_I\\\1_{\bar y}\end{bmatrix}
\end{align*}
and to the full row rank of the matrix $\begin{bmatrix}
        B_I\\\1_{\bar y}\end{bmatrix}$, which coincides with \eqref{ILICQ}.
 Since we assume $-Q_{22}$ to be negative semi-definite, the non-singularity of $-Q_{22}|_{T(\bar y,Y(\bar x))}$ is equivalent to its negative definiteness,{which coincides with \eqref{ISOSC}}.
This motivates the following definition.

\begin{definition}
    For $\bar x\in X$, a point $\bar y\in S(\bar x)$ is called{inner} non-degenerate if{\eqref{ILICQ}, \eqref{ISCS} and \eqref{ISOSC} are satisfied.}
\end{definition}
The above derivation shows the following result.
\begin{lemma}\label{lem:ynd}
 For $\bar x\in X$ a point $\bar y\in S(\bar x)$ is{inner} non-degenerate if and only if  the Jacobian
 $D_{(y,\lambda_I,s_I)}\Phi_I(\bar x,\bar y,\bar\lambda_I,\bar s_I)$ is non-singular.
\end{lemma}

\begin{remark}\label{rem:innerlinear}
We emphasize that even for a linear inner problem $(P_I(\bar x))$ a point $\bar y\in S(\bar x)$ may be{inner} non-degenerate. Linearity implies $Q_{22}=0$, so that negative definiteness of $-Q_{22}|_{T(\bar x,Y(\bar x))}$ requires $T(\bar x,Y(\bar x))=\{0\}$ (see Remark~\ref{rem:emptymatrix}). For an{inner} non-degenerate point $\bar y$ this means that \eqref{ILICQ} must hold with $|{\cal I}_{\bar y} |=m-q$, so that in linear programming terminology $\bar y$ is a non-degenerate vertex of $Y(\bar x)$. Moreover, \eqref{ISCS} yields $[s_I]_{{\cal I}_{\bar y}}>0$ in this case, which in linear programming is sometimes called dual non-degeneracy of $\bar y$ for $(P_I(\bar x))$. 
\end{remark}

Next we discuss the non-singularity of the full Jacobian matrix given in \eqref{eq:DPhi}.
Again, by Laplace expansion of the determinant of $D\Phi(\bar x,\bar\lambda_O,\bar s_O,\bar y,\bar\lambda_I,\bar s_I)$, one sees that its non-singularity is equivalent to \eqref{OSCS}, \eqref{ISCS} and the non-singularity of
\begin{align}\label{eq:K}
    K:=\begin{bmatrix}
        Q_{11} & A_O^\top & -\mathds{1}_{\bar x}^\top & Q_{12} & A_I^\top & 0\\
        A_O & 0 & 0 & 0 & 0 & 0\\
        -\mathds{1}_{\bar x} & 0 & 0 & 0 & 0 & 0\\
        Q_{12}^\top & 0 & 0 & -Q_{22} & B_I^\top & \mathds{1}_{\bar y}^\top\\
        A_I & 0 & 0 & B_I & 0 & 0\\
        0 & 0 & 0 & \mathds{1}_{\bar y} & 0 & 0
    \end{bmatrix}=\begin{bmatrix}
        K_{11} & K_{12}\\K_{12}^\top&K_{22}
    \end{bmatrix}
\end{align}
with
\begin{align*}
    K_{11}:=\begin{bmatrix}
        Q_{11} & A_O^\top & -\mathds{1}_{\bar x}^\top\\
        A_O & 0 & 0 \\
        -\mathds{1}_{\bar x} & 0 & 0
    \end{bmatrix},\quad  K_{22}:=\begin{bmatrix}
        -Q_{22} & B_I^\top & \mathds{1}_{\bar y}^\top\\
        B_I & 0 & 0\\
        \mathds{1}_{\bar y} & 0 & 0
    \end{bmatrix},\quad
    K_{12}:=\begin{bmatrix}
        Q_{12} & A_I^\top & 0\\
        0 & 0 & 0\\
        0 & 0 & 0
    \end{bmatrix},
\end{align*}
and where $K_{22}$ is the matrix from \eqref{eq:M}.
Hence, if we assume $\bar y$ to be{inner} non-degenerate, then $K_{22}$ is non-singular by Lemma~\ref{lem:ynd}, and $K$ is non-singular if and only if the Schur complement
\begin{align*}
    K/K_{22} &= K_{11}-K_{12} K_{22}^{-1} K_{12}^\top= \begin{bmatrix}
        Q_{11}'& A_O^\top& -\mathds{1}_{\bar x}^\top\\
        A_O & 0 & 0\\
        -\mathds{1}_{\bar x} & 0 & 0
    \end{bmatrix}
\end{align*}
with{$Q_{11}'$ from \eqref{eq:D2phi} is.}
By Lemma~\ref{lem:inertia}, and since the sign of $\1_{\bar x}$ is irrelevant for this argument, the{non-singularity of $K/K_{22}$} is equivalent to \eqref{OLICQ} at $\bar x$ in $X$ and the non-singularity of $Q_{11}'|_{T(\bar x,X)}$, where{$T(\bar x,X)$ is the tangent space from \eqref{eq:defTO}. 
While above we used that the non-singularity of $-Q_{22}|_{T(\bar y,Y(\bar x))}$ is equivalent to its negative definiteness and, thus, to \eqref{ISOSC} at any solution of \eqref{eq:KKTjoint}, an analogous argument cannot be made for $Q_{11}'|_{T(\bar x,X)}$. Therefore we introduce the outer second order \emph{regularity} condition
\begin{equation}\label{OSORC}\tag{OSORC}
 Q_{11}'|_{T(\bar x,X)} \mbox{ is non-singular},
\end{equation}
which is weaker than \eqref{OSOSC}.}
 
\begin{definition}\label{def:outernd}
    A stationary point $\bar x$ of \eqref{eq:MinMax} with{inner} non-degenerate $\bar y\in S(\bar x)$ is called outer non-degenerate, if \eqref{OLICQ}, \eqref{OSCS}, and{\eqref{OSORC}} hold.{If instead of \eqref{OSORC} even \eqref{OSOSC} holds, then $\bar x$ is called a non-degenerate local minimal point (cf. Theorem~\ref{the:ndmeaningfulness}).}
\end{definition}

The above derivation proves the following result.

\begin{theorem}\label{the:ndstat}
For $\bar x\in X$, let $\bar y\in S(\bar x)$ be{inner} non-degenerate. Then $\bar x$ is a non-degenerate stationary point of \eqref{eq:MinMax} if and only if $\bar x$ is stationary and outer non-degenerate.
\end{theorem}
  {
\begin{example}\label{ex:running4}  
In the situation of Example~\ref{ex:running} the considerations from Example~\ref{ex:running2} yield that \eqref{ISOSC} and \eqref{OSORC} are trivially satisfied (Remark~\ref{rem:emptymatrix}). Hence, by Theorem~\ref{the:ndstat} $\bar x$ is a non-degenerate stationary point of \eqref{eq:MinMax}. This illustrates, in particular, that the singularity of $Q_{11}$ and $Q_{22}$ does not prevent non-degenerate stationarity. In addition, this situation is stable under variations of the parameter $\delta>4/3$.
\end{example}
}

\subsection{  {Non-degenerate} stationarity and local saddle points.}\label{Stationarity and local saddle points}

  {A main idea of the algorithmic approach presented in Section~\ref{sec:ipm} is to treat the variables $x$ and $y$ not separately, but as a joint variable $z=(x,y)$. Indeed, standard solution methods for minmax problems alternately solve the outer problem in $x$ and the inner problem in $y$, while our proposal is a single loop method in the sense that it solves a joint problem in the joint variable $z$. In this joint problem, the stationary points of \eqref{eq:MinMax} in the sense of Definition~\ref{Def:StationarityConcept} correspond to local saddle points. The corresponding results need a   reordering of the variable vector as well as of the equations in \eqref{eq:KKTjoint}. Since the Jacobian matrix of the reordered system will turn out to be congruent to $D\Phi(\bar x,\bar\lambda_O,\bar s_O,\bar y,\bar\lambda_I,\bar s_I)$, it is also non-singular under the assumptions of Theorem~\ref{the:ndstat}, i.e., under the conditions derived for the separate variable model.}

  {In the following we will study a} local saddle point property of the function $f$ on the combined outer and inner feasible set
\begin{align*}
    \Omega&=\{(x,y)\in X\times\R^m:\ y\in Y(x)\}\\
    &=\{(x,y)\in\R^n\times\R^m:\ A_Ox=b_O,\ A_Ix+B_Iy=b_I,\ x,y\geq0\}.
\end{align*}
In fact, $\Omega$ is the graph of the set-valued mapping $Y:X\rightrightarrows\R^m$. Furthermore, let 
\begin{equation}\label{eq:Matrix_Notation}
    z :=\begin{bmatrix}
	x \\
	y
\end{bmatrix}, \;
Q:=\begin{bmatrix}
	Q_{11}& Q_{12} \\
	Q_{12}^\top & -Q_{22} 
\end{bmatrix}, \;
c:=\begin{bmatrix}
	c_x \\
	c_y
\end{bmatrix},\; 
A:=\begin{bmatrix}
	A_O & 0  \\
	A_I         &   B_I 
\end{bmatrix},\;
b:= \begin{bmatrix}
	{{b_O}} \\
    b_I
\end{bmatrix},
\end{equation}
and consider the function $f(z)=\frac12 z^\top Q z+c^\top z$ on the set $\Omega$, which can now be written compactly as $\Omega=\{z\in\R^{n+m}:Az=b,\,z\geq0\}$. 

Recall that for an unconstrained differentiable function a point is called a local saddle point if it is a critical point, but neither a local minimal nor a local maximal point. Therefore, each neighborhood of the critical point contains points with smaller as well as points with larger function values than at the critical point. For a detailed comparison of the concepts of local and global saddle points we refer to \cite[Subsection 3.5]{Ste25}.

The generalization of this concept to constrained functions needs a stationarity concept which encompasses, both, local minimal and local maximal points. Since this is not the case for KKT points, following the terminology from \cite{jongen1986critical} we will call $\bar z$ a critical point of $f$ on $\Omega$ if a relaxed KKT system is satisfied which does not require the non-negativity of multipliers corresponding to inequality constraints. 

\begin{definition}\label{def:criticalpoint}
A point $\bar z\in\Omega$ is called
\begin{itemize}
    \item[(a)] critical point of $f$ on $\Omega$ if
multipliers $\bar\lambda\in\R^{p+q}$, $\bar\sigma\in\R^{n+m}$ exist such that $(\bar z,\bar\lambda,\bar\sigma)$ solves the system
\begin{align}\label{eq:Psi}
0=\Psi(z,\lambda,\sigma):=\begin{bmatrix}
    Qz+c+A^\top\lambda+\sigma\\ Az-b\\z\circ \sigma
\end{bmatrix},
\end{align}
\item[(b)] local saddle point of $f$ on $\Omega$ if $\bar z$ is a critical point, but neither a local minimal nor a local maximal point of $f$ on $\Omega$.
\end{itemize}
\end{definition}

\begin{lemma}\label{lem:statcrit}
A point $\bar x\in X$ with $\bar y\in S(\bar x)$ is stationary for \eqref{eq:MinMax} with multipliers $\bar\lambda_O$, $\bar s_O$, $\bar\lambda_I$, and $\bar s_I$ if and only if $\bar z=(\bar x,\bar y)\in\Omega$ is a critical point of $f$ on $\Omega$ with multipliers 
\[
\bar\lambda=\begin{bmatrix}
	\bar\lambda_O \\
	\bar\lambda_I
\end{bmatrix} \;\, \mbox{ and }\;\,
\bar\sigma=\begin{bmatrix}
	-\bar s_O \\
	\bar s_I
\end{bmatrix}, \;\,\bar s_O, \;\, \bar s_I\geq0.
\]
\end{lemma}
\begin{proof}
     
    In more detailed form the system $0=\Psi(z,\lambda,\sigma)$ with $\sigma:=\begin{bmatrix}
	-s_O \\
	s_I
\end{bmatrix}$ reads
\begin{align*}
    0=\Psi(x,y,\lambda_O,\lambda_I,-s_O,s_I)=\begin{bmatrix}
    Q_{11}x+Q_{12}y+c_x+A_O^\top\lambda_O+A_I^\top \lambda_I-s_O\\
    Q_{12}^\top x-Q_{22}y+c_y+B_I^\top\lambda_I+s_I\\
    A_Ox-b_O\\
    A_Ix+B_Iy-b_I\\
    x\circ(-s_O)\\
    y\circ s_I
\end{bmatrix}
\end{align*}
and is thus, up to reordering and a flip of sign in the fifth equation, identical to the equations from the stationarity system \eqref{eq:KKTjoint} in Definition~\ref{Def:StationarityConcept}. Hence   the ``only if'' part of the assertion is clear, and to see the ``if'' part, recall that the assumption implies that $\bar y$ is a KKT point of the inner problem and, thus, an element of $S(\bar x)$.
\end{proof}

To show that the stationary points of \eqref{eq:MinMax} are{typically} not only critical points, but even saddle points of $f$ on $\Omega$,{we introduce an appropriate non-degeneracy concept for critical points of $f$ on $\Omega$.}

\begin{definition}\label{def:ndcriticalpoint}
A point $\bar z\in\Omega$ is called
a non-degenerate critical point of $f$ on $\Omega$ if $\bar z$ is a critical point with some multipliers $\bar\lambda,\bar\sigma$ such that the Jacobian
\begin{align*}
    D\Psi(\bar z,\bar\lambda,\bar\sigma)=\begin{bmatrix}
        Q & A^\top & \1\\
        A & 0 & 0\\
        \diag(\bar\sigma) & 0 & \diag(\bar z)
    \end{bmatrix}
\end{align*}
of $\Psi$ from \eqref{eq:Psi} is non-singular.
\end{definition}
\begin{lemma}\label{lem:DPhiDPsi}
    A point $\bar x\in X$ with $\bar y\in S(\bar x)$ is a non-degenerate stationary point of \eqref{eq:MinMax} with multipliers $\bar\lambda_O$, $\bar s_O$, $\bar\lambda_I$, $\bar s_I$ if and only if $\bar z=(\bar x,\bar y)\in\Omega$ is a non-degenerate critical point of $f$ on $\Omega$ with multipliers $\bar\lambda=\begin{bmatrix}
	\bar\lambda_O \\
	\bar\lambda_I
\end{bmatrix}$ and
$\bar\sigma:=\begin{bmatrix}
	-\bar s_O \\
	\bar s_I
\end{bmatrix}$, $\bar s_O,\bar s_I\geq0$.
\end{lemma}
\begin{proof}
The equivalence of stationarity and criticality holds by Lemma~\ref{lem:statcrit}. Hence, it remains to show that 
\begin{align*}
    D\Phi(\bar x,\bar\lambda_O,\bar s_O,\bar y,\bar\lambda_I,\bar s_I)=\begin{bmatrix}
        Q_{11} & A_O^\top & -\mathds{1} & Q_{12} & A_I^\top & 0\\
        A_O & 0 & 0 & 0 & 0 & 0\\
        \diag(\bar s_O) & 0 & \diag(\bar x) & 0 & 0 & 0\\
        Q_{12}^\top & 0 & 0 & -Q_{22} & B_I^\top & \mathds{1}\\
        A_I & 0 & 0 & B_I & 0 & 0\\
        0 & 0 & 0 & \diag(\bar s_I) & 0 & \diag(\bar y)
    \end{bmatrix}
\end{align*}
is non-singular iff
\begin{align}
    D\Psi(\bar z,\bar\lambda,\bar\sigma)&=D\Psi(\bar x,\bar y,\bar\lambda_O,\bar\lambda_I,-\bar s_O,\bar s_I)\nonumber\\&=\begin{bmatrix}
        Q_{11} & Q_{12} & A_O^\top & A_I^\top & \1 & 0\\
        Q_{12}^\top & -Q_{22} & 0 & B_I^\top & 0 & \1\\
        A_O & 0 & 0 & 0 & 0 & 0\\
        A_I & B_I & 0 & 0 & 0 & 0\\
        -\diag(\bar s_O) & 0 & 0 & 0 & \diag(\bar x) & 0\\
        0 & \diag(\bar s_I) & 0 & 0 & 0 & \diag(\bar y)\label{eq:DPsi_explicit}
            \end{bmatrix}
\end{align}
is.
 This readily follows  from the congruence $D\Psi(\bar z,\bar\lambda,\bar\sigma)=V^\top D\Phi(\bar x,\bar\lambda_O,\bar s_O,\bar y,\bar\lambda_I,\bar s_I)\,V$ with
 \begin{align*}
  V=\begin{bmatrix}
        \1 & 0 & 0 & 0 & 0 & 0 \\
        0 & 0 & \1 & 0 & 0 & 0 \\
        0 & 0 & 0 & 0 & -\1 & 0 \\
        0 & \1 & 0 & 0 & 0 & 0 \\
        0 & 0 & 0 & \1 & 0 & 0 \\
        0 & 0 & 0 & 0 & 0 & \1     
      \end{bmatrix}.
 \end{align*}
  
\end{proof}
Lemma~\ref{lem:DPhiDPsi} and Theorem~\ref{the:ndstat} entail the following characterization of non-degenerate criticality.
\begin{corollary}\label{cor:ndcrit}
For $\bar x\in X$ let $\bar y\in S(\bar x)$ be{inner} non-degenerate. Then $\bar z=(\bar x,\bar y)$ is a non-degenerate critical point of \eqref{eq:MinMax} if and only if $\bar z$ is a critical point and $\bar x$ is outer non-degenerate.
\end{corollary}

First- and second-order properties of a non-degenerate critical point $\bar z$ of $f$ on $\Omega$ are readily encoded in its linear and quadratic indices, which we introduce next. 
First observe that, by Laplace expansion of its determinant, the Jacobian $D\Psi(\bar z,\bar\lambda,\bar\sigma)$ is non-singular iff 
the absolute strict complementary slackness (ASCS) $\min_{i=1,\ldots,n+m}(|\bar\sigma_i|+\bar z_i)>0$ holds and
\begin{align}\label{eq:C}
     C:=\begin{bmatrix}
         Q & A^\top & \1^\top_{\bar z}\\
        A & 0 & 0\\
        \1_{\bar z} & 0 & 0
    \end{bmatrix}
\end{align}
    is non-singular. Note that a non-negativity constraint on $\bar\sigma$ is not needed for the latter argument, leading to the ASCS in place of the usual strict complementary slackness condition.
    
Full row rank of $\begin{bmatrix}
        A\\\1_{\bar z}
    \end{bmatrix}$ means that the linear independence constraint qualification (LICQ) holds at $\bar z$ in $\Omega$. Furthermore, $T(\bar z,\Omega):=\kernel\begin{bmatrix}
        A\\\1_{\bar z}
    \end{bmatrix}$ is the tangent space to $\Omega$ at $\bar z$. Therefore, by Lemma~\ref{lem:inertia} the matrix $C$ from \eqref{eq:C} is non-singular if and only if the LICQ holds and $Q|_{T(\bar z,\Omega)}$ is non-singular. Altogether, this shows that $\bar z\in\Omega$ is a non-degenerate critical point of $f$ on $\Omega$ iff multipliers $\bar\lambda,\bar\sigma$ exist which satisfy $\Psi(\bar z,\bar\lambda,\bar\sigma)=0$, ASCS, LICQ and non-singularity of $Q|_{T(\bar z,\Omega)}$. 
    
    Following the terminology from \cite{jongen1986critical}, we call the number of negative entries of $\bar\sigma_{{\cal I}_{\bar z}}$ linear index (LI) and the number of negative eigenvalues of $Q|_{T(\bar z,\Omega)}$ quadratic index (QI) of $\bar z$. Standard first- and second-order optimality conditions imply that a non-degenerate critical point is a local minimal point iff LI and QI both vanish, a local maximal point iff LI and QI attain their maximal values $|{\cal I}_{\bar z}|$ and $(n+m)-(p+q+|{\cal I}_{\bar z}|)$, respectively, and otherwise a local saddle point.

    For the statement of the following result recall that outer non-degeneracy of a stationary point $\bar x$ of \eqref{eq:MinMax} in the sense of Definition~\ref{def:outernd} entails that the $(n-(p+|{\cal I}_{\bar x}|),(n-(p+|{\cal I}_{\bar x}|))$-matrix $Q_{11}'|_{T(\bar x,X)}$ is non-singular. Hence, assuming that it possesses $\varrho$ negative eigenvalues is equivalent to its inertia being $(n-p-|{\cal I}_{\bar x}|-\varrho,\varrho,0)$.
   \begin{theorem}\label{the:ndsaddle}
    For $\bar x\in X$ let $\bar y\in S(\bar x)$ be inner non-degenerate, and let $\bar x$ be a stationary and outer non-degenerate point of \eqref{eq:MinMax}, where $Q_{11}'|_{T(\bar x,X)}$ possesses $\varrho$ negative eigenvalues. Then $\bar z=(\bar x,\bar y)$ is a non-degenerate critical point of $f$ on $\Omega$ with LI=$|{\cal I}_{\bar y}|$ and QI=$m-q-|{\cal I}_{\bar y}|+\varrho$.
\end{theorem}
\begin{proof}
Let $\bar\lambda_O$, $\bar s_O$, $\bar\lambda_I$, $\bar s_I$ be multipliers from the stationarity condition of $\bar x$. By Theorem~\ref{the:ndstat} $\bar z=(\bar x,\bar y)$ is a non-degenerate stationary point of \eqref{eq:MinMax}, and by Lemma~\ref{lem:DPhiDPsi} also a non-degenerate critical point with multipliers $\bar\lambda=\begin{bmatrix}
	\bar\lambda_O \\
	\bar\lambda_I
\end{bmatrix}$ and
$\bar\sigma=\begin{bmatrix}
	-\bar s_O \\
	\bar s_I
\end{bmatrix}$.
Furthermore, \eqref{OSCS} and \eqref{ISCS} imply that the first $|{\cal I}_{\bar x}|$ entries of $\bar\sigma_{{\cal I}_{\bar z}}$ are positive, whereas the remaining $|{\cal I}_{\bar y}|$ entries are negative. Therefore the linear index of $\bar z$ is $|{\cal I}_{\bar y}|$.

To determine the quadratic index QI of $\bar z$, we compute the inertia $\inertia(Q|_{T(\bar z,\Omega)})$. Lemma~\ref{lem:inertia} and Sylvester's law of inertia for the congruent matrices $C$ from \eqref{eq:C} and $K$ from \eqref{eq:K} imply
\begin{align*}
    \inertia(Q|_{T(\bar z,\Omega)})&=\inertia(C)-(p+q+|{\cal I}_{\bar z}|,p+q+|{\cal I}_{\bar z}|,0)=\inertia(K)-(p+q+|{\cal I}_{\bar z}|,p+q+|{\cal I}_{\bar z}|,0).
\end{align*}
Furthermore, by the additivity formula for the inertia of Schur complements from 
\cite{haynsworth1968determination} we have
\begin{align*}
    \inertia(K)&=\inertia(K/K_{22})+\inertia(K_{22})=\inertia \begin{bmatrix}
        Q_{11}'& A_O^\top& -\mathds{1}_{\bar x}^\top\\
        A_O & 0 & 0\\
        -\mathds{1}_{\bar x} & 0 & 0
    \end{bmatrix}+\inertia \begin{bmatrix}
        -Q_{22} & B_I^\top & \mathds{1}_{\bar y}^\top\\
        B_I & 0 & 0\\
        \mathds{1}_{\bar y} & 0 & 0
    \end{bmatrix}
\end{align*}
with $Q_{11}'$ from \eqref{eq:D2phi}. Another application of Sylvester's law of inertia and Lemma~\ref{lem:inertia} yield
\begin{align*}
    \inertia(K)=&\inertia(Q_{11}'|_{T(\bar x,X)})+(p+|{\cal I}_{\bar x}|,p+|{\cal I}_{\bar x}|,0)+\inertia(-Q_{22}|_{T(\bar y,Y(\bar x)})+(q+|{\cal I}_{\bar y}|,q+|{\cal I}_{\bar y}|,0).
\end{align*}
The assumption on the number of negative eigenvalues of $Q_{11}'|_{T(\bar x,X)}$ and the non-degeneracy of $\bar y\in S(\bar x)$, respectively, imply
\begin{align*}
    \inertia(Q_{11}'|_{T(\bar x,X)})&=(n-p-|{\cal I}_{\bar x}|-\varrho,\varrho,0),\\
    \inertia(-Q_{22}|_{T(\bar y,Y(\bar x)})&=(0,m-q-|{\cal I}_{\bar y}|,0)
\end{align*}
and thus
\begin{align*}
    \inertia(K)=(n-\varrho,p+|{\cal I}_{\bar x}|+\varrho,0)+(q+|{\cal I}_{\bar y}|,m,0)=(n+q+|{\cal I}_{\bar y}|-\varrho,m+p+|{\cal I}_{\bar x}|+\varrho,0)
    \end{align*}
as well as
\begin{align*}
    \inertia(Q|_{T(\bar z,\Omega)})&=(n+q+|{\cal I}_{\bar y}|-\varrho,m+p+|{\cal I}_{\bar x}|+\varrho,0)-(p+q+|{\cal I}_{\bar z}|,p+q+|{\cal I}_{\bar z}|,0)\\
    &=(n-p-|{\cal I}_{\bar x}|-\varrho,m-q-|{\cal I}_{\bar y}|+\varrho,0).
\end{align*}
Since QI is the second entry of the inertia triple, this shows the assertion.
\end{proof}

\begin{corollary}\label{cor:ndstat2}
Let the assumptions of Theorem~\ref{the:ndsaddle} hold and let $Q_{11}'|_{T(\bar x,X)}$ be positive definite. Then $\bar z$ is a local saddle point of $f$ on $\Omega$.
\end{corollary}
\begin{proof}
The positive definiteness of $Q_{11}'|_{T(\bar x,X)}$ is equivalent to $\varrho=0$.
Assume that the non-degenerate critical point $\bar z$ is a local minimal point of $f$ on $\Omega$.
As mentioned before the statement of Theorem~\ref{the:ndsaddle}, this is characterized by the vanishing of both LI and QI. In view of  Theorem~\ref{the:ndsaddle} the latter is equivalent to ${\cal I}_{\bar y}=\emptyset$ and $0=m-q-|{\cal I}_{\bar y}|+\varrho=m-q$. This, however, contradicts the blanket assumption $q<m$. 

Analogously, assume that $\bar z$ is a local maximal point of $f$ on $\Omega$. This is characterized by LI and QI attaining their maximal values $|{\cal I}_{\bar x}|+|{\cal I}_{\bar y}|$ and $n-p-|{\cal I}_{\bar x}|+m-q-|{\cal I}_{\bar y}|$, respectively. By Theorem~\ref{the:ndsaddle} this is equivalent to ${\cal I}_{\bar x}=\emptyset$ and $0=n-p-|{\cal I}_{\bar x}|-\varrho=n-p$. This contradicts the blanket assumption $p<n$. 
Therefore the critical point $\bar z$ is neither a local minimal nor a local maximal point of $f$ on $\Omega$, hence a local saddle point.
\end{proof}
The assumptions of Corollary~\ref{cor:ndstat2} mean that $\bar x$ is a non-degenerate local minimal point of \eqref{eq:MinMax}, that is, the type of outer non-degenerate point $\bar x$ which we wish to find algorithmically. 

\begin{example}\label{ex:running5}
    In the situation of Example~\ref{ex:running} we have seen in Example~\ref{ex:running4} that the assumptions of Corollary~\ref{cor:ndstat2} are satisfied. In particular, \eqref{OSOSC} holds trivially. Hence, $(\bar x,\bar y)$ is a local saddle point of $f$ on $\Omega$.
\end{example}




\subsection{Summary of the relationships.}
In this section, we pictorially summarize the relationships intervening between the stationarity concept presented in Definition \ref{Def:StationarityConcept} and the notions of local optimality and saddle point studied throughout this section; cf. Fig. \ref{fig:implications}.

\begin{figure}[htbp]
\centering
\begin{tikzpicture}[
  box/.style={
    draw,
    rounded corners,
    align=center,
    minimum width=3.2cm,
    minimum height=1.1cm
  },
  note/.style={font=\small},
  implication/.style={
    double,
    -{Implies[length=6pt,width=8pt]},
    line width=0.6pt
  },
  node distance=2.4cm and 2.4cm
]

\node[box] (b2) {$\bar x$ locally optimal for \eqref{eq:OuterProblem}}; 
\node[box, right=of b2] (b3) {$\bar x$ is stationary}; 

\node[box, below=of b2] (b1) {$(\bar x, \bar y, \bar{\lambda}_I, \bar{s}_I)$ locally optimal for \eqref{eq:QPmodel}}; 
\node[box, below=of b3] (b4) {$(\bar x, \bar y)$ is a local saddle point}; 


\draw[implication] (b2) -- (b3);

\draw[implication, bend left=25]
  (b3) to node[above] {(Req2)} (b2);

\draw[implication] (b2) -- (b1);

\draw[implication, bend left=25]
  (b1) to node[left] {(Req1)} (b2);

\draw[implication] 
  (b3) to node[right] {(Req3)} (b4);

\end{tikzpicture}
\caption{
Summary of relationships. 
The local optimality in the right hand side of the implication $\bar x$ \textit{is stationary} $\Longrightarrow$ $\bar x$ \textit{locally optimal for} \eqref{eq:OuterProblem} is, in fact, strict.}
\label{fig:implications}
\end{figure}
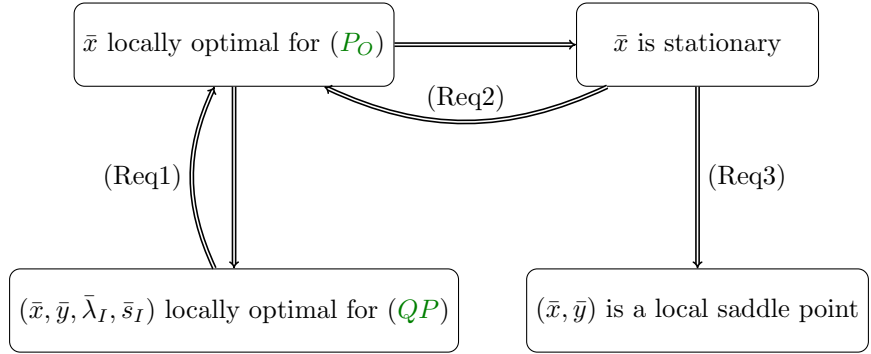

In Fig. \ref{fig:implications}, the first requirement (Req1) corresponds to the inner semicontinuity of $S_D$ at the point $(\bar x, \bar y, \bar{\lambda}_I, \bar{s}_I)${from Proposition~\ref{thm:LocalRelationshipTo(QP)}(b)}, while the second requirement (Req2) represents all the assumptions made in{Theorem \ref{the:ndmeaningfulness}}. As for the third requirement (Req3), it collects all the assumptions in{Corollary~\ref{cor:ndstat2}}.

\section{An interior point method.}\label{sec:ipm}
In this section, we introduce an interior point method to compute a solution of the stationarity condition \eqref{eq:KKTjoint}. To this aim, let us consider, 
given a parameter $\mu >0$,  the barrier minmax problem 
\begin{equation}\label{eq:MinMax_b}\tag{\text{P$_\mu$}}
	\min_{x\in \bar{X} } \, \max_{y\in \bar{Y}(x)} f(x,y) -\mu \sum^n_{i=1} \log(x_i) + \mu \sum^m_{j=1} \log(y_j),
\end{equation}
where $f$ is defined in \eqref{eq:f(x,y)} and $\bar{X}$ and $\bar{Y}(x)$ are given by  
\[
\bar{X}:=\left\{x \in \mathbb{R}^n:\;\;  A_Ox = b_O\right\} \; \mbox{ and }\; 
\bar{Y}(x) := \{ y \in \mathbb{R}^m:\;\;  A_Ix+B_Iy = b_I\},
\]
respectively. 
Obviously, problem \eqref{eq:MinMax_b} can be written as the minimization of the function $\varphi_{\mu}(x)$ over the outer feasibility set $\bar{X}$ with $\varphi_\mu(x):=\tilde{\varphi}_{\mu}(x) - \hat{\varphi}_{\mu}(x)$, where
\[
\left\{\begin{array}{lll}
   \tilde{\varphi}_\mu(x) & := & \frac{1}{2}x^\top (Q_{11})x + c^\top_x x  - \mu \sum^n_{i=1} \log(x_i),\\[2ex]
   \hat{\varphi}_\mu(x) & := & \underset{y\in \bar{Y}(x)}{\min} \,\frac{1}{2}y^\top (Q_{22}) y- x^\top Q_{12} y - c^\top_y y - \mu \sum^m_{j=1} \log(y_j).
\end{array}\right.
\]

Now, let $\bar x$ with $\bar{x}_i >0$, $i=1, \ldots, n$, be a local optimal solution of  \eqref{eq:MinMax_b}, in a sense similar to \eqref{def:localOptSol}; i.e., for some neighborhood $U$ of $\bar x$,
\[
\bar x\in \underset{\,\quad x\in \bar{X}\cap U}{\text{\text{argmin}}}~ \varphi_\mu(x).
\]
If we assume that the optimal solution set-valued mapping
\[
S_{\mu}(x) := \underset{y\in \bar{Y}(x)}{\arg\min} \,\frac{1}{2}y^\top (Q_{22}) y- x^\top Q_{12} y - c^\top_y y - \mu \sum^m_{j=1} \log(y_j)
\]
is {inner semicontinuous} at $(\bar x, \bar y)$ with $\bar y\in \text{S}_{\mu} (\bar x)$ and the matrix $B^\top_I$ possesses full column rank, 
then the function $\hat{\varphi}_\mu$ is Lipschitz continuous around $\bar x$ and its Clarke subdifferential at this point  can be estimated as  
\[
     \partial \hat{\varphi}_\mu(\bar x) \subset \left\{-Q_{12}\bar y + A^\top_I \lambda_I:\;\, Q_{22} \bar y -Q_{12}\bar y + B^\top_I \lambda_I - c_y -\left[\begin{array}{c}
     \frac{\mu}{\bar{y}_1}\\
     \vdots\\
     \frac{\mu}{\bar{y}_m}
     \end{array}\right] =0, \; \lambda_I \in \mathbb{R}^q\right\};
\]
 see, e.g., \cite{mordukhovich2012variational}. Therefore, as $\bar x$ is locally optimal for problem \eqref{eq:MinMax_b}, 
\begin{equation}\label{Eq:OptCond_b}
    0\in \partial \varphi_{\mu}(\bar x) + N_{\bar X} (\bar x),
\end{equation}
where  $N_{\bar X}$ stands for the normal cone to $\bar X$ in the sense of convex analysis. Then observing that we have the formulas 
\[
\nabla \tilde{\varphi}_{\mu}(\bar{x})= Q_{11} \bar x + c_x - \left[\begin{array}{c}
     \frac{\mu}{\bar{x}_1}\\
     \vdots\\
     \frac{\mu}{\bar{x}_n}
     \end{array}\right] \;\mbox{ and }\; N_{\bar{X}}(\bar x)= A^\top_O\mathbb{R}^p,
\]
it follows from \eqref{Eq:OptCond_b} and the upper estimate of $\partial \hat{\varphi}_\mu(\bar x)$ given above that there exist Lagrange multipliers $\lambda_O\in \mathbb{R}^p$ and $\lambda_I\in \mathbb{R}^q$ such that we have  the stationary conditions  
\begin{equation}\label{eq:KKT_final}
	\begin{split}
		& Qz + c +A^\top \lambda - \begin{bmatrix}
			\1 & 0 \\
			0 & -\1
		\end{bmatrix} s = 0, \\
		&  Az = b, \\ 
		&  z \circ s = \mu e,  \\
		& z, s > 0,
	\end{split}
\end{equation}
while considering the notation in \eqref{eq:Matrix_Notation} and setting $(s_O)_i:= \frac{\mu}{\bar{x}_i}$ for $i=1, \ldots, n$ and $(s_I)_j:= \frac{\mu}{\bar{y}_j}$ for $j=1, \ldots, m$. A direct comparison of equations \eqref{eq:KKTjoint} and \eqref{eq:KKT_final} reveals that the former differs from the latter solely in the presence of the parameter $\mu$ in the complementarity conditions and the positivity of $z$ and $s$. 

Interpreting the system \eqref{eq:KKT_final} as a perturbation of \eqref{eq:KKTjoint} constitutes the classical foundation for devising interior point methods  (see, e.g., \cite{MR4865731}), and indeed, intuitively, the stationarity system \eqref{eq:KKT_final} approximates \eqref{eq:KKTjoint} more and more closely when $\mu$ converges to $0$. Motivated by the barrier interpretation presented above, we analyze in the subsequent sections how classical IPM solvers can be adapted to solve \eqref{eq:KKTjoint} and establish their theoretical guarantees. 

Note that \eqref{eq:KKTjoint} represents a particular instance of a mixed linear complementarity problem and, more generally, of a constrained equation problem, both of which are amenable to specialized IPM-type solvers (see, e.g., \cite{MR3396730, MR1406748}). However, to the best of our knowledge, adopting such a perspective does not yield substantially improved convergence or complexity and guarantees. Hence, we employ here a standard infeasible IPM framework; see, e.g.,  \cite[Chapter 6]{MR1422257}.

Before presenting the details of the IPM method, let us present a series of notational changes or additions to enable our presentation to be closer to standard IPM notation. In particular, in the following, we will set for any $z \in \mathbb{R}^{n+m} $, $Z:=\diag(z)$ and hence, $z \circ s = ZSe$,{whereas any inequalities involving a vector and a scalar are meant to hold component-wise}. Moreover, aiming at devising IPM--type solvers for the system \eqref{eq:KKTjoint}, we introduce the map
\begin{equation} \label{eq:KKT_for_IPM}
    F({z,\lambda,s}) := \begin{bmatrix}
        Qz + c +A^\top \lambda - \begin{bmatrix}
			 \1 & 0 \\
			0 & -\1
		\end{bmatrix} s  \\
		  Az - b \\ 
		  ZSe   
    \end{bmatrix}  \hbox{ with } z,s \geq 0,
\end{equation}
as a version of the mapping $\Psi(z,\lambda,(-s_O,s_I))$ from \eqref{eq:Psi} in standard IPM notation. 

Multiplying the first block row of the function $F$ by $\begin{bmatrix}
		\1 & 0 \\
		0 & -\1
	\end{bmatrix}$,
we get the function 
	\begin{equation*} 
		\hat{F}(z,\lambda,s) := \begin{bmatrix}
			\begin{bmatrix}
				\1 & 0 \\
				0 & -\1
			\end{bmatrix} (Qz + c +A^\top \lambda) - s  \\
			Az - b \\ 
			ZSe   
		\end{bmatrix}   \hbox{ with } z,s \geq 0,
	\end{equation*}
which can be employed to equivalently write the stationarity system \eqref{eq:KKTjoint}.{We can hence succinctly re-write the stationarity system \eqref{eq:KKTjoint} as} 

	\begin{equation}\label{eq:F_def}
		\hat{F}(z,\lambda,s) :=\begin{bmatrix}
			\hat{Q}z +\hat{c} + \hat{A}^\top \lambda - s  \\
			Az - b \\ 
			ZSe   
		\end{bmatrix}{=0,}  \hbox{ with } z,s \geq 0,
	\end{equation}
  {where the expression $\hat{F}(z,\lambda,s)$ is written in the compact form  using the notation}
\begin{equation*}
    \begin{array}{rllll}
      \hat{Q} &:= & \begin{bmatrix}
				\1 & 0 \\
				0 & -\1
			\end{bmatrix}Q & = & \begin{bmatrix}
				Q_{11} & Q_{12} \\
				-Q_{12}^\top & Q_{22} 
			\end{bmatrix},\\[3ex]
    \hat{c} & := & \begin{bmatrix}
				\1 & 0 \\
				0 & -\1
			\end{bmatrix}c & = & \begin{bmatrix}
				c_x \\ -c_y
			\end{bmatrix},\\[3ex]
    \hat{A^\top} &:=&\begin{bmatrix}
				\1 & 0 \\
				0 & -\1
			\end{bmatrix}A^\top & = &\begin{bmatrix}
				A_O^\top & A_I^\top \\
				0 & -B_I^\top
			\end{bmatrix}.
    \end{array}
\end{equation*}
From now on, we will focus only on $\hat{F}$ {and hence, on the computation of stationary points}. Additionally, in the sequel, we will use the notation 
\[
B(v_0,r):=\left\{ v \; \mid \; \|v-v_0\| <r \right\} \; \mbox{ and } \; \overline{B(v_0,r)}:=\left\{ v \; \mid \; \|v-v_0\| \leq r \right\},
\]
while employing, {when beneficial for readability,} the shorthand $v := (z,\lambda, s )$.

\subsection{On the local existence of the central path.}

Of utmost importance for the development of any kind of IPM is the concept of \textit{central path}. To this aim, for any $\mu>0$, we introduce the perturbation
{of the function $\hat{F}$ in \eqref{eq:F_def}, i.e.,}
\begin{equation} \label{eq:Central_Path_Equation}
    {H(z,\lambda, s, \mu )} := \begin{bmatrix}
        \hat{Q}z + \hat{c} +\hat{A}^\top \lambda - s  \\
		  Az - b \\ 
		  ZSe- \mu e   
    \end{bmatrix} = \hat{F}(z,\lambda, s) - \mu \begin{bmatrix}
        0 \\
        0 \\
       e
    \end{bmatrix},
\end{equation}
{with $H: \mathbb{R}^{N+1} \rightarrow \mathbb{R}^{N}$ with $H \in C^1(\mathbb{R}^{N+1},\mathbb{R}^{N})$, where $N:=2(n+m) + p+q$. Subsequently, we} 
define  the \textit{central path}  by 
\[
\mathcal{C}:=\left\{(z_\mu, \lambda_{\mu}, s_{\mu}) \in  {\mathbb{R}^{n+m}_{>0} \times \mathbb{R}^{p+q} \times \mathbb{R}^{n+m}_{>0} } \; \mid \; {H(z,\lambda, s, \mu )=0} \; \hbox{ for } \mu>0\right\}.
\]

{The introduced notation will allow us to establish the local existence of the central path and other properties of the Newton matrices by means of the Implicit Function Theorem. In this regard, and throughout this section, we will make the crucial assumption:}



\begin{hypothesis} \label{hyp:non_singularity}
  {The point $v^*=(z^*,\lambda^*, s^*)$  satisfies $z^*,s^* \geq 0$, $H(v^*,0)=0$,  and $\partial_{v} H(v^*,0)$ is invertible.}
\end{hypothesis}

The next lemma ensures that Theorem~\ref{the:ndstat} provides a sufficient condition for Assumption~\ref{hyp:non_singularity} in terms of the mild inner and outer non-degeneracy requirements for \eqref{eq:MinMax} from Section~\ref{sec:Stationary and saddle points}.
\begin{lemma}
    Assumption~\ref{hyp:non_singularity} holds at $v^*=(z^*,\lambda^*, s^*)$ with $z^*:=(x^*,y^*)$, $\lambda^*:=(\lambda_O^*,\lambda_I^*)$, and $s^*:=(s_O^*,s_I^*)$ if and only if $x^*$ is a non-degenerate stationary point of \eqref{eq:MinMax} in the sense of Definition~\ref{def:ndstat} with $y^*\in S(x^*)$ and multipliers $(\lambda_O^*,s_O^*,\lambda_I^*,s_I^*)$.
\end{lemma}
\begin{proof}
    By \eqref{eq:F_def} stationarity of $x^*$ with corresponding $y^*\in S(x^*)$ and $(\lambda_O^*,s_O^*,\lambda_I^*,s_I^*)$ is equivalent to $H(v^*,0)=0$, $z^*,s^*\geq0$. Moreover, the matrix 
\begin{equation*}
	\partial_v H(v^*,0)= 		\begin{bmatrix}
				\hat{Q} & \hat{A}^\top & -\1 \\
				A & 0 &  0\\
				S^* & 0 & Z^*
			\end{bmatrix}=\begin{bmatrix}
        Q_{11} & Q_{12} & A_O^\top & A_I^\top & -\1 & 0\\
        -Q_{12}^\top & Q_{22} & 0 & -B_I^\top & 0 & -\1\\
        A_O & 0 & 0 & 0 & 0 & 0\\
        A_I & B_I & 0 & 0 & 0 & 0\\
        \diag(s_O^*) & 0 & 0 & 0 & \diag(x^*) & 0\\
        0 & \diag(s_I^*) & 0 & 0 & 0 & \diag(y^*)
            \end{bmatrix}
\end{equation*}
may be written as
\begin{align*}
\begin{bmatrix}
     \1 & 0 & 0 & 0 & 0 & 0 \\
    0 & \1 & 0 & 0 & 0 & 0 \\
    0 & 0 & \1 & 0 & 0 & 0  \\
     0 & 0 & 0 & \1 & 0 & 0 \\
     0 & 0 & 0 & 0 & -\1 & 0 \\
     0 & 0 & 0 & 0 & 0 & \1 
 \end{bmatrix}
    D\Psi(x^*,y^*,\lambda_O^*,\lambda_I^*,-s_O^*,s_I^*)\begin{bmatrix}
     \1 & 0 & 0 & 0 & 0 & 0 \\
    0 & -\1 & 0 & 0 & 0 & 0 \\
    0 & 0 & \1 & 0 & 0 & 0  \\
     0 & 0 & 0 & \1 & 0 & 0 \\
     0 & 0 & 0 & 0 & -\1 & 0 \\
     0 & 0 & 0 & 0 & 0 & \1 
 \end{bmatrix}
\end{align*}
with $D\Psi$ from \eqref{eq:DPsi_explicit}. Therefore the non-singularity of $\partial_v H(v^*,0)$ means that $z^\star\geq0$ is a non-degenerate critical point in the sense of Definition~\ref{def:ndcriticalpoint} with $s^*\geq0$, and Lemma~\ref{lem:DPhiDPsi} shows the assertion.
\end{proof}


Given the form of $\hat{F}$ and Assumption \ref{hyp:non_singularity}, we can state the following lemma, propaedeutic for the specific version of the Implicit Function Theorem that we will consider (see Theorem \ref{Th:IFT}), but also for developments in Subsection \ref{sec:coupled_local_convergence}.

\begin{lemma} \label{lem:uniform_boundedness}
There exist $\delta_1, \, L, \, K > 0$ such that for all $v, v' \in B(v^*, \delta_1)$,
\[
\|D\hat{F}(v) - D\hat{F}(v')\| \leq L\|v - v'\| \; \mbox{ and } \; \|D\hat{F}(v)^{-1}\| \leq K.
\]
\end{lemma}

The specific version of the Implicit Function Theorem here considered, see Theorem \ref{Th:IFT}, will be instrumental in proving the local existence of the central path. For the sake of completeness, we report the proof of such a theorem in Section \ref{sec:Implicit Function Theorem}.

\begin{theorem} \label{Th:IFT}
  Consider 
  \[
  V_{\delta}:= \left\{ (v,\mu)\left|\;  \|v-v^*\| \leq \delta, \; |\mu| \leq \delta\right.\right\}
  \] 
  and select $\delta>0$ such that 
\begin{equation*}
    \sup_{(v,\mu) \in V_{\delta}} \|I - \partial_v H(v^*,0)^{-1}\partial_v H(v,\mu) \| \leq \frac{1}{2}
\end{equation*}
and  $\partial_v H(v,\mu)^{-1}$ exists for all    $(v,\mu) \in V_{\delta}$. Furthermore, let 
\[
B_{\delta} := \max\left\{ 1,\;\,  \sup_{(v,\mu) \in V_{\delta}}  \| \partial_{\mu} H(v,\mu)\| \right\} \,\hbox{ and }\, M:=\max\left\{1,\;\, \| \partial_v H(v^*,0)^{-1}\|\right\}.
\]
Define $\bar{\mu}:=(2MB_{\delta})^{-1}\delta$ and $\Lambda_{\bar{\mu}}:=\{ \mu\in \mathbb{R}|\; |\mu| < \bar{\mu}\}$.  Then there exists a function $g \in C^{1}(\Lambda_{\bar{\mu}}, \mathbb{R}^N)$ such that all the solutions of the equation $H(v,\mu)=0$ in the set 
$
\overline{B({v^*},\delta)}\times \Lambda_{\bar{\mu}}
$
are given by $(g(\mu), \mu)$. In addition, there exists $C>0$ such that 
\begin{equation*}
    \|g(\mu_1)-g(\mu_2)\| \leq C |\mu_1-\mu_2| \hbox{ for all } \mu_1, \mu_2 \in \Lambda_{\bar{\mu}}
\end{equation*}
and  
\begin{equation*}
    \frac{d}{d\mu}g(\mu)=-(\partial_v H(g(\mu),\mu))^{-1}\partial_\mu H(g(\mu),\mu).
\end{equation*}
\end{theorem}

This remark presents the details about the local existence of the central path.
\begin{remark} \label{rema:Local_Positivity}
    Using the differentiability of $g$ at $\mu=0$, we have that
\begin{equation} \label{eq:central_path_from_sol}
    g(\mu)=v^* \underbrace{- (\partial_v H(v^*,0))^{-1}\begin{bmatrix}
        0 \\
        0 \\
       - e
    \end{bmatrix}}_{\frac{d}{d\mu}g(0) =: v'(0)} \mu +o(\mu).
\end{equation}
That is, 
\begin{equation*}
    (\partial_v H(v^*,0)) v'(0) = \begin{bmatrix}
        0 \\
        0 \\
        e
    \end{bmatrix} \;\mbox{ with }\; \partial_v H(v^*,0)= 		\begin{bmatrix}
				\hat{Q} & \hat{A}^\top & -\1 \\
				A & 0 &  0\\
				S^* & 0 & Z^*
			\end{bmatrix}.
\end{equation*}
Hence, looking at the last set of block equations $[s^*]_i [z'(0)]_i+ [z^*]_i [s'(0)]_i =1 $ for $i=1,\dots,m+n$ and recalling that $Z^*S^*e=0$, we have that it must hold

\begin{equation} \label{eq:a}
    a:= \min_i \{ [s^*]_i + [z^*]_i \} >0;
\end{equation}
i.e., $s^*_i$ and $z^*_i$ can not be zero at the same time. Hence, that $s^*_i= 0$ and $z^*_i =0$ imply $s'(0)_i>0$ and  $z'(0)_i>0$, respectively. On the other hand, when $s^*_i> 0$ or  $z^*_i >0 $ we do not have information about the signs of $s'(0)_i$ or $z'(0)_i$. This does not represent a problem, because we can still conclude that for $\mu$ small enough, $z(\mu),s(\mu)>0$. It is important to note that not only the quantity $a$ defined \eqref{eq:a} is a measure of the radius where $D\hat{F}(v)$ is invertible, it must be indeed $a<\delta_1$, see \cite[Section 4]{MR2411396}, where $\delta_1$ is from Lemma \ref{lem:uniform_boundedness}, but also, as it appears from \eqref{eq:central_path_from_sol}, a measure of \textit{how far} one can extend, in general, the central path from the solution $v^*$.
\end{remark}
Similarly to \cite{MR2411396}, with reference to the quantities defined in Lemma \ref{lem:uniform_boundedness} and Theorem \ref{Th:IFT}, let 
\begin{equation} \label{eq:delta_star}
\delta^* := \min\left\{\delta, {\delta_1}, \frac{1}{2KL}\right\} \quad \text{and} \quad \mu^* := \min\left\{\bar{\mu}, \frac{\delta^*}{C}\left(1 + \sqrt{1 + 2KL\delta^*}\right)^{-1}\right\}.
\end{equation}

\subsection{An infeasible interior point method.}

  {In Algorithm \ref{alg:ipm}, we present the main computational scheme considered in this work. To this aim, we need the following notation:}
	\begin{align}
		r_b &:= Az - b, \label{eq:6.1a} \\
		r_c &:= \hat{Q}z + \hat{A}^\top\lambda + \hat{c}- s . \label{eq:6.1b},
	\end{align}
and
	\begin{equation*} \label{eq:6.7}
		\begin{split}
			(z(\alpha), \lambda(\alpha), s(\alpha)) & := (z, \lambda,s)+\alpha(\Delta z, \Delta \lambda, \Delta s),\\
			\mu(\alpha) & := \frac{1}{n+m}z(\alpha)^\top s(\alpha).
		\end{split}
	\end{equation*}
    
	When{$r_b$ and $r_c$} are computed at the point $(z, \lambda, s) = (z^k, \lambda^k, s^k)$, we denote them by $r_b^k$ and $r_c^k$. The algorithm here considered, see \cite[Chapter 6]{MR1422257}, works with the central path neighborhood $  {\mathcal{N}_{-\infty}(\underline{\gamma},\beta)}$ defined by
	\begin{equation} \label{eq:6.2}
		\begin{array}{rll}
	  {\mathcal{N}_{-\infty}(\underline{\gamma},\beta)} & = & \left\{ (z,\lambda,s) \in  \mathbb{R}^{n+m}_{>0} \times \mathbb{R}^{p+q} \times \mathbb{R}^{n+m}_{>0} \right. \\[1.5ex]
	& & \left. \mid \|(r_b,r_c)\| \leq \left[\frac{\|(r_b^0,r_c^0)\|}{\mu_0}\right]\beta\mu,  \quad z_i s_i \geq \underline{\gamma}\mu,\; i=1,2,\ldots,n\right\},
\end{array}		
	\end{equation}
	where{$\underline{\gamma} \in (0,1)$} and $\beta > 1$ are given parameters and $(r_b^0, r_c^0)$ and $\mu_0$ are evaluated at the starting point $(z^0, \lambda^0, s^0)$. Note that we must have
	\begin{equation*} \label{eq:6.3}
		\|r_b^0\| \leq \beta\mu_0 \; \mbox{ and }\; \|r_c^0\| \leq \beta\mu_0
	\end{equation*}
	to ensure that the initial point $(z^0, \lambda^0, s^0)$ belongs to $  {\mathcal{N}_{-\infty}(\underline{\gamma},\beta)}$.

    \begin{algorithm}
\caption{Interior Point Algorithm (IPM)}
\label{alg:ipm}
\begin{algorithmic}[1]
\REQUIRE $\gamma \in (0,1)$, $\beta \geq 1$,  $\sigma_{\min}< \sigma_{\max} <1$.
\STATE Choose $(z^0, \lambda^0, s^0)$ with $(z^0, s^0) > 0$.
\FOR{$k = 0, 1, 2, \ldots$}
    \STATE Choose $\sigma_k \in [\sigma_{\min}, \sigma_{\max}]$ and solve
    \begin{equation} \label{eq:6.4}
        \begin{bmatrix}
            \hat{Q} & \hat{A}^\top & -\1 \\
            A & 0 & 0 \\
            S^k & 0 & Z^k
        \end{bmatrix}
        \begin{bmatrix}
            \Delta z \\
            \Delta \lambda \\
            \Delta s
        \end{bmatrix}
        =
        \begin{bmatrix}
            -r_c^k \\
            -r_b^k \\
            -Z^k S^k e + \sigma_k \mu_k e
        \end{bmatrix}.
    \end{equation}
    \STATE Choose $\alpha_k$ as the largest value of $\alpha \in [0,1]$ such that
    \begin{align}
        (z^k(\alpha), \lambda^k(\alpha), s^k(\alpha)) \in{\mathcal{N}_{-\infty}(\underline{\gamma},\beta)}
        \label{eq:6.5c}
    \end{align}
    and the Armijo condition holds:
    \begin{equation} \label{eq:6.6}
        \mu(\alpha) \leq (1 - (1-\sigma_{max}) \alpha) \mu_k.
    \end{equation}
    \STATE Set $(z^{k+1}, \lambda^{k+1}, s^{k+1}) = (z^k(\alpha_k), \lambda^k(\alpha_k), s^k(\alpha_k))$.
\ENDFOR
\end{algorithmic}
\end{algorithm}

\begin{lemma} \label{lemma:reductions_in_IPM}
  It holds that:
	\begin{equation*}
    \begin{array}{rll}
      {A}z(\alpha)-b & = & (1-\alpha)r_b, \\[1.5ex]
	\hat{Q}z(\alpha) + \hat{A}^\top\lambda(\alpha)-s(\alpha)+\hat{c} & = & (1-\alpha)r_c, \\[1.5ex]
    (z+\alpha \Delta z)^\top(s+\alpha \Delta s) & = & z^\top s(1+\alpha(\sigma-1))+\alpha^2 (\Delta z)^\top\Delta s, \\[1.5ex]
  s_i \Delta z_i  + z_i \Delta s_i & = & \sigma \frac{z^\top s}{n+m} - z_is_i.
    \end{array}
	\end{equation*}

\end{lemma}
\begin{proof}
We prove the statement only for $r_c$, the proof for $r_b$ is analogous. We have,
    \begin{equation*}
			\begin{split}
				& \hat{Q}(z+\alpha \Delta z) + \hat{A}^\top(\lambda +\alpha \Delta \lambda)-(s+\alpha \Delta s)+\hat{c} \\
				&=  \hat{Q}z + \hat{A}^\top\lambda -s+\hat{c}+ \alpha(\hat{Q} \Delta z + \hat{A}^\top\Delta - \Delta s) = (1-\alpha)r_c .
			\end{split}
		\end{equation*}
The last part of the statement follows by direct inspection of the last block equation in \eqref{eq:6.4}. In particular, 
$
		 s^\top\Delta z+z^\top\Delta s
		 =- z^\top s + \sigma (n+m) \mu 
		 =  (\sigma-1)z^\top s.
$
\end{proof}

\begin{remark}[  {Convergence and computational complexity per iteration}]
\label{rem:ipm-convergence-complexity}
For all points in the neighborhood $  {\mathcal{N}_{-\infty}(\underline{\gamma},\beta)}$, the infeasibility is  bounded by some multiple of the duality measure $\mu$. By forcing $\mu$ to zero monotonically -- see Algorithm \ref{alg:ipm} Line \ref{eq:6.6} -- and restricting all iterates $(z^k, \lambda^k, s^k)$ to the neighborhood $  {\mathcal{N}_{-\infty}(\underline{\gamma},\beta)}$, the interior point algorithm presented in Algorithm \ref{alg:ipm} ensures primal-dual convergence, i.e., that $r_b^k \to 0$ and $r_c^k \to 0$ as $k \to \infty$. The other condition in \eqref{eq:6.2} (i.e.,  $z_i s_i \geq \underline{\gamma}\mu$) keeps the pairwise products $z_i s_i$ roughly in balance and prevents the Newton-like search directions from being distorted by components of $(z, s)$ that approach zero prematurely.

From a computational point of view, the dominant cost of each iteration is the
solution of the primal--dual Newton system \eqref{eq:6.4}. For a direct linear
algebra implementation, this cost is governed by the factorization of the
corresponding saddle-point/KKT matrix. In the worst case, if the system is
treated as a dense linear system of dimension equal to the total number of
primal-dual variables, the cost is cubic in that dimension. In
practice, however, interior point solvers exploit the symmetry, sparsity, and
block structure of the Newton system, often reducing the computational burden
substantially. Thus, the per-iteration cost should be understood as the cost of
forming and solving the structured Newton system, rather than as the cost of the
line-search procedure. And indeed, the Armijo condition in \eqref{eq:6.6} should not be interpreted as introducing
an additional optimization loop comparable to the Newton step itself. Its role is
only to select an admissible stepsize along the already computed Newton
direction. This requires inexpensive scalar operations, such as checking
positivity, membership in the neighborhood, and the decrease of the duality
measure. These operations are negligible compared with the solution of
\eqref{eq:6.4}. Moreover, in many practical implementations of primal--dual
interior point methods, including the implementation considered here, see Section \ref{sec:Numerical experiments}, the
Armijo backtracking procedure is not explicitly used. Instead, the stepsize is
chosen by a fraction-to-the-boundary rule, which takes a fixed fraction of the
largest step preserving positivity of the primal-dual variables. This
choice retains the main purpose of the line-search safeguard while avoiding an
inner backtracking loop.
\end{remark}

\subsection{Convergence analysis: the case $A_I=0$.} \label{sec:convergence_decoupled}
In this section, we conduct the convergence and complexity analysis of Algorithm \ref{alg:ipm} for the particular case $A_I=0$ (cf. Remark~\ref{rem:AI=0}). While this case does not fully capture the generality of the original minmax problem \eqref{eq:MinMax}, it remains an important version of such problem for several reasons. Firstly, it provides valuable insights as a motivational example. Secondly, it establishes that stationary points of {\textit{minmax problems with decoupled constraints}}   -- that is, problems of the form \eqref{eq:MinMax} where $Y(x) \equiv Y$ -- can be computed in polynomial time. This result is further corroborated by the equivalent formulation \eqref{eq:QPmodel} of problem \eqref{eq:MinMax}, which reduces to a convex programming problem when $A_I=0$, and was also preliminarily observed in \cite[Section 2]{MR4654111} in the strongly convex-concave case. The analysis carried out in this section will generalize such an observation without requiring strong convexity-concavity. It is important to note, moreover, that such an extension is particularly relevant as it postulates the existence of polynomial algorithms for the resolution of problems of a rich class of problems
 
\begin{align*} 
\min_{x\in X}  \; & \max_{{ 0 \leq y \leq p }} \;\; y^\top x
\end{align*}
(with $X$ given in \eqref{eq:XandY}), 
which can be used, for example, to model applications where the objective function of a linear program is subject to uncertainty.  {Finally, we observe that the polynomial solvability of this kind of problems is well understood by the robust optimization community, see \cite{MR2546839}, but usually overlooked by the min-max one.}
 
\begin{lemma} \label{lme:crucial_lemma}
	If  $A_I=0$, then $\bar{z}^\top\bar{s} \geq 0$ for any point $(\bar{z},\bar{\lambda},\bar{s})$ such that
	\begin{equation*} \label{eq:homogneus}
		A\bar{z}=0 \; \hbox{ and }\; \hat{Q}\bar{z} +\hat{A}^\top \bar{\lambda} - \bar{s}=0.
	\end{equation*}
\end{lemma} 
\begin{proof}
	We have
	$
		\bar{z}^\top\bar{s}= \bar{z}^\top\hat{Q}\bar{z}  +\bar{z}^\top\hat{A}^\top\bar{\lambda}. 
	$
Since $A_I=0$, equality $A\bar{z} =0$ implies $\hat{A}z=0$; hence, $\bar{z}^\top\hat{A}^\top\bar{\lambda} =0$. The result follows by observing that 
\[
\bar{z}^\top\hat{Q}\bar{z}  = \frac{1}{2}\bar{z}^\top(\hat{Q}+\hat{Q}^\top)\bar{z} \geq 0.
\]
\end{proof}

It is important to note that Lemma \ref{lme:crucial_lemma} also guarantees that for every $X,S>0$ the IPM matrices 
\[
\begin{bmatrix}
\hat{Q} & \hat{A}^\top & -\1 \\
A & 0 &  0\\
S & 0 & Z
\end{bmatrix}
\]
are invertible. The proof of such a statement can be easily obtained by observing that Lemma \ref{lme:crucial_lemma} implies that the matrix 
$\begin{bmatrix}
\hat{Q} & \hat{A}^\top & -\1 \\
A & 0 &  0
\end{bmatrix}$
has the \textit{the mixed-$P_0$ property} by an application of \cite[Lemma 2]{MR1403309}. Moreover, {using Assumption \ref{hyp:non_singularity} for $(z^*,\lambda^*, s^*)$, Lemma \ref{lme:crucial_lemma}, the \textit{mixed-$P_0$} property in the interior,} it can easily be shown that the central path, i.e., the solutions of \eqref{eq:Central_Path_Equation},  exists for every $\mu \in [ 0, +\infty{)}$. 

{To end this section, in the following Theorem \ref{th:polynomial_conv}, we} summarizes convergence and complexity guarantees for Algorithm  \ref{alg:ipm} {in the case $A_I=0$}. 

\begin{theorem} \label{th:polynomial_conv}
	Suppose that $\varepsilon>0$ is given and let $(z^0,\lambda^0,s^0) = (\zeta e, 0, \zeta e)$ be such that  $\zeta$ satisfies
    \[ 
    \|(z^*,s^*)  \| \leq \zeta \; \hbox{ and } \;  \zeta^2\leq \frac{C}{\varepsilon^\kappa}
    \]
	for suitable positive constants $C, \kappa$. There exists $K \in \mathbb{N}$ with 
	\begin{equation*}
		K = O((n+m)^2|\log(\varepsilon)| )
	\end{equation*}
	such that the iterates $\{(z^k,\lambda^k,s^k)\}_{k \in \mathbb{N}}$, generated by Algorithm \ref{alg:ipm}, {for solving the system \eqref{eq:F_def} when $A_I=0$}, satisfy
	\begin{equation*}
	 \mu^k \leq \varepsilon \hbox{ for all } k \geq K.
	\end{equation*}
\end{theorem}
An outline of the proof is reported in Section \ref{sec:appendix_AI=0}.
\begin{remark} It is important to note that by definition of $  {\mathcal{N}_{-\infty}(\underline{\gamma},\beta)}$, see  \eqref{eq:6.2}, if $\mu^k \leq \varepsilon$, then   $\|(r_c, r_b)\| \in O(\varepsilon)$, i.e., $(z,\lambda,s)$ is an $\varepsilon$-approximate solution of the stationary conditions \eqref{eq:KKT_for_IPM}.
	\end{remark}


\subsection{Convergence analysis: the case $A_I \neq 0$.} \label{sec:coupled_local_convergence}
In this section, we  analyze the local convergence of Algorithm \eqref{alg:ipm} when $A_I\neq 0 $. We emphasize that, as outlined in Remark \ref{rem:AI=0}, problem \eqref{eq:MinMax} is nonconvex in this case. Hence, polynomial complexity can not be expected in this more general case. To prove such a local convergence result, we will use a slightly different neighborhood of the infeasible central path w.r.t. the one used in the previous section, i.e., we consider

\begin{equation}\label{eqn:neighborhood}
	\begin{split}
	& \mathcal{N}(\bar {\gamma},\underline{\gamma},\gamma_p,\gamma_d):=  \left\{(z,\lambda,s) \in  {\mathbb{R}^{n+m}_{>0} \times \mathbb{R}^{p+q} \times \mathbb{R}^{n+m}_{>0} }\;\right| \\[1.5ex]
	&{\bar{\gamma} \mu  \geq z_is_i \geq \underline{\gamma} \mu}\; \hbox{ for } \; i=1,\dots,n+m, \;  \left.z^\top s \geq \gamma_p \|r_b\|, \;\; z^\top s \geq  \gamma_d\|r_c\|\right\}, 
	\end{split}
\end{equation}
where $\bar{\gamma} > 1 >\underline{\gamma}> 0$ and $(\gamma_p,\gamma_d)>0$.
Therefore, condition \eqref{eq:6.5c} in Algorithm \ref{alg:ipm} has to be replaced with
\begin{align}
        (z^k(\alpha), \lambda^k(\alpha), s^k(\alpha)) \in \mathcal{N}(\bar {\gamma},\underline{\gamma},\gamma_p,\gamma_d).
        \label{eq:6.5c_Bis}
    \end{align}

Before continuing, we would like to remind the reader that the constants $\delta^*$ and $\mu^*$ used in the remainder of this section are defined in \eqref{eq:delta_star}{and ensure, thanks to Theorem \ref{Th:IFT}, that the Newton system remains  in a suitable neighborhood of a solution of \ref{eq:F_def}.} The {main idea} of the proof is based on the observation that when an initial point is chosen from 
\[
N^*:=B(v^*,\delta^*)\cap \mathcal{N}(\bar {\gamma},\underline{\gamma},\gamma_p,\gamma_d),
\]
all the other iterates of Algorithm \ref{alg:ipm} remain {in this set}. The convergence will essentially follow by observing that in $N^*$, the IPM matrices and the right-hand side in \eqref{eq:6.4} are uniformly bounded {, see proof of Corollary \ref{rem:uniform_boundedness_rhs},  and by the definition $\mathcal{N}(\bar {\gamma},\underline{\gamma},\gamma_p,\gamma_d)$, see proof Theorem \ref{th:convergence}. } All the{missing} proofs of the results presented in this section can be found in Section \ref{sec:appendix_AI_neq0}.

To establish the local convergence of Algorithm \ref{alg:ipm}, we need the following preliminary results.

\begin{lemma}[see Lemma 2 in \cite{MR2411396}]\label{lem:original-lemma3}
For all $z \in B(v^*, \delta^*)$ and for all $\mu^+ \in [0, \mu^*)$, the full-Newton iterate defined by
\begin{equation}\label{eq:newton-iterate}
v^+ = v - D\hat{F}(v)^{-1}(\hat{F}(v) - \mu^+ \tilde{e})
\end{equation}
  {is well defined and satisfies}
\begin{equation*}
\label{eq:bound-central-path}
\|v^+ - g(\mu^+)\| \leq \frac{KL}{2}\|v - g(\mu^+)\|^2
\end{equation*}
and
\begin{equation}
\label{eq:bound-solution}
\|v^+ - v^*\| \leq \frac{1}{2}\|v - v^*\| + \frac{\mu^+}{\mu^*}\frac{\delta^*}{2}.
\end{equation}
In particular, $v^+ \in B(v^*, \delta^*)$.
\end{lemma}

\begin{lemma}[Damped Newton step]
\label{lem:damped-lemma3}
For all $v \in B(v^*, \delta^*)$ and all $\mu^+ \in [0, \mu^*)$, let $v^+$ be the Newton iterate given by \eqref{eq:newton-iterate}. For $\alpha \in [0,1]$, define the damped iterate
\begin{equation*}
v_\alpha = v + \alpha(v^+ - v).
\end{equation*}
Then the following bounds hold:
\begin{equation*}
\label{eq:damped-central-path}
\|v_\alpha - g(\mu^+)\| \leq (1-\alpha)\|v - g(\mu^+)\| + \alpha\frac{KL}{2}\|v - g(\mu^+)\|^2
\end{equation*}
and
\begin{equation}
\label{eq:damped-solution}
\|v_\alpha - v^*\| \leq \left(1 - \frac{\alpha}{2}\right)\|v - v^*\| + \alpha\frac{\mu^+}{\mu^*}\frac{\delta^*}{2}.
\end{equation}
In particular, $v_\alpha \in B(v^*, \delta^*)$.
\end{lemma}

\begin{theorem} \label{th:IPM_well_definitness}
Suppose that $(z^k,\lambda^k,s^k) \in  \mathcal{N}(\bar {\gamma},\underline{\gamma},\gamma_p,\gamma_d)$ is a point satisfying  $(z^k)^\top s^k>0$. Moreover, define 
\begin{equation*}
\begin{array}{rll}
\alpha_p^{*,k} & := & \sup \left\{\alpha \in{[0,1]}\;|\quad  z^k(\alpha) \geq 0 \right\},\\[1.5ex]
\alpha_d^{*,k} & := & \sup \left\{\alpha \in{[0,1]}\;|\quad  {s}^k(\alpha) \geq 0 \right\},\\[1.5ex]
\alpha^{*,k} & := & \min \left\{ \alpha_p^{*,k},\;\;\alpha_d^{*,k}\right\}.
\end{array}
\end{equation*}
If  $(z(\alpha^{*,k}), \,\lambda(\alpha^{*,k}), \,s(\alpha^{*,k}))$ is not a solution of problem \eqref{eq:F_def}, then there exists a scalar $\hat{\alpha}^k$ with $0<\hat{\alpha}^k <\alpha^{*,k}$ such that \eqref{eq:6.5c_Bis} and \eqref{eq:6.6} hold for all $\alpha \in [0,\hat{\alpha}^k]$.   
\end{theorem}

\begin{proof}
In this proof, we omit the IPM iterate counter $k$; i.e., $(z^k,\lambda^k,s^k)\equiv (z, \lambda,s)$. Let us define the following functions, for all $i=1,\dots,n$:
    \[
    \begin{array}{rll}
    f_i(\alpha) &:= & (z_i+\alpha \Delta z_i)(s_i+\alpha \Delta s_i)- \underline{\gamma}(z+\alpha \Delta z)^\top(s+\alpha \Delta s)/(n+m), \\[1.5ex]
		\bar{f}_i(\alpha) &:= & \bar{\gamma}(z+\alpha \Delta z)^\top(s+\alpha \Delta s)/(n+m)-(z_i+\alpha \Delta z_i)(s_i+\alpha \Delta s_i),\\[1.5ex] 
		h(\alpha)&:= &(1 -(1-\sigma_{max})\alpha) z^\top s - (z+\alpha \Delta z)^\top(s+\alpha \Delta s), \\[1.5ex]
		g_d(\alpha)&:= & (z+\alpha \Delta z)^\top(s+\alpha \Delta s) \\[1.5ex]
              & & \qquad - \gamma_d \|\hat{Q}(z+\alpha \Delta z)- \hat{A}^\top(\lambda+\alpha \Delta \lambda) +\hat{c} -(s+\alpha \Delta s)\|, \\[1.5ex]
		g_p(\alpha)&:= & (z+\alpha \Delta z)^\top(s+\alpha \Delta s) - \gamma_p \| A(z+\alpha \Delta z) -b \|.
    \end{array}
    \]
Using Lemma \ref{lemma:reductions_in_IPM} in the expressions of $g_d(\alpha)$  and $g_p(\alpha)$, we have
\[
\begin{array}{rll}
 g_d(\alpha) &=&(z+\alpha \Delta z)^\top(s+\alpha \Delta s) -\gamma_d (1-\alpha) \|r_c\|,\\[1.5ex]
 g_p(\alpha) &\geq & (z+\alpha \Delta z)^\top(s+\alpha \Delta s) -   \gamma_p (1-\alpha) \|r_b\|.
\end{array}
\]

We start by proving that there exists $\hat{\alpha}^{k}>0$ such that
\begin{equation*}
	f_i(\alpha)\geq 0, \quad \bar{f}_i(\alpha)\geq 0, \quad h(\alpha)\geq 0, \quad g_p(\alpha)\geq 0, \quad g_d(\alpha)\geq 0
\end{equation*}
for all $i = 1, \dots, n+m$ and for all $\alpha \in [0, \hat{\alpha}^k]$. In the following, we will extensively use the identities in Lemma \ref{lemma:reductions_in_IPM}. We have
\begin{equation} \label{eq:ineq_1}
\begin{array}{rll}
  f_{i}(\alpha) & = &\underbrace{(1-\alpha)\left(z_is_i-\underline{\gamma} \frac{z^\top s}{n+m}\right)}_{\geq 0}\\[2.2ex]
                &   & + \;\alpha^2\left(\Delta z_i \Delta s_i - \underline{\gamma}\frac{(\Delta z)^\top \Delta s}{n+m}\right) + \alpha \sigma \left(1-\underline{\gamma}\right)\frac{z^\top s}{n+m}, \\[1.5ex]
		& \geq & \quad\, \alpha^2\left(\Delta z_i \Delta s_i - \underline{\gamma}\frac{(\Delta z)^\top \Delta s}{n+m}\right) + \alpha \sigma (1-\underline{\gamma})\frac{z^\top s}{n+m}.
\end{array}
\end{equation}
Since $z^\top s>0$, using a simple continuity argument, we can infer the existence of a small enough 
$\underline{f}_i>0$ such that $ f_{i}(\alpha)\geq 0$ for all $\alpha \in [0, \underline{f}_i]$. Reasoning analogously,
\begin{equation} \label{eq:ineq_2}
\begin{array}{rll}
    \bar{f}_{i}(\alpha) & = & \underbrace{(1-\alpha)\left(\bar{\gamma} \frac{z^\top s}{n+m} -z_is_i\right)}_{\geq 0}\\[2.2ex]
     & & +\;\alpha^2\left(\bar{\gamma}\frac{(\Delta z)^\top \Delta s}{n+m}-\Delta z_i \Delta s_i\right) + \alpha \sigma (\bar{\gamma}-1)\frac{z^\top s}{n+m}, \\[1.5ex]
		& \geq & \quad\, \alpha^2\left(\bar{\gamma}\frac{(\Delta z)^\top \Delta s}{n+m}-\Delta z_i \Delta s_i\right) + \alpha \sigma (\bar{\gamma}-1)\frac{z^\top s}{n+m},
\end{array}
\end{equation}
and hence, there exists a small enough 
$\bar{f}_i>0$ such that $\bar{f}_{i}(\alpha)\geq 0$ for all $\alpha \in [0, \bar{f}_i]$.

Concerning $h(\alpha)$, we have
\begin{equation} \label{eq:ineq_3}
	h(\alpha)= z^\top s(\sigma_{max}-\sigma) \alpha- \alpha^2 (\Delta z)^\top\Delta s,
\end{equation}
and since $z^\top s(\sigma_{max}-\sigma)>0$, there exists $\hat{h}>0$ small enough such that we have $h(\alpha)\geq 0$ for all $\alpha \in [0, \hat{h}]$.

Concerning $g_d(\alpha)$, we have
\begin{equation} \label{eq:ineq_4}
\begin{array}{rll}
  g_d(\alpha) & = & \underbrace{(1-\alpha)(z^\top s- \gamma_d \|r_c\|)}_{\geq 0} + \alpha \sigma z^\top s +\alpha^2 (\Delta z)^\top\Delta s \\[2.5ex]
		& \geq & \alpha \sigma z^\top s +\alpha^2 (\Delta z)^\top\Delta s,
\end{array}
\end{equation}
and hence, there exists $\hat{g}_d>0$ small enough such that $g_d(\alpha)\geq 0$ for all $\alpha \in [0, \hat{g}_d]$.

Finally, concerning $g_p(\alpha)$, we have
\begin{equation} \label{eq:ineq_5}
\begin{array}{rll}
   g_p(\alpha) & \geq & \underbrace{(1-\alpha)(z^\top s- \gamma_p \|r_b\|)}_{\geq 0} + \alpha \sigma z^\top s +\alpha^2 (\Delta z)^\top\Delta s \\[2.5ex]
		& \geq & \alpha \sigma z^\top s +\alpha^2 (\Delta z)^\top\Delta s,
\end{array}
\end{equation}
and hence there exists $\hat{g}_p>0$ small enough such that $g_p(\alpha)\geq 0$ for all $\alpha \in [0, \hat{g}_p]$. 

{Let us define
\begin{equation*}
	\hat{\alpha}^k = \min \left\{ \min_{i} \underline{f}_i, \;\; \min_{i} \bar{f}_i,  \;\; \hat{h}, \;\; \hat{g}_d, \;\; \hat{g}_p, \; 1\right\} >0.
\end{equation*}
{To prove the thesis}, it remains that
\begin{equation*}
(z(\alpha),\lambda(\alpha),s(\alpha) ) \in \mathbb{R}^n_{>0} \times \mathbb{R}^m \times \mathbb{R}^n_{>0}	\hbox{ for all } \alpha \in [0,\hat{\alpha}^k].
\end{equation*}
{To show the last part of the thesis}, it is enough to prove that $\alpha^{*,k}>\hat{\alpha}^k$. To this aim, let us suppose by contradiction that $\alpha^{*,k} \leq \hat{\alpha}^k$. By  definition of $\alpha^{*,k}$, there exists $\bar{\ell} \in \{1, \dots,n+m\}$ such that $(z_{\bar{\ell}}+\alpha^{*,k} \Delta z_{\bar{\ell}})(s_{\bar{\ell}}+\alpha^{*,k} \Delta s_{\bar{\ell}})=0$. Hence,
\begin{equation*}
	f_{\bar{\ell}}(\alpha^{*,k})= - \underline{\gamma}(z\big(\alpha^{*,k})\big)^\top \frac{s(\alpha^{*,k})}{n+m} \geq 0 \;\; \Longrightarrow \;\; (z\big(\alpha^{*,k})\big)^\top s(\alpha^{*,k}) =0. 
\end{equation*}
From this implication,  using $g_d(\alpha^{*,k})$ and $g_p(\alpha^{*,k})$, we obtain that
\begin{equation*}
\begin{array}{rll}
     Az(\alpha^{*,k})  - b & = & 0, \\[1.5ex]
	\hat{Q}z(\alpha^{*,k}) +\hat{A}^\top\lambda(\alpha^{*,k}) - s(\alpha^{*,k}) +\hat{c}& = &0, 
\end{array}
\end{equation*}
i.e.,\; $\left(z(\alpha^{*,k}), \; \lambda(\alpha^{*,k}), \; s(\alpha^{*,k})\right)$ is a solution of problem \eqref{eq:F_def}.  
}
\end{proof}

The next two results establish that Algorithm \ref{alg:ipm}  converges to a solution of problem \eqref{eq:F_def}. This is done by establishing that the right-hand sides of the Newton systems are uniformly bounded, see Corollary \ref{rem:uniform_boundedness_rhs}, and by showing that the complementarity product $(z^k)^\top s^k$ cannot be bounded away from zero.

\begin{corollary}\label{rem:uniform_boundedness_rhs}
	The right-hand sides of the Newton systems \eqref{eq:6.4}  are uniformly bounded. 
\end{corollary}
\begin{proof}
   {The proof of Corollary \ref{rem:uniform_boundedness_rhs} is somehow standard, see Section \ref{sec:appendix_AI_neq0} for more details}.
\end{proof}

\begin{theorem} \label{th:convergence} Suppose that $(z^0,\lambda^0,s^0) \in B(v^*, \delta^*) \cap \mathcal{N}(\bar {\gamma},\underline{\gamma},\gamma_p,\gamma_d)   $. Then
 Algorithm \ref{alg:ipm}{for solving the system \eqref{eq:F_def} when $A_I\neq0$,} produces a sequence of iterates in $B(v^*, \delta^*) \cap \mathcal{N}(\bar {\gamma},\underline{\gamma},\gamma_p,\gamma_d) $ such that we have $\lim  (z^k)^\top s^k=0$, i.e.,{using the definition of $\mathcal{N}(\bar {\gamma},\underline{\gamma},\gamma_p,\gamma_d)$, } $(z^k,\lambda^k,s^k) $ converges to a solution of problem \eqref{eq:F_def}. 
\end{theorem}

\begin{proof}
From Lemma \ref{lem:damped-lemma3} and Theorem \ref{th:IPM_well_definitness}, we can assume that {for all} $k\in \mathbb{N}$, we have  $(z^k,\lambda^k,s^k) \in B(v^*, \delta^*) \cap \mathcal{N}(\bar {\gamma},\underline{\gamma},\gamma_p,\gamma_d) $. 
 Let us argue by contradiction supposing that there exists $\varepsilon^*>0$ such that we have $(z^k)^\top s^k> \varepsilon^*$ for all $k \in \mathbb{N}$.\\
 
 \textbf{Claim 1} There exists a constant $C_1$  such that
 \[
 \left\|  \left[\Delta z^k,\; \Delta \lambda^k, \; \Delta s^k \right]^\top\right\|\leq C_1 \; \hbox{ for all } \; k \in \mathbb{N}.
 \]
 The proof of this fact follows observing that in $B(v^*, \delta^*)$, the Newton matrices have a uniformly bounded inverse, and that the right-hand sides of \eqref{eq:6.4} are uniformly bounded; see Lemma \ref{lem:uniform_boundedness} and Corollary~\ref{rem:uniform_boundedness_rhs}. As a consequence, there exists another constant $C_2$ such that
 \begin{equation} \label{eq:uniform_bound_deltas}
 \begin{array}{rll}
\left|\Delta z_i^k \Delta s_i^k - \underline{\gamma}\frac{(\Delta z^k)^\top \Delta s^k}{n+m}\right| & \leq & C_2, \\[1.5ex]
\left |\bar{\gamma}\frac{(\Delta z^k)^\top \Delta s^k}{n+m}-\Delta z_i^k \Delta s_i^k \right| & \leq & C_2, \\[1.5ex]
 		\left|(\Delta z^k)^\top\Delta s^k\right| &\leq& C_2.
 \end{array}
 \end{equation} 
 
 \textbf{Claim 2} There exists $\alpha^*>0$ such that $\alpha^k \geq \alpha^*$ for all $j \in \mathbb{N}$. \\
 
 Using \eqref{eq:uniform_bound_deltas} in equations \eqref{eq:ineq_1}, \eqref{eq:ineq_2}, \eqref{eq:ineq_3}, \eqref{eq:ineq_4}, and \eqref{eq:ineq_5} we have
 \[
 \begin{array}{rll}
 f_i(\alpha) & \geq &\; \alpha^2\left(\Delta z_i^k \Delta s_i^k - \underline{\gamma}\frac{(\Delta z^k)^\top \Delta s^k}{n+m}\right) + \alpha \sigma (1-\underline{\gamma})\frac{(z^k)^\top s^k}{n+m}, \\[1.5ex]
 	              &\geq &\; - C_2 \alpha^2 +\alpha \sigma (1-\underline{\gamma}) \frac{\varepsilon^*}{n+m},\\ [1.5ex]
 	\bar{f}_i(\alpha) & \geq &\; \left(\bar{\gamma}\frac{(\Delta z^k)^\top \Delta s^k}{n+m}-\Delta z_i^k \Delta s_i^k \right) + \alpha \sigma (\bar{\gamma}-1)\frac{(z^k)^\top s^k}{n+m}, \\[1.5ex]
 	             & \geq &\; - C_2 \alpha^2 +\alpha \sigma (\bar{\gamma}-1) \frac{\varepsilon^*}{n+m},\\ [1.5ex]
 	 h(\alpha) & = &\;(z^k)^\top s^k(\sigma_{max}-\sigma) \alpha- \alpha^2 (\Delta z^k)^\top\Delta s^k, \\[1.5ex]
 	         & \geq  &\; \varepsilon^*(\sigma_{max}-\sigma) \alpha - C_2 \alpha^2,\\[1.5ex]
 	 g_d(\alpha) & \geq  &\;  \alpha \sigma (z^k)^\top s^k +\alpha^2 (\Delta z^k)^\top\Delta s^k \;\;\,\geq\;\;\, \varepsilon^*\sigma \alpha - C_2 \alpha^2,  \\[1.5ex]    
 	 g_p(\alpha) & \geq  &\;   \alpha (\sigma - q \gamma_p ) (z^k)^\top s^k +\alpha^2 (\Delta z^k)^\top\Delta s^k \;\;\,\geq \;\;\,\varepsilon^*\alpha \sigma  - C_2 \alpha^2.
 \end{array}
 \]
 Hence, $\alpha^k\ge\alpha^*$ with
 \begin{equation*}\label{eqn:alpha_lower_bound}
\alpha^*:=\min\left\{{ 1},\; \frac{\sigma(1-\underline{\gamma})\varepsilon^*}{C_2(n+m)},\;\frac{\sigma(\bar{\gamma}-1)\varepsilon^*}{C_2(n+m)},\; \frac{(\sigma_{max}-\sigma)\varepsilon^*}{C_2}, \;\frac{\sigma \varepsilon^*}{C_2}\right\}.
 \end{equation*}
The convergence claim follows by observing that the inequality
\begin{equation*}
	\varepsilon^* \leq (z^k)^\top s^k \leq (1 -(1-\sigma_{max})\alpha^* )^k(z^0)^\top s^0
\end{equation*}
leads to a contradiction when $k \to \infty$.
\end{proof}

Before concluding this section,  we would like to emphasize that the above proof complements and expands \cite[Remark 3.1]{MR1242461} and \cite[Section 3]{MR4784765}.

\begin{figure}[htbp]
\begin{subfigure}[b]{\textwidth}
    \centering
    \includegraphics[width=\textwidth]{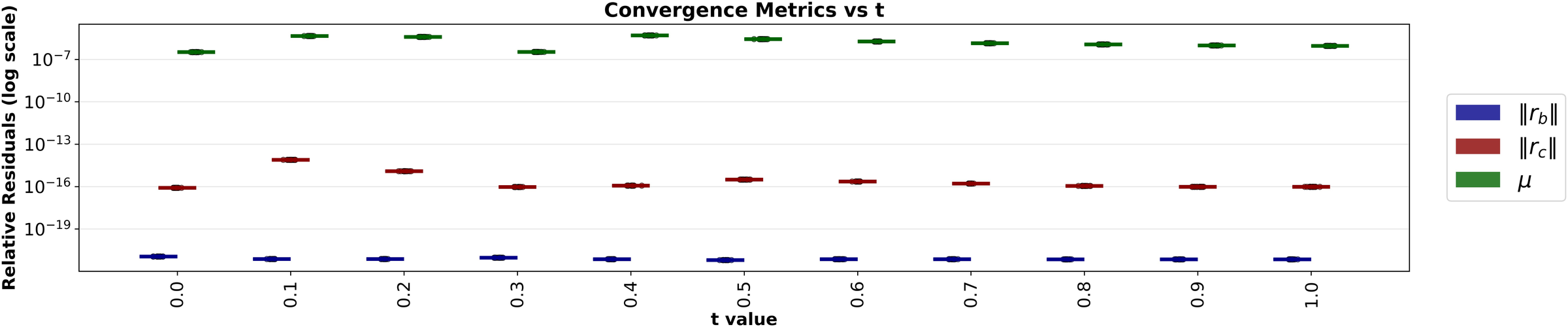} 
        \caption{$n=m=100, \; \; p=q=50$}
    \end{subfigure}
     \begin{subfigure}[b]{\textwidth}
        \centering
       \includegraphics[width=\textwidth]{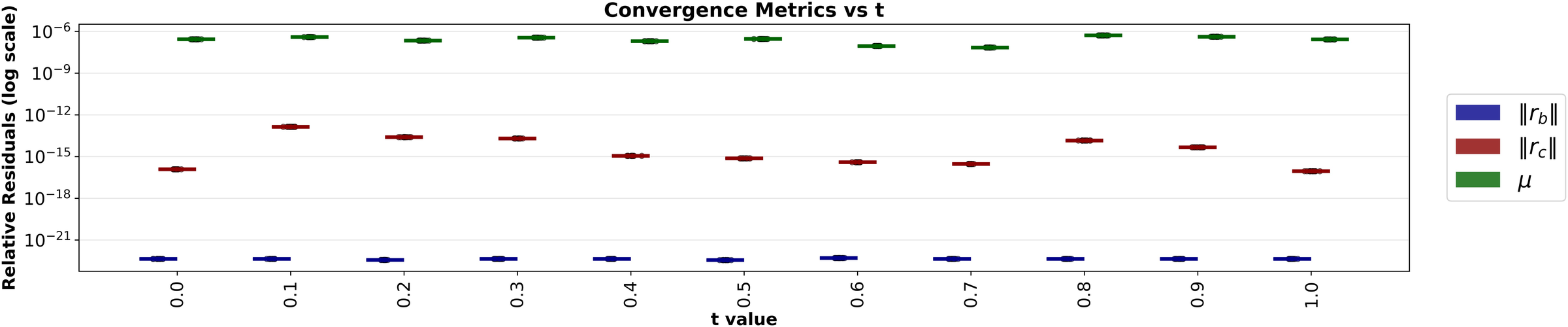} 
        \caption{$n=m=1000, \; \; p=q=500$}
    \end{subfigure}
    \caption{Convergence metrics versus problem scaling parameter $t$. Grouped boxplots (three per $t$ value) show primal feasibility residual $\|r_b\|$ (blue), dual feasibility residual $\|r_c\|$ (red), and complementarity measure $\mu$ (green) across $15$ independent trials for each $t \in [0, 1]$.} 
     \label{fig:convergence_metrics}
\end{figure}

\section{Numerical experiments.}\label{sec:Numerical experiments}
All the numerical results are obtained using  \texttt{Python 3} on a \texttt{MacBook Pro} using \texttt{Apple M3 Pro} CPU with \texttt{18 GB} of RAM.  The code is publicly available at \url{https://github.com/StefanoCipolla/MinMax}.
\subsection{Implementation Details.}
As  is customary in the IPM literature, the implementation usually deviates from the theoretical developments in order to gain computational efficiency.
In particular, we do not require the iterates of Algorithm \ref{alg:ipm} to lie in any of the neighborhoods previously defined, using, instead, a well-known Predictor-Corrector strategy \cite{MR1186163}. {Newton's directions are computed by solving \eqref{eq:6.4} by a well-known $2 \times 2$ block reduction, see, e.g. \cite{MR4865731}.} Moreover, following a standard practice in IPM literature, stopping criteria are chosen as

\begin{equation*}
\| \hat{Q}z + \hat{A}^{\top}\lambda +\hat{c} - {s}\|_{\infty} \leq R \cdot \mathrm{tol} 
\quad \land \quad 
\|Az-b\|_{1} \leq R \cdot \mathrm{tol} 
\quad \land \quad 
c_{z,s} \leq \mathrm{tol},
\end{equation*}
where
\begin{equation*}
\mathrm{tol} = 10^{-6}, \quad R := \max\{\|{Q}\|_{\infty},\|A\|_{\infty}, \|b\|_1, \|{c}\|_1\},
\end{equation*}
and
\begin{equation*}
c_{{z},{s}} := \max_{i} \left\{ \min\left\{ |(z_i s_i)|, |z_i|, |s_i| \right\} \right\}.
\end{equation*}

For the initial point, we adapt the technique from \cite{MR1186163}.

\subsection{Dataset.}
The random instances, {used for the experiments in Subsections \ref{Subsect:Decoupled Constraints Case} and \ref{Subsect: Coupled Constraints Case},} are generated by adapting the procedure used in \cite{MR4906232} for SDP problems.  In particular, such random instances are generated with varying dimensions $(n, m, p, q)$. {The vector} $(z^*, \lambda^*, s^*)$ is drawn uniformly from $[10^{-3}, 10^{3}]$ with strict complementarity enforced by setting $z$ to zero at $k = \min\{10, \lfloor(n+m)/5\rfloor\}$ random positions and ensuring $s$ is non-zero only at these positions. Blocks $Q_{11}, Q_{22}$ are constructed via spectral decomposition $UDU^\top$ with controlled eigenvalue spectra to ensure positive definiteness. Constraint matrices have standard Gaussian entries, with $A_I = 0$ for the case with decoupled constraints. The routines to generate the random instances are publicly available at \url{https://github.com/StefanoCipolla/MinMax}.

\begin{figure}[htbp]
\begin{subfigure}[b]{\textwidth}
    \centering
    \includegraphics[width=\textwidth]{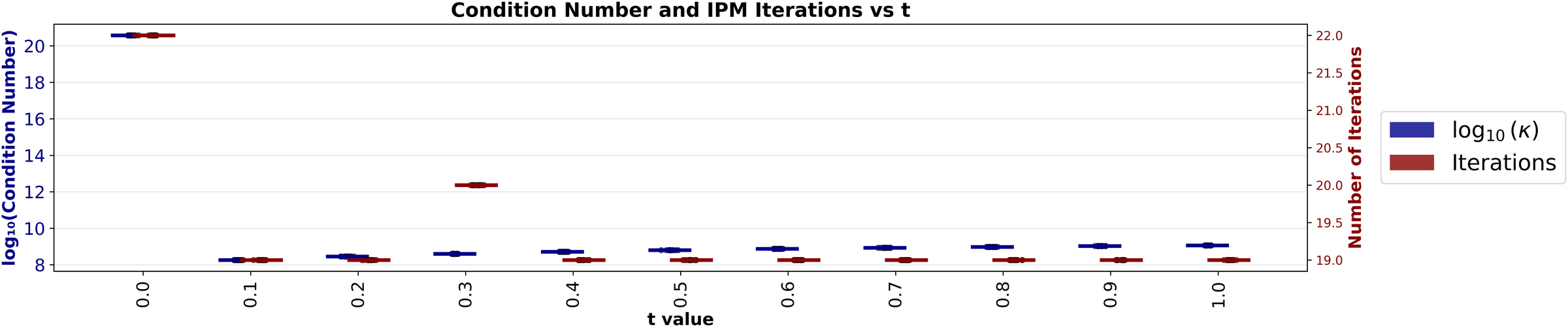} 
        \caption{$n=m=100, \; \; p=q=50$}
    \end{subfigure}
     \begin{subfigure}[b]{\textwidth}
        \centering
       \includegraphics[width=\textwidth]{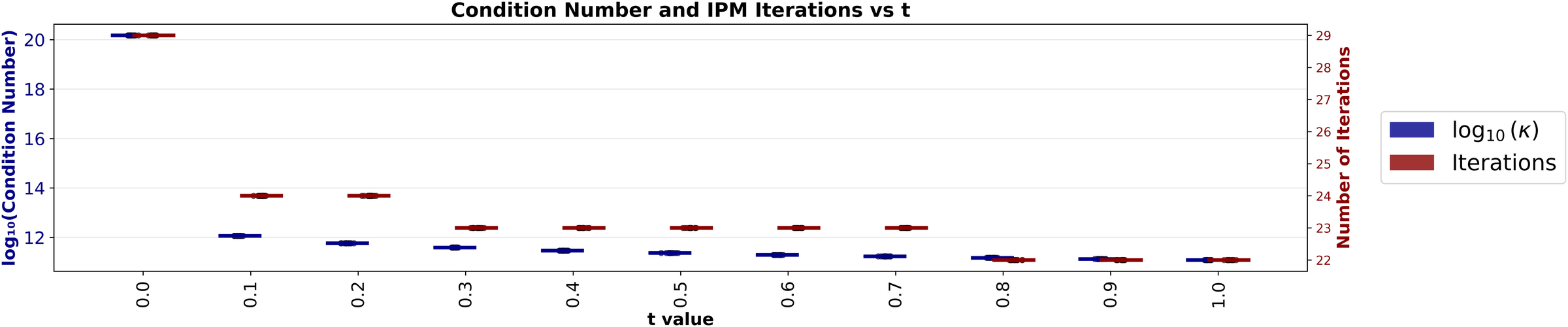} 
        \caption{$n=m=1000, \; \; p=q=500$}
    \end{subfigure}
   \caption{Stationarity matrix condition number and iteration count versus problem scaling parameter $t$. Grouped boxplots (two per $t$ value) display $\log_{10}(\kappa)$ (blue) and IPM iteration count (red) across $15$ trials. Despite the significant increase in condition number at $t = 0$ (loss of strong convexity-concavity), the iteration count remains stable at approximately $20$ iterations.}
    \label{fig:condition_iterations}
\end{figure}

\subsection{Decoupled constraints case.}\label{Subsect:Decoupled Constraints Case}

The objective of this section is to demonstrate the robustness of the proposed IPM with respect to the degradation of strong convexity and strong concavity properties in the decoupled case ($A_I = 0$). To systematically assess this robustness, we generate random instances of the problem data $(\hat{Q}, \hat{c}, b, A_O, B_I)$ and parametrize the diagonal blocks of the matrix $\hat{Q}$ as $tQ_{11}$ and $tQ_{22}$ for $t \in [0, 1]$. As $t$ approaches zero, the strong convexity of the maximization objective and the strong concavity of the minimization objective are progressively eliminated, allowing us to evaluate algorithmic performance as the problem structure degenerates from strongly convex-strongly concave to merely convex-concave.

Figure~\ref{fig:convergence_metrics} presents the convergence behavior of the IPM across this parametrized family of problems. The upper panel displays results for problem dimensions $n = m = 100$ and $p = q = 50$, while the lower panel presents results for the larger-scale setting with $n = m = 1000$ and $p = q = 500$. Similarly, Figure~\ref{fig:condition_iterations} illustrates the iteration count and numerical conditioning of the Jacobian of the stationarity system $D\hat{F}(z^*, \lambda^*, s^*)$ for both problem sizes, with the upper and lower panels corresponding to the small-scale and larger-scale instances, respectively. The numerical experiments reveal several important findings that underscore the robustness of the proposed method:

\paragraph{Convergence accuracy.}
As expected from a well-designed IPM, the magnitude of the relative residuals remains below the selected tolerance $\mathrm{tol=1e-6}$, see {Figure~\ref{fig:convergence_metrics}}, demonstrating that the algorithm consistently delivers solutions meeting the specified accuracy requirements regardless of the degree of strong convexity or strong concavity.

\begin{figure}[htbp]
\begin{subfigure}[b]{\textwidth}
    \centering
    \includegraphics[width=\textwidth]{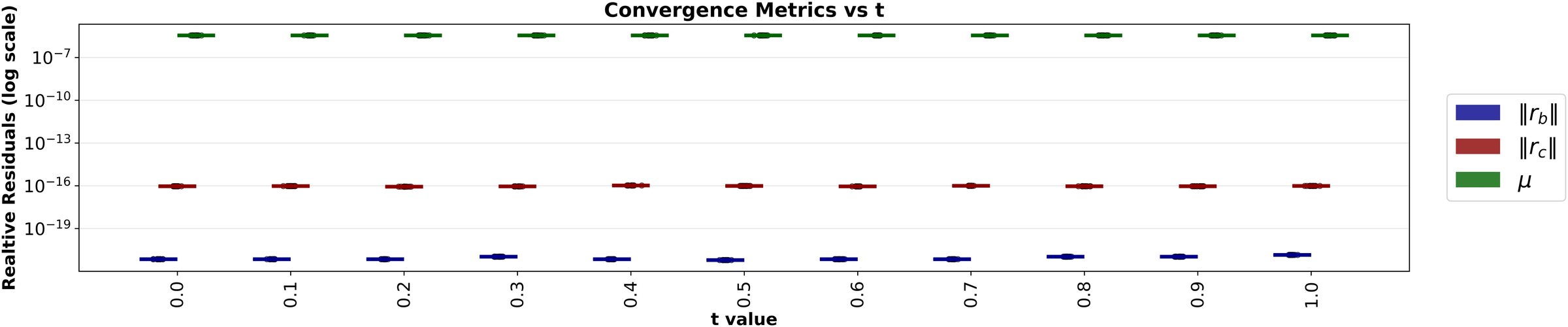} 
        \caption{$n=m=100, \; \; p=q=50$}
        \label{fig:1a_A_=0}
    \end{subfigure}
     \begin{subfigure}[b]{\textwidth}
        \centering
       \includegraphics[width=\textwidth]{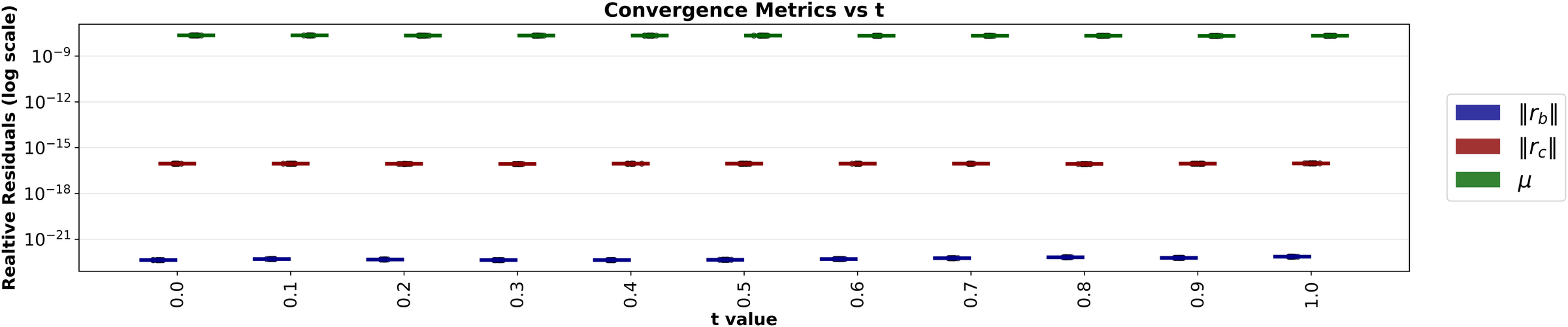} 
        \caption{$n=m=1000, \; \; p=q=500$}
        \label{fig:1a_B_=0}
    \end{subfigure}
    \caption{Convergence metrics versus constraints coupling parameter $t$. Grouped boxplots (three per $t$ value) show primal feasibility residual $\|r_b\|$ (blue), dual feasibility residual $\|r_c\|$ (red), and complementarity measure $\mu$ (green) across $15$ independent trials for each $t \in [0, 1]$.} 
     \label{fig:convergence_metrics_coupled}
\end{figure}

\paragraph{Iteration stability.}
Perhaps most remarkably, the iteration count exhibits striking consistency across the entire parameter range $t \in [0, 1]$. Despite the substantial degradation in problem conditioning as $t \to 0$ -- where the diagonal blocks $Q_{11}$ and $Q_{22}$ vanish entirely -- the number of iterations required for convergence shows only negligible variation. This stability persists even as the condition number $\kappa$ of the of the Jacobian $D\hat{F}(z^*,\lambda^*,s^*)$  increases significantly, as evidenced in Figure~\ref{fig:condition_iterations}. 
\paragraph{Dimension scalability.}
The comparison of the upper and lower panels of Figure~\ref{fig:condition_iterations} reveals the algorithm's favorable scaling properties. Increasing the problem dimension by one order of magnitude -- from $n = m = 100$ to $n = m = 1000$ -- results in virtually no increase in iteration count when $t \to 1$ (strongly convex-concave regime). Even in the most challenging scenario where $t \to 0$ (complete loss of strong convexity-concavity), the average iteration count increases by only approximately $5$ iterations. This mild growth in iteration counts w.r.t. the problem dimension provides strong empirical support for the polynomial complexity established theoretically in Subsection~\ref{sec:convergence_decoupled}.

\subsection{Coupled constraints case.}\label{Subsect: Coupled Constraints Case}
The objective of this section is to demonstrate the behavior of the proposed IPM framework in the \textit{coupled constraints} case. Analogous to the experiments in the previous section, we examine this behavior by \textit{calibrating the strength} of the coupling matrix $A_I$. Specifically, we consider $tA_I$ for $t \in [0,1]$, noting that the case $t=0$ recovers the decoupled constraints setting.

It is important to emphasize that the numerical experiments presented in this section employ initial points generated as random perturbations of the solution of each problem instance. We deliberately select the perturbation magnitude to be $0.3$, which is two orders of magnitude larger than the quantity $a$ defined in \eqref{eq:a} -- in our case, $a \in \mathcal{O}(10^{-3})$. As discussed previously, this quantity provides an estimate for the radius of local convergence of Algorithm~\ref{alg:ipm}. Our choice of the perturbation magnitude aims to demonstrate that this theoretical estimate may be overly conservative for certain problem classes, such as the randomly generated instances considered here with dimensions as in the previous section.

Figure~\ref{fig:convergence_metrics_coupled} presents the convergence behavior of the IPM across this parametrized family of problems. Similarly, Figure~\ref{fig:condition_iterations_coupled} illustrates the iteration count and numerical conditioning of the Jacobian of the stationarity system $D\hat{F}(z^*, \lambda^*, s^*)$, with the upper and lower panels corresponding to the small-scale and larger-scale instances, respectively. The numerical experiments presented here confirm the robustness, stability and scalability observed in the previous case, showcasing, as well, the theoretical local convergence results proved in Subsection \ref{sec:coupled_local_convergence}. {Indeed, Figure~\ref{fig:convergence_metrics_coupled} supports the local convergence result presented in Subsection \ref{sec:coupled_local_convergence}, as Algorithm \ref{alg:ipm} is able to deliver the required accuracy across all the $t$ when initialized sufficiently close to the solution. Figure~\ref{fig:condition_iterations_coupled} showcases, instead, how the number of IPM iterations remains roughly constant when the \textit{strength} of the coupling constraint term increases $(t \to 1)$ while the conditioning of the Jacobian of the stationarity system remains almost unchanged.} 
\begin{figure}[htbp]
\begin{subfigure}[b]{\textwidth}
    \centering
    \includegraphics[width=\textwidth]{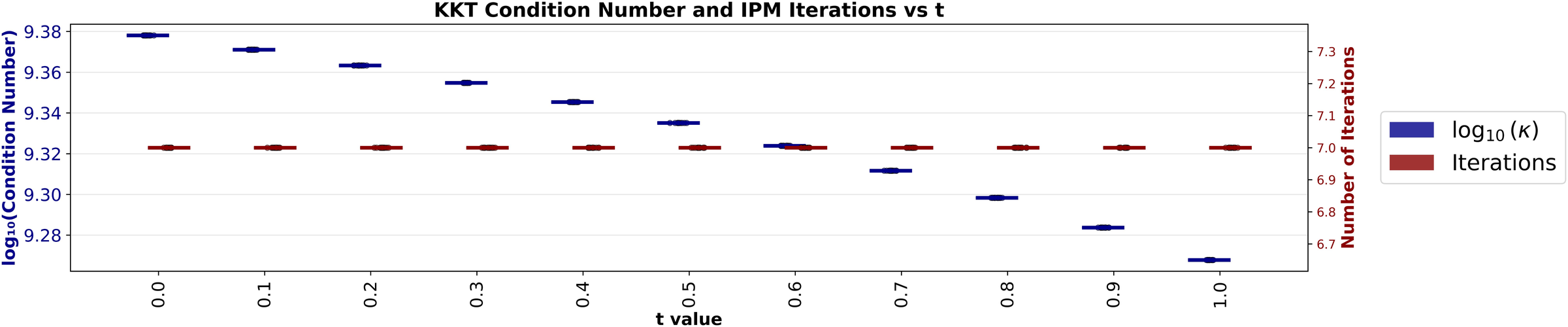} 
        \caption{$n=m=100, \; \; p=q=50$}
    \end{subfigure}
     \begin{subfigure}[b]{\textwidth}
        \centering
       \includegraphics[width=\textwidth]{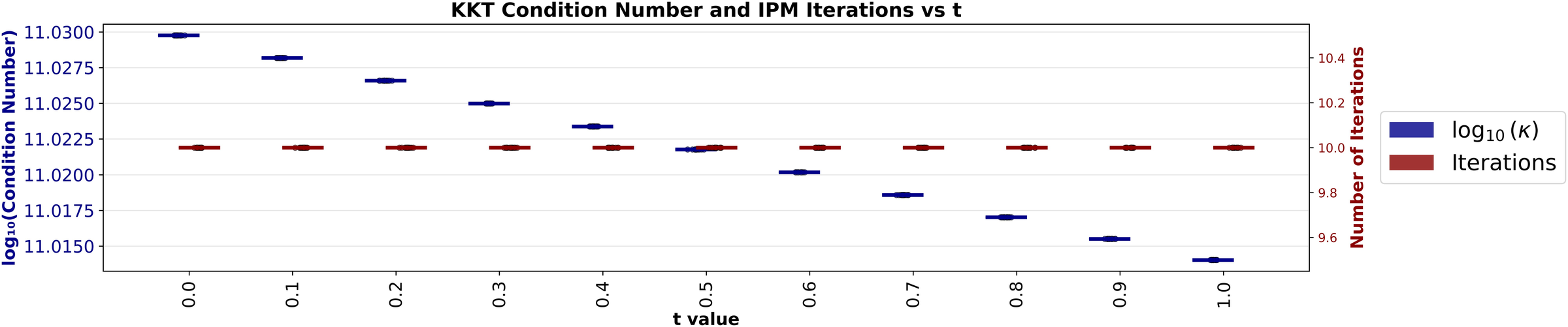} 
        \caption{$n=m=1000, \; \; p=q=500$}
    \end{subfigure}
   \caption{Stationarity matrix condition number and iteration count versus constraints coupling parameter $t$. Grouped boxplots (two per $t$ value) display $\log_{10}(\kappa)$ (blue) and IPM iteration count (red) across $15$ trials. Despite the significantly increasing strength of constraints coupling at $t = 1$, the iteration count remains stable at approximately $20$ iterations.}
    \label{fig:condition_iterations_coupled}
\end{figure}

\subsection{Min cost flow problems and adversarial attacks.}
The aim of this section is to showcase the \textit{quality} of the computed stationary points using our proposed Algorithm \ref{alg:ipm}. For this purpose, we consider a well-established class of adversarial attacks on networks that can be formulated using the \textit{min-max-with-coupled-constraints} analyzed in this work, see \cite{MR4654111,zhang2025iterativeminimaxgamescoupled,hu2024minimizationapproachminimaxoptimization}. Such a problem can be formulated as

\begin{figure}[htbp!]
    \centering
    \begin{subfigure}[b]{\textwidth}
        \centering
        \includegraphics[width=1\textwidth]{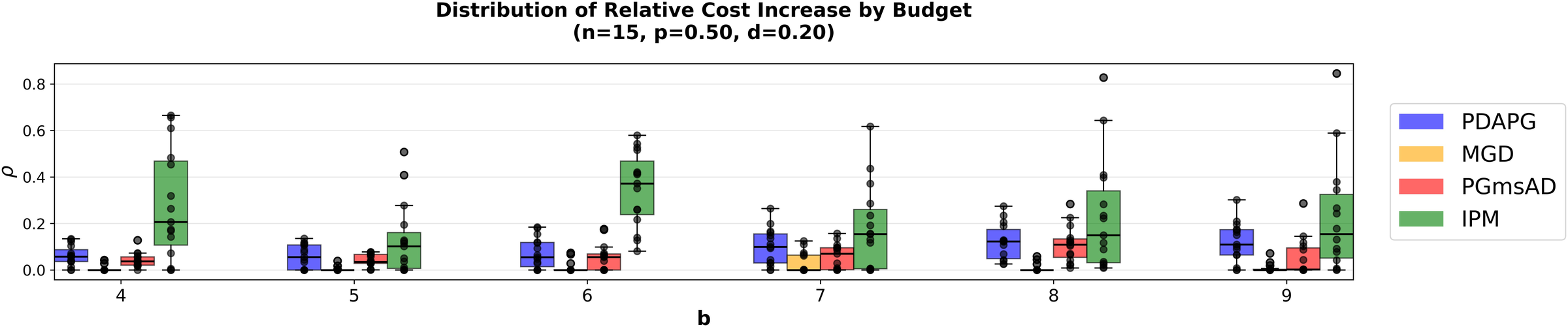} 
        \caption{n=15, p=0.50, d=0.20}

    \end{subfigure} \\
     \begin{subfigure}[b]{\textwidth}
        \centering
       \includegraphics[width=1\textwidth]{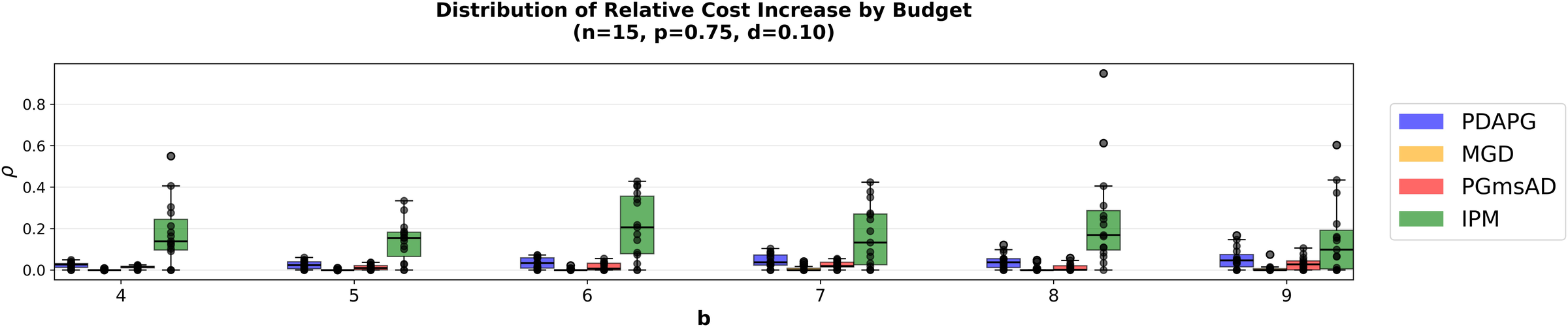} 
        \caption{n=15, p=0.75, d=0.10}

    \end{subfigure}\\
     \begin{subfigure}[b]{\textwidth}
        \centering
        \includegraphics[width=1\textwidth]{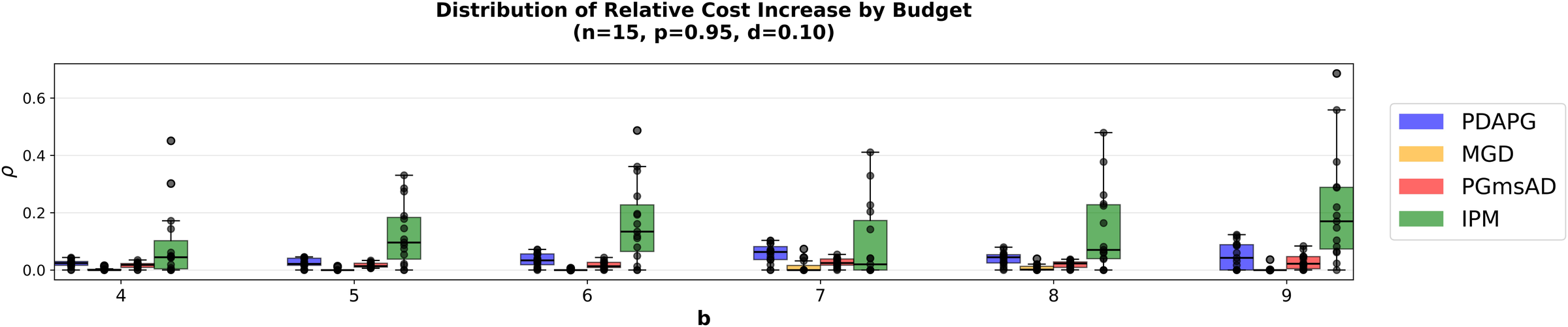}
        \caption{n=15, p=0.95, d=0.10}

    \end{subfigure}\\

       \caption{Comparison of algorithms for network attack: relative cost increase vs. budget parameter.}
    \label{fig:network_attack}
\end{figure}

\begin{align} \label{eq:min_flow_interdiction}
\min_{\substack{(x,z) \geq 0  \\ {e}^\top{x} = a_b}}  \;\; & \max_{\substack{ y \geq 0 }} \sum_{e \in E} - w^\top y - y^\top \diag(w)y - y^\top \diag(w) x + \frac{\eta}{2}\|x\|^2 + \frac{\eta}{2}\|z\|^2 \\
&  {x} + {y} +z = {p}, \nonumber \\
& {\sum_{ e\in E:\, e =(v, t)}}    y_e = r_t,  \nonumber  \\
& P^{s,t}{y} = 0. \nonumber
\end{align}
 The notation details for problem \eqref{eq:min_flow_interdiction} and a detailed problem description are reported Section \ref{sec:adverarial_attack_presentation}. It can be readily seen that problem \eqref{eq:min_flow_interdiction} can be recast in the form \eqref{sec:intro}, observing that the inner can be written as 
\begin{equation}\label{eq:primal_interdiction}
     \begin{bmatrix}
        e^\top & 0 & 0 \\
        0 & 0 &P_{s,t}  \\
        0 & 0 & \hat{e}^\top  \\
        \1 &  \1 & \1 \\
    \end{bmatrix}\begin{bmatrix}
        x \\ z\\ y
    \end{bmatrix}=\begin{bmatrix}
        a_b \\ 0\\ r_t \\ p
    \end{bmatrix},
\end{equation}
i.e., $A_O :=[e^\top,0]$, $b_O:=a_b$, $A_I:=\begin{bmatrix}
    0 & 0 \\ 0 & 0 \\ \1 & \1
\end{bmatrix}$, $B_I:=\begin{bmatrix}
  P_{s,t}  \\
  \hat{e}^\top  \\
 \1  \\
\end{bmatrix}$, and $b_I:=\begin{bmatrix}
    0\\ r_t \\ p
\end{bmatrix}$. Concerning the objective function, we have

\begin{equation*}
    \begin{split}
       &  Q_{11} := \begin{bmatrix}
           \eta I & 0 \\
           0 & \eta I
       \end{bmatrix}, \; Q_{12}:=\begin{bmatrix}
           - \frac{1}{2} \diag(w) \\
           0
       \end{bmatrix},  \;  Q_{22}=:
            \diag(w), \\ 
        & c_x = [0,0],  \; c_y =-w^\top.
    \end{split}
\end{equation*}
Having computed a solution $(x^*,z^*,y^*)$ of problem \eqref{eq:min_flow_interdiction}, we recompute the min-cost flow
\begin{align*}
 \rho_{att}:=\min & \; \sum_{e \in E} c_e(y_e) y_e  \\
\mbox{s.t.} \; & P{y} = b, \\
& 0 \leq  {y} \leq {p}-x^*, \nonumber \\
\end{align*}
and, as a measure of the quality of the computed solution, we consider 
\begin{equation*}
    \rho = \frac{\rho_{att}-\rho_{cf}}{\rho_{cf}},
\end{equation*}
where $\rho_{cf}$ and $\rho_{att}$ are, as defined above, the cost flow before and after the attack. Clearly, bigger $\rho$s correspond to \textit{more effective attacks}. 
{In the following, the total flow $r_t$ is $100*d\%$ of the sum of capacities of the edges exiting the source, with $d$ being a parameter to be chosen, whereas $p$ is the probability that an edge appears in the Erdos-Renyi model random graph. }
In Figure \ref{fig:network_attack}, we report numerical results {for selected $d$ and $p$,} of the approach proposed here, \texttt{IPM}, when compared with \texttt{PGmsAD} and \texttt{MGD}, see \cite{MR4654111}, and \texttt{PDAPG}, see \cite{zhang2025iterativeminimaxgamescoupled}. 
The implementation of the three baseline methods and the selection of hyperparameters are based on the code kindly provided by the authors of \cite{zhang2025iterativeminimaxgamescoupled}. In particular, we choose as regularization parameter $\eta=1e-2$ whereas we redirect the interested reader to \cite{zhang2025iterativeminimaxgamescoupled} for all other details, including how the instances are generated.

The results presented in Figure \ref{fig:network_attack} showcase how Algorithm \ref{alg:ipm}, when applied to solve problem \eqref{eq:min_flow_interdiction}, is able to generate, typically, {more effective} \textit{attacks} {w.r.t. the obtained $\rho$}, than the baseline terms of comparison. Moreover, since Algorithm \ref{alg:ipm} is an infeasible one, in contrast to \texttt{PGmsAD}, \texttt{MGD} and \texttt{PDAPG}, we complement Figure \ref{fig:network_attack} with Figures \ref{fig:network_attack_c_by_budget}, \ref{fig:network_attack_c_by_budget_1} and \ref{fig:network_attack_c_by_budget_2} in Section \ref{sec:adverarial_attack_presentation}, which explore the correlation of $\|r_b\|$, i.e., the residual associated to \eqref{eq:primal_interdiction}, with $\rho$. Such figures clearly show, indeed, how the effectiveness of the attack is not positively correlated with the infeasibility of the computed point and, in general, effective attacks are obtained also when such infeasibility measure is very small.

\section{Conclusions.}
\label{sec:conclusions}
This paper presents a comprehensive framework for computing stationary points of quadratic minmax optimization problems with coupled linear constraints. Our main contributions span theoretical characterizations, algorithmic development, and practical validation.

We developed an infeasible interior point method (Algorithm~\ref{alg:ipm}) that computes stationary points by following a local central path associated with the perturbed stationarity system~\eqref{eq:KKT_final}.{For the important special case of decoupled constraints ($A_I = 0$, cf.\ Remark~\ref{rem:AI=0}), we established polynomial complexity with iteration count $O((n+m)^2|\log(\varepsilon)|)$ to achieve $\varepsilon$-approximate solutions {without assuming strong convexity-concavity (Theorem~\ref{th:polynomial_conv}), as it is the case in the existing literature (see \cite[Section 2]{MR4654111})}. This result demonstrates that the stationary points of minmax problems with decoupled constraints can be computed efficiently.
For the general case with coupled constraints ($A_I \neq 0$), while polynomial complexity cannot be expected due to the inherent non-convexity of the equivalent formulation~\eqref{eq:QPmodel}, we proved local convergence under the assumption that the initial point lies in a neighborhood of a non-degenerate stationary point (Theorem~\ref{th:convergence}).}

Our computational experiments on randomly generated instances demonstrate remarkable robustness across a spectrum of problem structures. The algorithm maintains consistent iteration counts even as strong convexity-concavity properties degrade completely, and as problem dimensions grow. These empirical findings strongly support our polynomial complexity theory for the decoupled case and demonstrate practical local convergence for the coupled case. Applications to adversarial network flow interdiction problems~\eqref{eq:min_flow_interdiction} showcase the algorithm's ability to compute high-quality solutions that significantly increase network costs under optimal attacks.

  {It might be important to note that the quadratic setting of the minmax problem studied in this paper enables the derivation of the necessary optimality conditions of problem \eqref{eq:OuterProblem} for free (cf. Corollary \ref{theo:NecessaryOptCond-withoutCQ}) and construction of the polynomiality results highlighted above. As far as the other main results of the paper  are concerned, they could be extended to more general smooth convex-concave minmax problems under suitable versions of our assumptions. Of course, if nonlinear outer and/or inner constraints are involved in this general settings,  the validity of a constraint qualifications will also become important in the corresponding extensions. Also, in the extension of our methods to such general settings, multivariate first and second order derivatives will be necessary, rather than the more explicit matrix-vector expressions used in this paper.}

\section*{Acknowledgments.} This work was initiated while the 3rd author was visiting the Continuous Optimization Chair at the Institute for Operations Research, Karlsruhe Institute of Technology (KIT), as part of his Alexander von Humboldt Research Fellowship for Experienced Researchers.


\bibliographystyle{siamplain}
\bibliography{references}

\newpage
\appendix


\section{An Alternative Meaningfulness Result}\label{Proof of Theorem Meaningfulness}
%
Here, we present an alternative meaningfulness result to Theorem \ref{the:ndmeaningfulness}, where strict complementarity slackness is not necessary.
The approach is completely different from the one in Subsection \ref{sec:meaningfulness}, and will mainly consists of showing that under appropriate conditions, a stationarity point in the sense of \eqref{eq:KKTjoint} is a strict local optimal solution, in the sense that it satisfies the \textit{second order growth condition}. The second order growth condition will be said to hold at a point $\bar x\in C \subset \mathbb{R}^n$ for an optimization problem \textit{to minimize $\psi(x)$ subject to $x\in C$}, if there exist a neighborhood $U$ of $\bar x$ and a constant $\alpha >0$ such that 
\begin{equation}\label{eq:SecondOrderGrowth}
    \psi(\bar x) + \alpha \|x-\bar x\|^2 \leq \psi(x) \;\, \mbox{ for all }\;\, x\in C\cap U.
\end{equation}

To proceed, we need a few other concepts, including first and second order directional derivatives, as well as the notion of second order epi-regularity. 
The \textit{directional derivative} (in the sense of G\^{a}teaux
) of a function $\psi: \mathbb{R}^n \rightarrow \overline{\mathbb{R}}$ at a point $\bar x \in \mbox{dom }\psi$ in the direction $d\in \mathbb{R}^n$ is the following limit if it exists:
\[
\psi'(\bar x; d)=\underset{t\downarrow 0}{\lim}\frac{\psi(\bar x + td) - \psi(\bar x)}{t}.
\]
Based on this first order notion, we can introduce the second order directional derivative concept that seems to have originated from \cite{BenTal1982NecessaryAS}, and which will play a key role in establishing the results in this section.  If the function $\psi$ is directional differentiable at $\bar x$, its \textit{second order directional derivative} at this point, in the directions $d\in \mathbb{R}^n$ and $w\in \mathbb{R}^n$, is the following limit if it exists:
\[
\psi''(\bar x; d, w)=\underset{t\downarrow 0}{\lim}\frac{\psi(\bar x + td +\frac{1}{2}t^2 w) - \psi(\bar x) - t\psi'(\bar x; d)}{\frac{1}{2}t^2}.
\]
By means of the Taylor expansion of order two, we can easily check that if $\psi$ is twice continuously differentiable at $\bar x$, then we have $\psi''(\bar x; d, w)=\nabla \psi(\bar x)^\top w + d^\top \nabla^2 \psi(\bar x) d$. 

The concept of epi-regularity will be associated to that of second order directional differentiability to generate our main result here. So, we next introduce the former. Let $\psi$ be an extended real-valued function that is second order directionally differentiable at a point $\bar x\in \text{dom} \psi$. It will be said to be \textit{second-order epi-regular} at $\bar x$ in direction $d\in \mathbb{R}^n$ if for each path $w :\mathbb{R}_+ \rightarrow \mathbb{R}^n$ such that $tw(t) \rightarrow 0$  as $t\downarrow 0$, 
\[
\psi\big(\bar x + td + \frac{1}{2}w(t)\big) \geq \psi(\bar x) + t\psi'(\bar x; d) + \frac{1}{2}t^2 \psi''(\bar x; w(t)) + o(t).
\]
The function $\psi$ will be said to be second-order epi-regular at $\bar x$ if it is second-order epi-regular at $\bar x$ in each direction from $\mathbb{R}^n$. 

The following lemma will serve as base to establish the alternative meaningfulness result here.

\begin{lemma}[\cite{bonnans2013perturbation}]\label{Lem:SufficientCond} The second order growth condition holds at $\bar x\in \mbox{dom}\,\psi$ for the minimization of the function $\psi: \mathbb{R}^n \rightarrow \overline{\mathbb{R}}$  if the following conditions hold:
\begin{itemize}
    \item[(i)\;] $\psi$ is \textit{second order directionally differentiable} and \textit{second order epi-regular};
    \item[(ii)\,] $\psi'(\bar x; d)\geq 0$ for all $d\in \mathbb{R}^n$;
    \item[(iii)] $\underset{\omega\in \mathbb{R}^n}{\inf}~\psi''(\bar x; d, \omega) >0$ for all $d\in \left\{\left. r\in \mathbb{R}^n\right|\; \psi'(\bar x; r)=0\right\}$.
\end{itemize}
\end{lemma}

Before stating our result, we introduce the suitable versions of the inner and outer second order sufficient conditions that will be used here. To proceed, note that the \textit{inner critical cone} at $\bar y$ is defined by
\[
\mathcal{C}_{Y(\bar x)}(\bar y)\;:=\;\left\{d\in \mathbb{R}^m\left|\; \left(Q_{22}\bar y - Q_{12}^\top\bar x - c_y\right)^\top d=0, \;\, B_Id=0, \;\, d_j\geq 0, \;\, j\in \mathcal{I}_{\bar y} \right.\right\}
\]
for a fixed outer variable $x:=\bar x$. Thanks to this notion, we can introduce the alternative inner second order sufficient
condition, which will be said to hold at  $(\bar x, \bar y)$ if 
\begin{equation}\label{eq:SecondOrderCond}\tag{ISOSC'}
 d^\top Q_{22}d >0 \;\, \mbox{ for all } d\in \mathcal{C}_{Y(\bar x)}(\bar y)\setminus\{0 \}.
\end{equation}
It is well-known that under the \eqref{eq:SecondOrderCond}, it holds that $\{\bar y\}=S(\bar x)$. Subsequently, under the fulfillment of the \eqref{ILICQ} and \eqref{eq:SecondOrderCond} at $\bar y$ and $(\bar x, \bar y)$ with $\{\bar y\}=S(\bar x)$, respectively, we can introduce the 
outer critical cone 
 \begin{equation}\label{eq:Outer-Critical-Cone}
\begin{array}{rcl}
\mathcal{C}_X(\bar x) &:=&\left\{d\in \mathbb{R}^n\left|\; \left(Q_{11}\bar x + Q_{12}\bar y + c_x + A^\top_I \lambda_I \right)^\top d =  0,\right.\right.\\[1ex]
  & & \qquad \qquad \qquad \qquad \qquad \qquad  A_Od=0, \;\, d_i\geq 0, \;\, i\in \mathcal{I}_{\bar x} \Big\},
\end{array}
 \end{equation}
where  $\lambda_I$ is such that the pair $(\lambda_I, s_I)$ is the unique Lagrange multiplier pair corresponding to the point $(\bar x, \bar y)$ with $\{\bar y\}=S(\bar x)$. Analogously, based on the cone \eqref{eq:Outer-Critical-Cone}, we can introduce the outer second order sufficient
condition to be used here, which will be said to hold at  $(\bar x, \bar y)$ if 
\begin{equation}\label{eq:OuterSecondOrderCond}\tag{OSOSC'}
\forall d\in \mathcal{C}_X(\bar x)\setminus \{0\}: \quad \underset{e}{\sup}\left\{\left[d^\top \; e^\top\right] Q \left[\begin{array}{c}
     d\\ e
\end{array} \right]\\ \left|\;\;\,\begin{array}{l}
          e_j =0,\;\; j\in \mathcal{I}_{\bar y}, \;({{s_I}})_j>0\\[1ex]
          e_j\geq  0, \;\; j\in \mathcal{I}_{\bar y}, \; ({{s_I}})_j=0\\[1ex]
            A_Id + B_Ie =0  
         \end{array}  \right.\right\} >0.
\end{equation}

We are now ready to state and prove our alternative meaningfulness result, in the absence of the strict complementarity slackness. 
\begin{theorem}\label{theo:PreciseMeaningfulness}
Let $\bar x$ be a stationary point in the sense of Definition \ref{Def:StationarityConcept} such that the corresponding pair $(\bar x, \bar y)$ with $\bar y\in S(\bar x)$ satisfies  \eqref{ILICQ}, \eqref{eq:SecondOrderCond}, \eqref{OLICQ}, and  \eqref{eq:OuterSecondOrderCond}. 
 Then $\bar x$ is a strict local optimal solution for problem \eqref{eq:OuterProblem}. 
\end{theorem}

\begin{proof}
The assertion follows from the second-order growth condition at $\bar x$ for problem \eqref{eq:MinMax}. The proof will first ensure that the latter follows from the second-order growth condition at $\bar x$ for the auxiliary unconstrained optimization problem
    \begin{equation}\label{UCO}
    \begin{array}{rcl}
         \underset{x\in \mathbb{R}^n}\min~\psi(x) &:=&\max\left\{\phi(x)-\phi(\bar x),\right.\\[1ex]
                      &  & \left. -x_1, \ldots, -x_n,\; \pm\left[(a^1_O)^\top x-b_1\right], \ldots, \pm\left[(a^p_O)^\top x-b_p\right]\right\}.
    \end{array}
    \end{equation}
Subsequently, the rest of the proof will consists of showing that the requirements for the application of Lemma \ref{Lem:SufficientCond} to the function $\psi$ \eqref{UCO} are satisfied. 
Note that $\psi$ is an extended real-valued function with $\phi$ defined by
\begin{equation}\label{eq:phi_sup}
  \phi(x):=\underset{y\in Y(x)}{\sup}f(x,y) \;\mbox{ for all }\; x\in \mathbb{R}^n.   
\end{equation}
This obviously means that $\phi(x)=\varphi(x)$ for all $x\in X$ considering assumption  \eqref{ass:NonEmptyness}.
Also note that for a vector $a$, we use the notation $\pm a:=\{+a, -a\}$. 

Now, assuming that the second order growth condition holds for  problem \eqref{UCO} at the point $\bar x$, we can find $\epsilon >0$ and $\alpha >0$ such that 
  \begin{equation}\label{growth cond for UCO}
      \psi(x) \geq \alpha \|x-\bar x \|^2 \;\;\mbox{ for all }\;\, x\in \mathbb{B}_\epsilon(\bar x),
  \end{equation}
  considering the fact that $\psi(\bar x) =0$, as $\bar x\in X$. Suppose for a moment now that the second order growth condition does not hold at $\bar x$ for problem \eqref{eq:MinMax}; i.e., for the problem
    \begin{equation}\label{CO}
        \underset{x\in X}{\min}~\varphi(x),
    \end{equation}
where $\varphi$ is defined in \eqref{VarPhi}. 
This implies that we can find a sequence $\{x^k\}_{k\in \mathbb{N}} \subseteq X$ converging to $\bar x$ such that we have  
\[
\varphi(x^k) < \varphi(\bar x) + \frac{1}{k}\|x^k - \bar x\|^2\;\;\mbox{ for all }\;\, k\in \mathbb{N}.
\]
Subsequently, this leads to a contradiction of condition \eqref{growth cond for UCO}, given that
\[
\psi(x^k) \leq \max\left\{\varphi(x^k)-\varphi(\bar x), \;\, 0\right\} < \frac{1}{k}\|x^k - \bar x\|^2\;\;\mbox{ for all }\;\, k\in \mathbb{N}.
\]
In other words, this means that to show that the point $\bar x$ is a strict local optimal solution of problem \eqref{eq:MinMax} in the sense of \eqref{def:localOptSol}, it will suffice to show that problem \eqref{UCO} satisfies the second order growth condition at this point. To prove this, it suffices to show that the properties (i)--(iii) in Lemma \ref{Lem:SufficientCond} hold.

Next, observe that the function $\phi$ \eqref{eq:phi_sup} can be written as $\phi(x):=\tilde{\phi}(x) - \hat{\phi}(x)$ for all $\mathbb{R}^n$, where 
\begin{equation}\label{eq:VarPhi_Decomp_OK}
    \left\{\begin{array}{lll}
      \tilde{\phi}(x)&:=& \frac{1}{2}x^\top Q_{11}x + c^\top_x x, \\[1ex]
       \hat{\phi}(x)&:=&\underset{y\in Y(x)}{\inf} \,\frac{1}{2}y^\top Q_{22} y- x^\top Q^\top_{12} y - c^\top_y y.
    \end{array}\right.
\end{equation}
Hence, based on \cite[Proposition 4]{mehlitz2021sufficient}, it follows from the fulfillment of \eqref{ILICQ} and \eqref{eq:SecondOrderCond} that $\phi$ is continuously differentiable at $\bar x$ with 
\begin{equation}\label{eq:NablaPhi}
   \nabla \phi(\bar x) = Q_{11}\bar x + Q_{12}\bar y + c_x + A^\top_I \lambda_I,
\end{equation}
where $\lambda_I$ is the unique (thanks to the fulfillment of \eqref{ILICQ}) Lagrange multiplier associated to the inner equality constraint, and moreover, $\hat{\phi}$ and $-\hat{\phi}$ are second order directionally differentiable and epi-regular at $\bar x$, and we have 
\[     
\hat{\phi}''(\bar x; d,\omega) =  \left(\lambda^\top_IA_I - \bar{y}^\top Q_{12}^\top\right) \omega - \underset{e}{\sup}\left\{\left.\zeta^\top \hat{Q}\zeta\; \right|\;\;\, e\in \mathcal{Z}_d(\bar x, \bar y, {{s_I}}, \lambda_I)  \right\},
\]
where $\{({{s_I}}, \lambda_I)\}=\Lambda (\bar x, \bar y)$ thanks to \eqref{ILICQ},  while $\zeta := (d, e)\in \mathbb{R}^n\times \mathbb{R}^m$ and the matrix $\hat{Q}$ and set  $\mathcal{Z}_d(\bar x, \bar y, {{s_I}}, \lambda_I)$ are respectively defined by 
 \[
\hat{Q} :=\left[\begin{array}{lr}
  O  &  Q_{12}\\
   Q_{12}^\top  & -Q_{22}
\end{array} \right] \, \mbox{ and }\,  \mathcal{Z}_d(\bar x, \bar y, {{s_I}}, \lambda_I) := \left\{e \left|\begin{array}{l}
          e_j =0,\;\; j\in \mathcal{I}_{\bar y}, \;(s_I)_j>0\\[1ex]
          e_j\geq  0, \;\; j\in \mathcal{I}_{\bar y}, \; (s_I)_j=0\\[1ex]
            A_Id + B_Ie =0  
         \end{array}  \right.\right\}.
\]
And subsequently, considering the expression of $\phi$ as decomposed in \eqref{eq:VarPhi_Decomp_OK}, 
\begin{equation}\label{eq:Varphi2Primes}
    \begin{array}{rcl}
\phi''(\bar x; d,\omega) &=&  d^\top Q_{11}d  + \bar{z}^\top \bar{Q}\omega + \underset{e}{\sup}\left\{\left.\zeta^\top \hat{Q}\zeta\; \right|\;\;\, e\in \mathcal{Z}_d(\bar x, \bar y, {{s_I}}, \lambda_I)  \right\}\\[2ex]
& = & \bar{z}^\top \bar{Q}\omega + \underset{e}{\sup}\left\{\left.\zeta^\top Q\zeta\; \right|\;\;\, e\in \mathcal{Z}_d(\bar x, \bar y, {{s_I}}, \lambda_I)  \right\},
\end{array}
\end{equation}
where $\bar{z}^\top$ and $\bar{Q}$ are respectively defined by 
\[
\bar{z}^\top := \left[\bar{x}^\top, \; \bar{y}^\top, \; \lambda^\top_I, \; 1\right] \;\mbox{ and }\;  \bar{Q}:=\left[Q_{11} \;\; Q_{12} \;\; -A^\top_I \;\; c_x\right]^\top.
\]

To prove the fulfillment of condition (ii) in Lemma~\ref{Lem:SufficientCond}, we will show that the stationarity of the point $\bar x$ ensures  
\begin{equation}\label{eq:varphi-prime}
     \phi'(\bar x; d) \geq 0 \;\, \mbox{ for all }\;\, d\in \mathcal{L}_X(\bar x),
\end{equation}
where $\mathcal{L}_X(\bar x)$ corresponds to the linearized cone of $X$ at the point $\bar x\in X$:
\begin{equation}\label{eq:LinearizedCone}
    \mathcal{L}_X(\bar x):=\left\{d\in \mathbb{R}^n|\;\, A_Od=0, \;\, d_i\geq 0, \;\, i\in \mathcal{I}_{\bar x} \right\}.
\end{equation}
To see why the fulfillment of  \eqref{eq:varphi-prime} ensures (ii) for $\psi$, observe that for all $d\in \mathbb{R}^n$, 
\begin{equation}\label{eq:Kat-1}
\begin{array}{rcl}
  \psi'(\bar x; d) &=& \max\left\{\phi'(\bar x; d), \;\, \max\left\{-d_i|\;\, i\in \mathcal{I}_{\bar x}\right\},\right.\\[1.2ex]
  & & \qquad \qquad \qquad \qquad\qquad  \left.\pm(a^1_O)^\top d, \ldots, \pm(a^p_O)^\top d\right\},
\end{array}
\end{equation}
while accounting for the fact that since $\phi$ is continuously differentiable at $\bar x$, $\phi'(\bar x; d)$ exists and based on \eqref{eq:NablaPhi}, we have 
\begin{equation}\label{eq:VarPhiPrime}
    \phi'(\bar x; d) =\left(Q_{11}\bar x + Q_{12}\bar y + c_x + A^\top_I \lambda_I \right)^\top d.
\end{equation}
It therefore follows that based on the formula \eqref{eq:Kat-1}, if condition \eqref{eq:varphi-prime} is satisfied, then assumption (ii) in Lemma \ref{Lem:SufficientCond} holds. 

Now, observe that based on \eqref{eq:VarPhiPrime}, we have $\phi'(\bar x; 0)=0$. Additionally, with the  linearized cone $\mathcal{L}_X(\bar x)$ being convex, condition \eqref{eq:varphi-prime} is equivalent to having $0$ as an optimal solution to problem 
\begin{equation}\label{eq:OptCond-0-d}
    \underset{d}{\min}~\phi'(\bar x; d) \; \mbox{ s.t. } \; d\in \mathcal{L}_X(\bar x),
\end{equation}
which is a linear program, considering the expressions in \eqref{eq:LinearizedCone} and \eqref{eq:VarPhiPrime}. Hence, the  $0$ being an optimal solution of problem \eqref{eq:OptCond-0-d} is equivalent to the existence of  $(\lambda_O,\, s_O)\in\R^p\times \R^n$ such that 
\begin{equation}\label{eq:EqLabel-0-R}
    \begin{array}{r}
         \nabla\phi(\bar x)+A_O^\top\lambda_O-(s_O)_{{\cal I}_{\bar x}}=0,\\[1ex]
         (s_O)_i \geq 0, \;\, \in {\cal I}_{\bar x},
    \end{array}
\end{equation}
while taking into account that $\nabla_d \phi'(\bar x; 0) = \nabla \phi (\bar x)$. 
Since $\bar x$ is a stationary point in the sense of Definition \ref{Def:StationarityConcept}, it follows that from the \textit{outer KKT system} component of  \eqref{eq:KKTjoint} and the formula in \eqref{eq:NablaPhi} that the condition \eqref{eq:EqLabel-0-R} is satisfied. Therefore \eqref{eq:OptCond-0-d} holds, and hence, \eqref{eq:varphi-prime} is satisfied. So, condition (ii) in Lemma \ref{Lem:SufficientCond} holds. 

Finally, to establish the fulfillment of assumption (iii) in Lemma \ref{Lem:SufficientCond}, start by observing that based on  \eqref{eq:VarPhiPrime}, we can easily check that the cone  \eqref{eq:Outer-Critical-Cone} can be rewritten as 
\[
\mathcal{C}_X(\bar x) = \left\{d\left.\in \mathbb{R}^n\right|\;\,\psi'(\bar x; d) = 0\right\}.
\]
Now, let $d\in \mathcal{C}_X(\bar x)$. Then considering \eqref{eq:VarPhiPrime} and the formula \eqref{eq:Outer-Critical-Cone} it holds that $\phi'(\bar x; d)=0$. Hence, for all $\omega\in \mathbb{R}^n$,
\[
\begin{array}{rll}
 \psi''(\bar x; d, \omega) &=& \max\left\{\phi''(\bar x; d,\omega), \;\,\left\{-\omega_i,\;\, i\in \mathcal{I}_{\bar x}(d)\right\} \right.,\\[1ex]
                           & &   \qquad \qquad \qquad \;\, \left.\left\{ \pm(a^k_O)^\top w,\;\, j=1, \ldots, q \right\}\right\}
\end{array}
\]
with the expression of $\phi''(\bar x; d,\omega)$ given in \eqref{eq:Varphi2Primes}, 
while  $\mathcal{I}_{\bar x}(d)$ is given by 
\[
\mathcal{I}_{\bar x}(d):=\left\{i\in \mathcal{I}_{\bar x}\left|\;\,d_i=0\right.\right\}.
\]
Based on this formula, \eqref{eq:OuterSecondOrderCond} ensures that (iii) in Lemma \ref{Lem:SufficientCond} is satisfied if  
\begin{equation}\label{eq:SCond}
    \underset{\omega,\, \alpha}{\min}\left\{\alpha \left|\begin{array}{rl}
        \phi''(\bar x; d, \omega)\leq \alpha &  \\[1ex]
        -w_i \leq \alpha, & i\in \mathcal{I}_{\bar x}(d) \\[1ex]
        \pm(a^k_O)^\top w\leq \alpha, & j=1, \ldots, q
    \end{array}  \right.\right\} > 0 \;\mbox{ for all }\; d\in \mathcal{C}_X(\bar x)\setminus \{0\}
\end{equation}
with $a^k_O$ denotes the $j^{\text{th}}$ row of the matrix $A_O$.

Based on the formula of $\phi''(\bar x; d,\omega)$ given in \eqref{eq:Varphi2Primes}, condition \eqref{eq:SCond} reduces to
\begin{equation}\label{eq:SCond-NoVel}
       \underset{\omega,\, \alpha}{\min}\left\{\alpha \left|\begin{array}{rll}
       \alpha - \bar{z}^\top \bar{Q}\omega  &\geq &  \rho_d(\bar x, \bar y, {{s_I}}, \lambda_I)    \\[1ex]
     \alpha + e^\top_i\omega & \geq &  0,\;\;  i\in \mathcal{I}_{\bar x}(d) \\[1ex]
     \alpha -   (a^k_O)^\top w & \geq & 0,\;\; j=1, \ldots, q\\[1ex]
     \alpha +  (a^k_O)^\top w & \geq & 0, \;\; j=1, \ldots, q
    \end{array}  \right.\right\} > 0,
\end{equation}
where we have 
\[
\rho_d(\bar x, \bar y, {{s_I}}, \lambda_I):=\underset{e}{\sup}\left\{\left.\zeta^\top Q\zeta\; \right|\;\;\, e\in \mathcal{Z}_d(\bar x, \bar y, {{s_I}}, \lambda_I)  \right\}.
\]
The dual of the linear program in the left-hand-side of equation \eqref{eq:SCond-NoVel} is
\begin{equation}\label{eq:DualPb-1}
    \underset{\kappa, {{s_O}}, \lambda_O}{\max}\left\{\left.\kappa \rho_d(\bar x, \bar y, {{s_I}}, \lambda_I)\right|\;\, (\kappa, {{s_O}}, \lambda_O)\in  \Delta_d(\lambda_I)\right\}
\end{equation}
with the set $\Delta_d(\lambda_I)$ defined by
\begin{equation}\label{eq:SCond-2-1}
      \Delta_d(\lambda_I):= \left\{(\kappa, {{s_O}}, \lambda_O) \left|\begin{array}{l}
       \kappa\bar{Q}^\top \bar{z} - \underset{{i\in \mathcal{I}_{\bar x}(d)}}{\sum} ({{s_O}})_ie^i_n - \underset{{j=1}}{\overset{q}{\sum}} (\lambda_O)_j a^k_O  =0\\[1.7ex]
       \kappa  + \underset{{i\in \mathcal{I}_{\bar x}(d)}}{\sum} ({{s_O}})_i + \underset{{j=1}}{\overset{q}{\sum}} (\lambda_O)_j =1\\[1.7ex]
       \kappa \geq 0, \;\; ({{s_O}})_i \geq 0, \;\; i\in \mathcal{I}_{\bar x}(d),\;\; \lambda_O\in \mathbb{R}^q
    \end{array}  \right.\right\}.
\end{equation}
Considering the fulfillment of \eqref{OLICQ} at $\bar x$, we must have $\kappa \neq 0$ and hence, we get the result from the combination of condition \eqref{eq:SCond-NoVel} and the dual of the left-hand-side linear program, which is provided in \eqref{eq:DualPb-1}--\eqref{eq:SCond-2-1}. 
\end{proof}

\section{Proof of Theorem \ref{Th:IFT}} \label{sec:Implicit Function Theorem}
    Consider $(v,\mu) \in B({v^*},\delta) \times \Lambda_{\bar{\mu}} \subset V_{\delta}$ and define the function 
    \[
    \Theta_{\mu}(v):= v - \partial_{v}H(v^*,0)^{-1} H(v,\mu).
    \]
Clearly, we have $\Theta_{\mu}(v)=v$ if and only if $H(v,\mu)=0$. Moreover,

\begin{equation}\label{eq:support_1}
\begin{split}
   & \|\Theta_{\mu}v^*- v^*\| =\|\Theta_{\mu}v^*- \Theta_{0}v^*\| \leq M\| H(v^*,\mu)- H(v^*,0)\| \leq MB_{\delta } \bar{\mu}.
\end{split}
\end{equation}

We want to use the fixed point theorem to prove that for all $\mu \in \Lambda_{\bar{\mu}}$, there exists $v \in \overline{B({v^*}, \delta)}$ such that $\Theta_\mu (v)=v$. To proceed, we need to show the following: 
\begin{enumerate}
    \item $\Theta_{\mu}(\overline{B({v^*},\delta)}) \subset \overline{B({v^*},\delta)}$; 
    \item $\exists \; \sigma \in (0,1)$:\;\, $\|\Theta_\mu (v_1)-\Theta_\mu(v_2) \| \leq \sigma \|v_1-v_2\|$ for all $v_1,\;v_2 \in \overline{B({v^*},\delta)}$.
\end{enumerate}
For the first assertion, observe that for $v \in \overline{B({v^*},\delta)}$, we have
\[
\begin{array}{rll}
    \|\Theta_{\mu}(v)-v^*\| & \leq & \|\Theta_{\mu}(v) -\Theta_\mu(v^*)\| + \| \Theta_\mu(v^*) -v^*\| \\[1.5ex]
                            & \leq &  \sup_{(v,\mu) \in V_{\delta}}\|\partial_v \Theta_\mu(v)\|\|v-v^*\|+ MB_\delta \bar{\mu} \\[1.5ex]
                            & \leq & \frac{1}{2}\|v-v^*\| + \frac{1}{2} \delta\\[1.5ex]
                            & \leq & \delta,
\end{array}
\]
where the last but one inequality follows by the definition of $V_{\delta}$, by \eqref{eq:support_1} and the definition of $\bar{\mu}$. To prove 2., observe that $\sigma=1/2$ since
\begin{equation*}
    \begin{split}
    & \|\Theta_\mu (v_1)-\Theta_\mu(v_2)  \| \leq  \frac{1}{2}\|v_1-v_2\|,
    \end{split}
\end{equation*}
using again the definition of $V_{\delta}$. Hence, the function 
$g:\Lambda_{\bar{\mu}}  \rightarrow \overline{B({v^*},\delta)} \subset \mathbb{R}^N$ defined by $g(\mu):= v_\mu $
such that $\Theta_\mu(v_\mu)=v_{\mu}$ and $H(g(\mu),\mu)=0$ is well defined. To prove regularity, observe that for $\mu_1,\mu_2 \in \Lambda_{\bar{\mu}}$, we have
\[
\begin{array}{rll}
    \|g(\mu_1)-g(\mu_2)\| & = & \|\Theta_{\mu_1}(v_{\mu_1}) - \Theta_{\mu_2}(v_{\mu_2}) \|\\[1.5ex]
                          & = & \|\Theta_{\mu_1}(v_{\mu_1}) - \Theta_{\mu_1}(v_{\mu_2}) +  \Theta_{\mu_1}(v_{\mu_2}) -\Theta_{\mu_2}(v_{\mu_2})\|\\[1.5ex]
                          &\leq & \frac{1}{2}\| v_{\mu_1}-v_{\mu_2} \| +  M\| H(v_{\mu_2},\mu_1)- H(v_{\mu_2},\mu_2)\| \\[1.5ex]
                          & \leq & \frac{1}{2}\|g(\mu_1)-g(\mu_2)\|+MB_{\delta}|\mu_1-\mu_2|
\end{array}
\]
from which Lipschitz continuity  of $g$ on $\Lambda_{\bar{\mu}}$ follows. By the differentiability of $\hat{F}$ and the Lipschitz-continuity of $g$, given $\mu \in \Lambda_{\bar{\mu}}$ and $h$ sufficiently small, we have that
\[
\begin{array}{l}
   \left\|H(g(\mu+h),\mu+h) - H(g(\mu),\mu) \right. \\[1.5ex]
   \qquad \qquad \qquad 
   +\;\partial_v H(g(\mu),\mu)\left[g(\mu+h)-g(\mu)\right] \;+\; \left.\partial_\mu H(g(\mu),\mu)h\right\| = o(|h|), 
\end{array}
\]
which implies that 
\begin{equation*}
\|\partial_v H(g(\mu),\mu)[g(\mu+h)-g(\mu)] + \partial_\mu H(g(\mu),\mu)h\| = o(|h|)
\end{equation*}
given that $H(g(\mu+h),\mu+h) = H(g(\mu),\mu) =0$. Hence, 
\begin{equation*}
    \lim_{|h| \to 0 } \|g(\mu+h)-g(\mu) + \partial_v H(g(\mu),\mu)^{-1} \partial_\mu H(g(\mu),\mu)h\| =0,
\end{equation*}
 and the continuity of  $\frac{d}{d\mu}g(\lambda)$ follows from the given  expression of the derivative. 

\section{Proof of Theorem \ref{th:polynomial_conv}} \label{sec:appendix_AI=0}
The proof of Theorem \ref{th:polynomial_conv} follows closely the one from  \cite[Theorem 6.2]{MR1422257} and using Lemma \ref{lme:crucial_lemma}.  For the sake of completeness, in the remainder of this section, we will outline the main steps. Before, observe that defining $\nu_k:=\prod_{j=0}^{k-1} (1-\alpha_j)$  with $\nu_0=1$, and using the linearity of  the first two equations in  \eqref{eq:KKT_final}, we have that 
\begin{equation*}
	(r_b^k, r_c^k) = \nu_k (r_b^0, r_c^0).
\end{equation*}
Coming back to the Polynomial Convergence proof, let us start observing that also in the case here considered the following estimations for $\|(z^k,s^k)\|_1 $ are true:

\vspace{0.2cm}

\begin{enumerate}[({E}1)]
	\item  There exists a constant $C_1$ such that 
    \[
    \nu_k \|(z^k,s^k)\|_1 \leq C_1\mu^k \; \mbox{ for all } \; k \geq 0.
    \]
    For the proof, our Lemma \ref{lme:crucial_lemma} can be applied in the context of the proof of \cite[Lemma 6.3]{MR1422257}, see in particular \cite[Eq. 6.19]{MR1422257};
		\item  If the initial point is chosen as in Theorem \ref{th:polynomial_conv}, then 
        \[
        \zeta \nu_k \|(z^k,s^k)\|_1 \leq 4 \beta (n+m) \mu^k \; \mbox{ for all } \;k \geq 0.
        \]
        The proof  follows using (E1)  and an analogous  proof of \cite[Lemma 6.4]{MR1422257}.
\end{enumerate}

\vspace{0.2cm}

Moreover, also in the case here considered, defining $D^k:=(Z^k)^{1/2}(S^k)^{-1/2}$, the following estimations about the norm of the Newton directions hold:

\vspace{0.2cm}

\begin{enumerate}[({N}1)]
	\item  There exists a constant $C_2$ such that 
    \[
    \| (D^k)^{-1}\Delta z^k \|  \leq C_2 (\mu^k)^{1/2}, \;\; \| D^k\Delta s^k \|  \leq C_2 (\mu^k)^{1/2}     \hbox{ for all } k \geq 0.
    \]
	For the proof use our Lemma \ref{lme:crucial_lemma} in the proof of \cite[Lemma 6.5]{MR1422257};  in particular in \cite[Eq. (6.30)]{MR1422257}   the sign ``$=$'' becomes ``$\leq$'';
	\item  If the initial point is chosen as in Theorem \ref{th:polynomial_conv}, then there exists a constant $\omega$ independent from $m+n$ such that 
	$$ \| (D^k)^{-1}\Delta z^k \|  \leq \omega (m+n) (\mu^k)^{1/2}, \;\; \| D^k\Delta s^k \|  \leq \omega (m+n) (\mu^k)^{1/2}     \hbox{ for all } k \geq 0 .$$
	The proof  follows using (N1) and an analogous  proof of \cite[Lemma 6.6]{MR1422257}).
\end{enumerate}

\vspace{0.2cm}

Finally, the proof of Theorem \ref{th:polynomial_conv} follows by observing that using (E1), (E2), (N1), and (N2), the statement of \cite[Lemma 6.7]{MR1422257} holds also in our case.

\section{Proofs from Subsection \ref{sec:coupled_local_convergence}} \label{sec:appendix_AI_neq0}

\subsection{Proof of Lemma \ref{lem:damped-lemma3}}

\textbf{(a)} Using the triangle inequality and \eqref{eq:bound-central-path}, it holds that 
\begin{align*}
\|v_\alpha - g(\mu^+)\| &= \|(1-\alpha)(v - g(\mu^+)) + \alpha(v^+ - g(\mu^+))\|, \notag \\
&\leq (1-\alpha)\|v - g(\mu^+)\| + \alpha\|v^+ - g(\mu^+)\|, \notag \\
&\leq (1-\alpha)\|v - g(\mu^+)\| + \alpha\frac{KL}{2}\|v - g(\mu^+)\|^2.
\end{align*}

\textbf{(b)} Similarly, using \eqref{eq:bound-solution}, we have 
\begin{align*}
\|v_\alpha - v^*\| &= \|(1-\alpha)(v - v^*) + \alpha(v^+ - v^*)\|, \notag \\
&\leq (1-\alpha)\|v - v^*\| + \alpha\|v^+ - v^*\|, \notag \\
&\leq (1-\alpha)\|v - v^*\| + \alpha\left(\frac{1}{2}\|v - v^*\| + \frac{\mu^+}{\mu^*}\frac{\delta^*}{2}\right), \notag \\
&= \left(1 - \frac{\alpha}{2}\right)\|v - v^*\| + \alpha\frac{\mu^+}{\mu^*}\frac{\delta^*}{2}.
\end{align*}
The fact that $v_\alpha \in B(v^*, \delta^*)$ follows from the fact that the right-hand side of \eqref{eq:damped-solution} is maximized at $\alpha = 1$ (the full step), which by Lemma~\ref{lem:original-lemma3} yields a value $< \delta^*$.

\subsection{Proof of Corollary \ref{rem:uniform_boundedness_rhs}}

As a consequence of Theorem \ref{th:IPM_well_definitness}, we can suppose the existence of a sequence of iterates  $\{(z^k,\lambda^k,s^k)\}_{k \in \mathbb{N}}$ produced by Algorithm \ref{alg:ipm} such that $$(z^k,\lambda^k,s^k) \in \mathcal{N}(\bar {\gamma},\underline{\gamma},\gamma_p,\gamma_d). $$
	Since by construction $(z^k)^\top s^k \leq (z^0)^\top s^0 $, we have from \eqref{eqn:neighborhood} that
	\begin{equation*}
    \begin{array}{rll}
\|Az^k-b\| & \leq & (z^0)^\top s^0/\gamma_p, \\[1.5ex]
 \|\hat{Q}z^k + \hat{A}^\top\lambda^k-s^k + \hat{c}\| & \leq & (z^0)^\top s^0/\gamma_d. 
    \end{array}
	\end{equation*}
	Moreover, we have
	\begin{equation*}
    \begin{array}{rll}
       \|S^kZ^k\mathbf{e}-\sigma \mu^k \mathbf{e}\| & \leq & \|S^kZ^k\mathbf{e}\|+\sigma \mu^k \|\mathbf{e}\|, \\[1.5ex]
			& \leq & \frac{\bar{\gamma}}{\sqrt{n+m}} (z^k)^\top s^k +  \frac{\sigma}{\sqrt{n+m}} (z^k)^\top s^k, \\[1.5ex]
            &\leq &\frac{\bar{\gamma}+\sigma}{\sqrt{n+m}} (z^0)^\top s^0.
    \end{array}
	\end{equation*}

\section{Min-Cost Flow Adversarial Attack} \label{sec:adverarial_attack_presentation}
For the sake of clarity, we will briefly present how the formulation \eqref{eq:min_flow_interdiction} is obtained and the related notational details. To this aim, consider a directed graph $G = (V, E)$, where $V$ is the set of vertices and $E$ is the set of edges. Define, moreover, the incidence matrix $P$ of $G$ as the matrix  $P \in \mathbb{R}^{|V| \times |E|}$ with entries defined as
\begin{equation} \label{eq:incidence}
P_{ve} = \begin{cases}
+1 & \mbox{ if } e = (w,v) \text{ for some } w \in V,\\
-1 & \mbox{ if }  e = (v,w) \text{ for some } w \in V, \\
0 & \mbox{ otherwise}.
\end{cases}
\end{equation}

Given $b \in \mathbb{R}^{|V|}$ such that $\sum_{v \in V} b_v =0$, where for{$v \in V$, $b_v>0$} is an exogenous source whereas $b_v<0$ is exogenous demand, and a capacity vector $ u \in \mathbb{R}^{|E|}$,  the min-cost flow problem on $G$ can be formulated as 

\begin{align*}
 \rho_{cf} :=\min & \; \sum_{e \in E} c_e(y_e) y_e  \\
\mbox{s.t.} \; & P{y} = b, \\
& 0 \leq  {y} \leq {p}, \nonumber \\
\end{align*}
where ${y} := \{y_e\}_{e \in E} \in \mathbb{R}^{|E|}$ and  $c_e(y_e)$ denote the cost of moving one unit of flow across edge $e$. In the following, we fix two nodes, source and sink, $s,t$, and a sink demand $b_t=r_t \Rightarrow b_s=-r_t$. Let us now consider an extension to such a problem, where, in addition to the regular network user, \textit{the follower}, there is an adversary, \textit{the leader}, who will inject flows to the network in order to force the regular user to use more expensive paths. Let ${x} := \{x_e\}_{e \in E}$ denote flows controlled by the adversary, with a total attack-budget $a_b > 0$. Then the problem of the adversary can be formulated as 
\begin{align*}
\max_{\substack{0 \leq {x} \leq {p} \\ {e}^\top{x} = a_b}}  \;\; & \min_{\substack{0 \leq {y} \leq {p} } } \sum_{e \in E} c_{e}(y_e, x_{e})y_e \\
&  {x} + {y} \leq {p}, \nonumber \\
& {\sum_{ e\in E:\, e =(v, t)}y_e = r_t},  \nonumber  \\
& {P^{s,t}}{y} = b,  \nonumber 
\end{align*}
where, for $e\in E$, $c_{e}(y_e, x_{e})$ is related to the cost of routing both $x_{e}$ and $y_{e}$ onto the network and {$P^{s,t}$} is the incidence matrix defined {in \eqref{eq:incidence}}, but having had removed the rows $s,t$. In the following, we will consider $c_{e}(y_e, x_{e})=w_e(y_e+x_e)$, i.e., the cost of routing the flow is a function of the amount of flow there routed. It is important to note that, using the particular form of $c_e$ and eliminating redundant constraints, the above problem can be written as 

\begin{align*} 
\max_{\substack{({x}, z) \geq 0  \\ {e}^\top{x} = a_b}}  \;\; & \min_{\substack{ y \geq 0 }} \sum_{e \in E} w^\top y + y^\top \diag(w)y + y^\top \diag(w) x -\frac{\eta}{2}\|x\|^2 -\frac{\eta}{2}\|z\|^2 \\
&  {x} + {y} +z = {p}, \nonumber \\
& {\sum_{ e\in E:\,e =(v, t)}}    y_e = r_t,  \nonumber  \\
& P^{s,t}{y} = 0, \nonumber
\end{align*}
where we note that the variable $z$ belongs to the attacker, and we added the term $\frac{\eta}{2}\|x\|^2+\frac{\eta}{2}\|z\|^2$ to ensure strong convexity for the outer problem as done in \cite{MR4654111,zhang2025iterativeminimaxgamescoupled,hu2024minimizationapproachminimaxoptimization}. 

\subsection{Correlation figures.}  Figures \ref{fig:network_attack_c_by_budget}, \ref{fig:network_attack_c_by_budget_1}, and  \ref{fig:network_attack_c_by_budget_2} show the correlations among $\|r_b\|$ and $\rho$.

\begin{figure}[htbp!]
    \centering
    \begin{subfigure}[b]{\textwidth}
        \centering
        \includegraphics[width=1\textwidth]{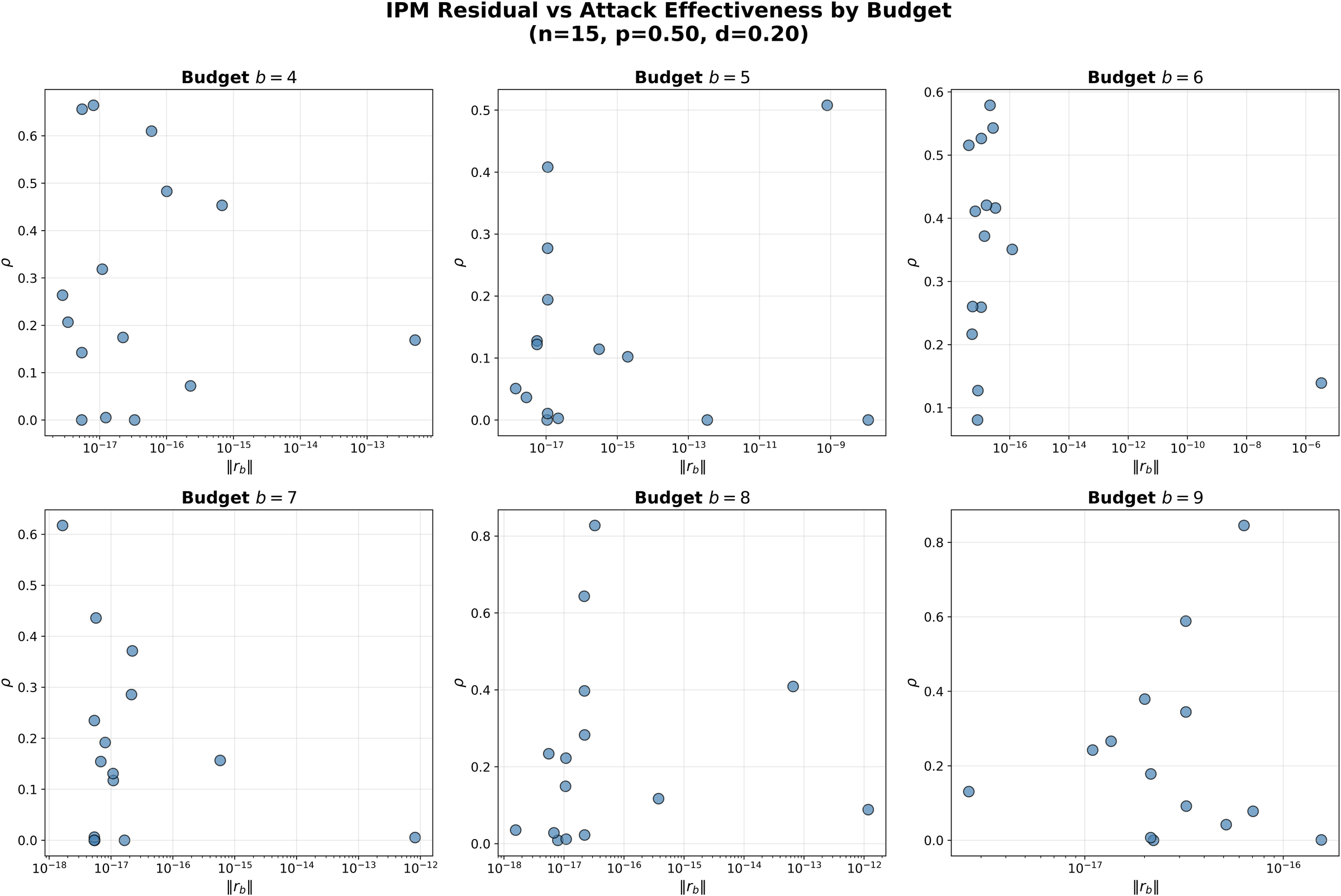} 
        \caption{n=15, p=0.50, d=0.20}
    \end{subfigure} \\
       \caption{Correlation among $\|r_b\|$ and $\rho$.}
    \label{fig:network_attack_c_by_budget}
\end{figure}

\begin{figure}[htbp!]
    \centering
     \begin{subfigure}[b]{\textwidth}
        \centering
       \includegraphics[width=1\textwidth]{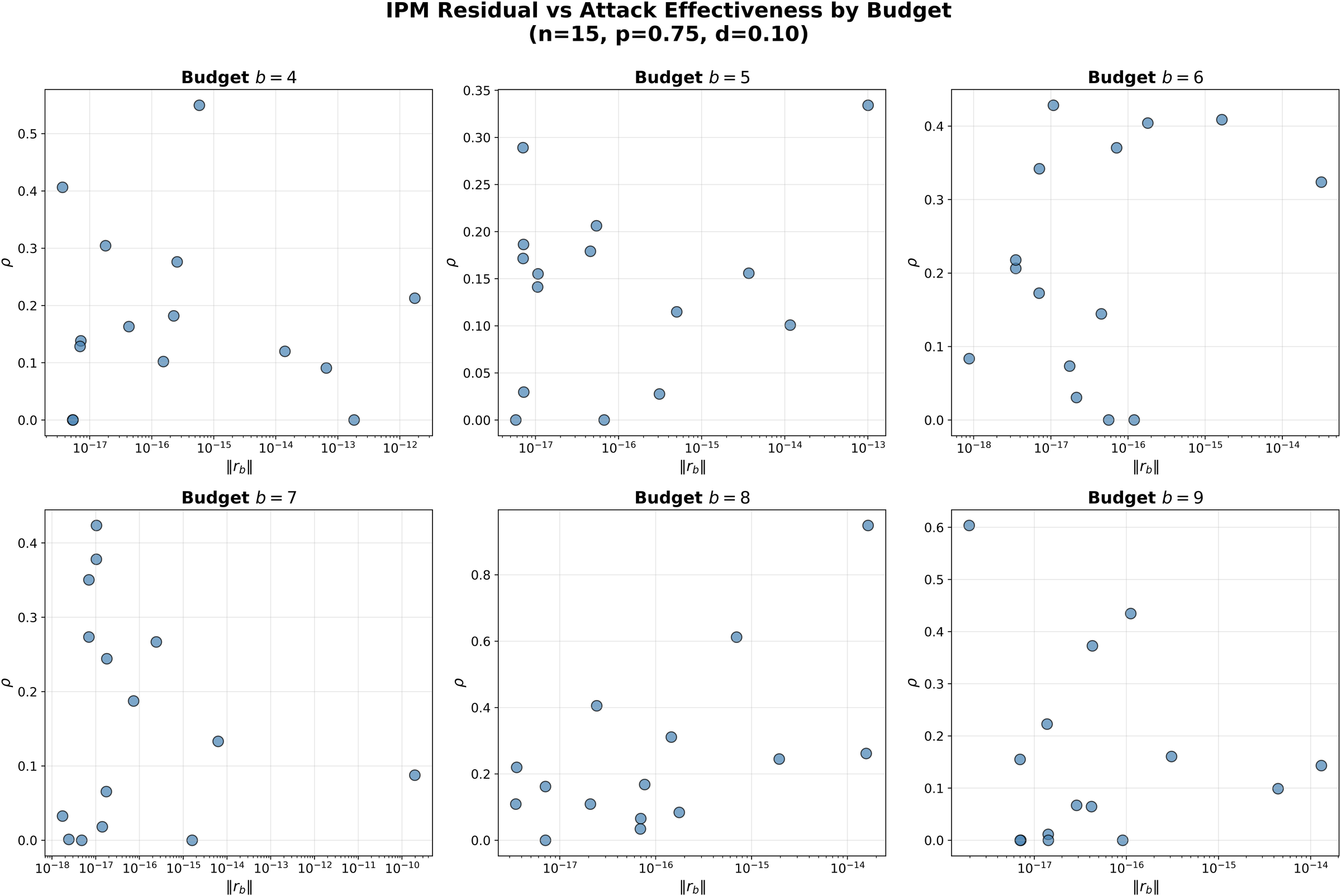} 
        \caption{n=15, p=0.75, d=0.10}
    \end{subfigure}
       \caption{Correlation among $\|r_b\|$ and $\rho$.}
    \label{fig:network_attack_c_by_budget_1}
\end{figure}

 \begin{figure}[htbp!]
    \centering
     \begin{subfigure}[b]{\textwidth}
        \centering
        \includegraphics[width=1\textwidth]{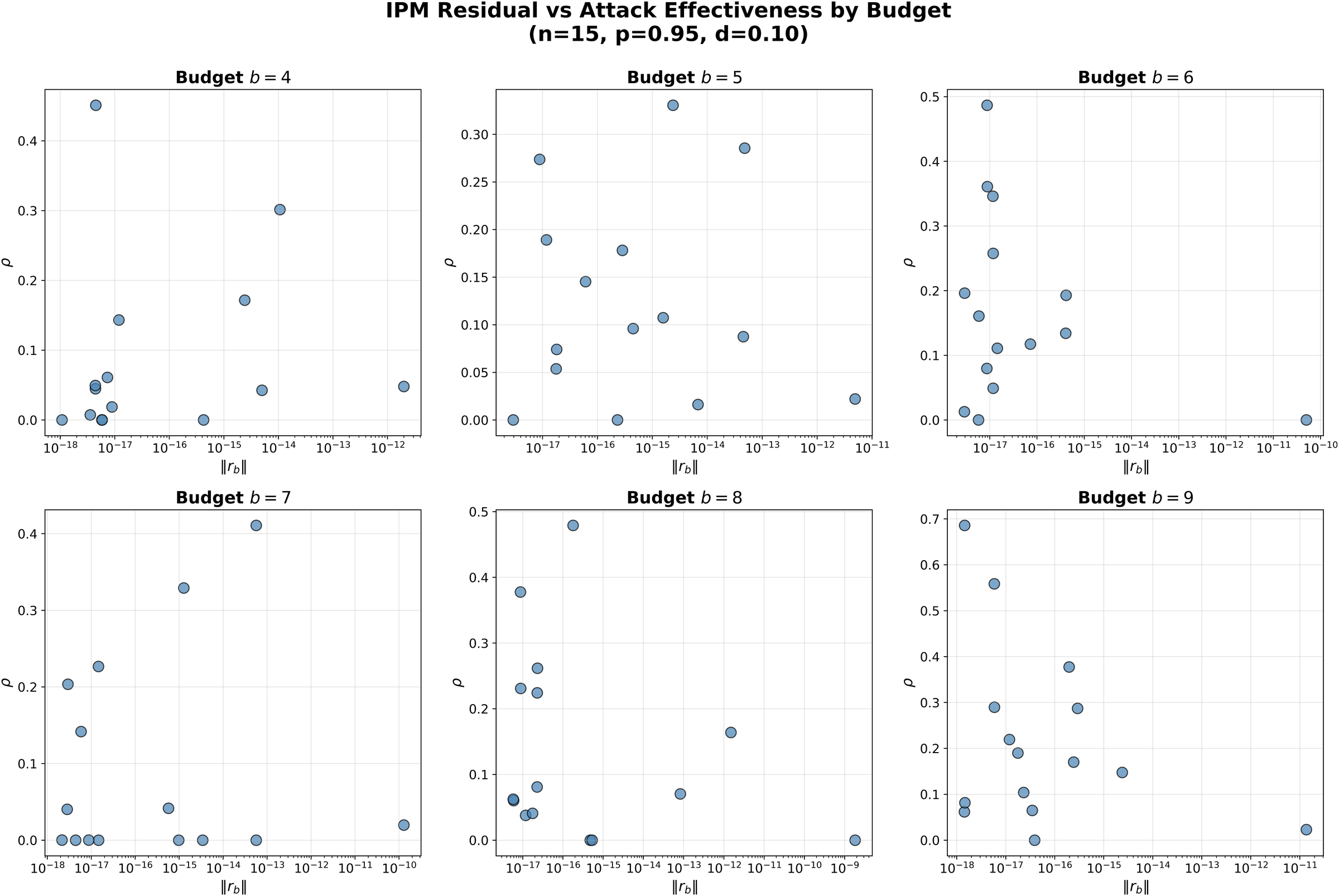}
        \caption{n=15, p=0.95, d=0.10}
    \end{subfigure}
       \caption{Correlation among $\|r_b\|$ and $\rho$.}
    \label{fig:network_attack_c_by_budget_2}
\end{figure}

\end{document}